\documentclass{amsart}
  \usepackage{amscd,amssymb,epsfig}
  \usepackage{epic,eepic}

                    \numberwithin{equation}{section}

                    \newtheorem{propo}{Proposition}[section]
                    
                    \newtheorem{theor}[propo]{Theorem}
                    \newtheorem{lemma}[propo]{Lemma}
                    \theoremstyle{definition}

                    \theoremstyle{remark}

            \newcommand{\CC}{\mathbb{C}}
            
                    \newcommand{\ZZ}{\mathbb{Z}}
                    \newcommand{\RR}{\mathbb{R}}

                    \newcommand{\Hom}{\operatorname{Hom}}
            \newcommand{\Ker}{\operatorname{Ker}}

                    \newcommand{\id}{\operatorname{id}}

\newcommand{\Int}{\operatorname{Int}}

\def\1{\hbox{\rm\rlap {1}\hskip .03in{\rm I}}}
\def\Int{{ \rm {Int}}}

\def\id{{  \rm {id}}}
\def\Mat{{  \rm {Mat}}}

\def\Edg{{  \rm {Edg}}}

\def\vrt {{  \rm {Vert}}}

\def\orlin {G}
\def\rup {B}
\def\ol {I}

  \def\Hom{{ \rm  {Hom}}}

\def\rel {{ \rm  {rel}}}

  \def\exp {{ \rm  {exp}}}

  \def\log {{ \rm  {log}}}

\def\Tr {{ \rm  {Tr}}}

  \def\Bas {{ \rm  {bas}}}

\def\Tr {{ \rm  {Tr}}}

\def\mod {{ \rm  {mod}}}

\def\Dim {{ \rm  {Dim}}}
  \def\dim  {{ \rm  {dim}}}

\def\mod {{ \rm  {mod}}\,}
  
  \def\mid {{ \rm  {Bas}}}

\def\id {{ \rm  {id}}}

\def\pr {{ \rm  {pr}}}

\def\Ker {{ \rm  {Ker}}}
  \def\Im {{ \rm  {Im}}}

\def\Irr {{ \rm  {Irr}}}

\begin{document}
      \title{Sections of fiber bundles over surfaces}
                    \author[Vladimir Turaev]{Vladimir Turaev}
                    \address{%
              Department of Mathematics \newline
\indent  Indiana University \newline
                    \indent Bloomington IN47405 \newline
                    \indent USA\newline
\indent vtouraev@indiana.edu}
                    \begin{abstract} We study the existence problem and
                      the enumeration problem for sections
                    of  Serre
fibrations over compact orientable surfaces. When the fundamental
group of the fiber is finite, a complete solution is given in terms
of 2-dimensional cohomology classes associated with certain
irreducible representations of this group. The proofs are based on
Topological Quantum Field Theory.
                    \end{abstract}
                    \maketitle



  \section{Introduction}\label{int}

      A
  {\it section}  of a
fibration  $p:E\to W$  is a continuous mapping $s:W\to E$ such that
$ps=\id_W$.  The existence problem for sections  is  fundamental in
the theory of   fibrations. If $p$ has sections, then one may ask
about the number of their homotopy classes. We address these
problems for Serre fibrations over compact orientable surfaces and
solve them under certain assumptions on the fiber.

Assume   that $p:E\to W$ is a Serre fibration over a compact
connected oriented surface $W$ of positive genus.
  If $\partial W\neq
  \emptyset$,  then we fix    a continuous section
$s_\partial:\partial W\to E$ of $p $
  over $\partial W$ (so that $ps_\partial=\id_{\partial W}$) and search for its
  extensions
    to $W$.
By definition,  a section  $s:W\to E$ of $p$ extends $s_\partial$ if
$s\vert_{\partial W}=s_\partial$. For closed $W$, this condition is
void. The reader wishing to avoid technicalities related to   $\partial W$ and $s_\partial$  may pursue the case $\partial W=\emptyset$.

One can approach the existence problem for sections via the
obstruction theory. To this end,  fix    base
  points  $w \in \Int \, W=W-\partial W$ and $e\in  p^{-1}(w)\subset E$.  Set
$\pi'=\pi_1(E,e) $,  $ \pi=\pi_1(W,w)$, and consider the
  homomorphism
  $p_\#:\pi'\to \pi$ induced by~$p$.  A section of
$p$ carrying $w$ to $e$ induces a
  right inverse $\pi\to \pi'$ of $p_\#$. A simple obstruction theory yields that
  any right inverse   of~$p_\# $ can be realized by a section $s$  of~$p$.
The condition $s\vert_{\partial W}=s_\partial$
  may  also be translated into the algebraic language.
However, this   reformulation in terms of the fundamental groups is
as difficult as the original problem and  does not  shed much light.

 We give a necessary and sufficient condition for the existence of
a section of $p$ extending~$s_\partial$ in the case where the
fundamental group of the fiber $F=p^{-1}(w)$ of $p$ is finite.   Our
condition involves 2-dimensional characteristic classes of~$p:E\to
W$ derived from certain linear representations of  $\pi_1(F)$. If,
additionally, $\pi_2(F)=0$, then we give a formula for the number of
homotopy classes of such sections.

We  now  state  our main theorems.  Let $C_1,\ldots, C_m$
  be the  components of $\partial W$
  endowed with orientation  induced by that of $W$. Here
    $m \geq 0$ is the number of   components of $\partial W$.
For
        $k=1, \ldots , m$, fix an embedded path $c_k:[0,1]\to W$ leading
        from   $w\in \Int\,  W$
          to  a point  of $  C_k $.
        We  assume that these $m$ paths  meet only at~$w$.
For each $k=1, \ldots , m$, fix a lift $\widetilde c_k:[0,1]\to E$
of $c_k$ leading from $e$
  to    $  s_\partial (c_k(1))$.
 Conjugating the loop $
s_\partial \vert_{C_k}$ by $ \widetilde c_k$, we obtain a loop in
$E$ based at~$e$.
  It represents a certain  $\gamma_k\in \pi'$.

Fix an algebraically closed field of  characteristic zero $K$. Set
$\Phi= \pi_1(F,e)$ and denote by  $\Irr (\Phi)$ the set of
equivalence classes of irreducible finite dimensional  linear
representation of $\Phi$ over $K$. Since $\pi_2(W)=0$,
  the inclusion $F\hookrightarrow E$ induces an injection  $\Phi\hookrightarrow  \pi'$
  so that we can view $\Phi$ as a subgroup of  $\pi'$.
Clearly, $\Phi=\Ker\, (p_\#:\pi'\to \pi) $ is   normal in  $\pi'$.
The action of   $\pi'$   on $\Phi$ by conjugations induces an action
of $ \pi' $ on $\Irr (\Phi)$. The restriction of the latter action
to $\Phi$ is trivial, and    we obtain  an action of $\pi=\pi'/\Phi$
on $\Irr (\Phi)$. Denote the fixed point set of
  this action by $I_0$. For any $\rho \in I_0$ and
any $ a \in \pi'$, there is a  matrix $ M_a \in GL_{\dim \rho} (K)$
such that $\rho ( a^{-1} \, h\,  a )= M_a^{-1}  \, \rho (h) \, M_a$
for all $h\in \Phi$.  Set $t_\rho(a)=0$ if $\Tr M_a=0$ and
$t_\rho(a)=1$ if $\Tr M_a \neq 0$, where $\Tr$ is the  matrix trace.
The number $t_\rho(a)$  is well-defined because  $M_a$ is  unique up
to multiplication by elements of $K^*=K\backslash \{0\}$. In Section
\ref{section2} we    associate  with $\rho \in I_0$ and the set
$\gamma=\{\gamma_k\}_{k=1}^m\subset \pi'$    a cohomology class
$\zeta^{\rho, \gamma} \in H^2(W,
\partial W;K^*)$.

\begin{theor}\label{t1}  Let the  fiber $F$ of $p:E\to W$ be path-connected and
the  group $\Phi=\pi_1(F,e)$ be finite.  The  section $s_\partial$
  of $p$ over $\partial W$ extends to a section of $p$ over~$W$
if and only if
\begin{equation}\label{1.2.a-}
\sum_{\rho\in I_0  } \left\{
( {\dim \, \rho})^{ \chi(W)} \, \zeta^{\rho, \gamma}  ([W, \partial W])\,
  \prod_{k=1}^m t_\rho (\gamma_k) \right\} \neq 0
.\end{equation}
\end{theor}

Here  $\chi(W)$ is the Euler characteristic of $W$ and
    $ \zeta^{\rho, \gamma} ([W, \partial W]) \in K^*$ is  the
    evaluation of
$ \zeta^{\rho, \gamma}  $  on the fundamental class $[W, \partial
W]\in H_2(W, \partial W;\ZZ) $. The evaluation   is
induced by the bilinear form $K^*\times \ZZ\to K^*, (k,n)\mapsto
k^n$.

    The sum in (\ref{1.2.a-})
is finite  because the sets   $\Irr(\Phi)$ and $I_0\subset
\Irr(\Phi)$ are finite. This sum  is non-empty because the trivial
1-dimensional representation of $\Phi$  belongs to $I_0$. The
individual terms in (\ref{1.2.a-}) may depend on the paths $
c_1,..., c_m$ but, as we shall see, their sum does not depend on
these paths and is a non-negative rational number. Theorem \ref{t1}
shows that $p$ has no sections extending $s_\partial$  if and only
if this number takes its minimal possible value (equal to zero).

Theorem \ref{t1}  follows from  a  Lefschetz-type formula for the
number of extensions of $s_\partial $ to $W$ considered up to an
appropriate equivalence relation. Here it is convenient to switch to
  pointed sections. Recall the base points $e\in E$
and $w=p(e)\in \Int\, W$. A section $s:W\to E$ of $ p$ is {\it
pointed} if $s(w)=e$. Two pointed sections of $p$ are {\it
homotopic} if they can be
  deformed into each other $\rel \, \partial W$ in the class of  pointed
sections  of $p$.  Two pointed sections of $p $ are related by {\it
bubbling} if they are equal  outside a small open 2-disk $D\subset
\Int \, W-\{w\}$. The restrictions of such two sections on the
closed 2-disk $\overline D\subset W$  form a mapping $S^2\to E$, a
\lq\lq bubble". Two pointed sections of $ p$ are
    {\it
bubble equivalent}    if they can be obtained from each other by a
finite sequence of  bubblings.  The set of bubble equivalence
classes of pointed sections of $p$ extending $s_\partial$ is
denoted~${\mathcal S}(p, s_\partial)$.    Decomposing a  deformation
of a section  into local deformations, one  observes that homotopic
pointed sections are bubble equivalent.  If  $\pi_2(F)=0$,    then
the converse is also true, and  the bubble equivalence is just the
homotopy.

\begin{theor}\label{t2}  Let the  fiber $F$ of $p:E\to W$ be path-connected and
the  group $\Phi=\pi_1(F,e)$ be finite. Then
\begin{equation}\label{t2-}\vert {\mathcal S}(p, s_\partial)\vert=\vert  \Phi \vert \,
\sum_{\rho\in I_0  } \left\{
(\vert  \Phi
\vert/\dim \, \rho)^{-\chi(W)} \, \zeta^{\rho, \gamma}  ([W, \partial W])\,
  \prod_{k=1}^m t_\rho (\gamma_k) \right\}\, ,
\end{equation}
where the vertical bars stand for the cardinality of a set.
\end{theor}

Theorem \ref{t2}  implies Theorem \ref{t1} (one should observe that
when  $F$  is path-connected, the bundle $p$ has a section if and
only if $p$ has a pointed section). Theorems \ref{t1} and \ref{t2}
are satisfactory from a topologist's viewpoint because they provide
a complete solution  to natural geometric problems concerning the
bundle~$p$  in terms of certain characteristic  classes associated
with  $p$.

  Theorems \ref{t1} and  \ref{t2}     are new and nontrivial already
  in the case $\partial W=\emptyset$. They were announced in this
  case
  in \cite{Tu4}.   Note that if  $\partial W=\emptyset$, then
  $\gamma=\emptyset$, $m=0$, and
  the product
  $\prod_{k=1}^m t_\rho (\gamma_k)=1$  may be deleted from Formulas
  (\ref{1.2.a-}) and (\ref{t2-}).

The connectedness assumption in   Theorems \ref{t1} and  \ref{t2} is
not a serious restriction. For a section of $p$ to exist, the
homomorphism $p_\#:\pi'\to \pi$ has to be surjective and so we must
have $\pi_0(F)=\pi_0(E)$. It would be very interesting to find a
version of  these theorems  for infinite $\Phi$.

The cohomology classes $\zeta^{\rho, \gamma}$ in (\ref{t2-}) may be
efficiently computed. This yields a number of interesting
applications  in topology and group theory. I mention  here  a few
applications. In the case where $\partial W=\emptyset$ and   $\Phi$
is finite abelian, the bundle $p:E\to W$ has a section if and only
if the induced homomorphism $p_\ast: H_2(E;\ZZ)\to H_2(W;\ZZ) $ is
surjective. The same equivalence holds for any finite $\Phi$
provided $\partial W=\emptyset$ and the genus of $W$ is greater than
or equal to $(1/2)\, \log_2 \vert \Phi\vert$. For a deduction of
these claims from Theorem \ref{t2}, see \cite{Tu4}. Similar claims
hold for $\partial
  W \neq \emptyset$, though the surjectivity of $p_\ast$ should be
  replaced  with the condition that the 1-cycle $s_\partial (\partial
  W)$ is homologically trivial in~$E$. Another application \cite{Tu4} concerns
 non-abelian cohomology of surfaces.

As a typical  group-theoretic  application of Theorem \ref{t2}, I
mention the following result of M.\ Natapov and myself \cite{NT}.
Let $a,b$ be two commuting elements of a group $G$. Suppose that $G$
  contains as a normal subgroup the quartenion group $Q=\{\pm 1,
\pm i, \pm j, \pm k\}$. Question: for how many pairs $\alpha,
\beta\in Q$ the elements $a\alpha, b\beta\in G$ commute? Answer: the
number of such pairs is equal to $ 8$,  $16$, $24$, or $40 $ and all
the numbers $ 8$,  $16$,
  $24$,  $40 $ are realized by some $G,a,b$. Similar results may be
  obtained
for any finite group  through a study of its representations.

   Theorem \ref{t1} is
uninteresting  for trivial bundles (product    fibrations)  because
they certainly have sections. Theorem \ref{t2}  for trivial bundles
is equivalent
  to the classical formula for the number of  homomorphisms
  $\pi=\pi_1(W,w)\to \Phi$  due to Frobenius    \cite{Fr}
  and   Mednykh
\cite{Me}, see   Section \ref{Tcoatb}.

Theorem \ref{t2} can be extended to non-orientable surfaces. In
another direction,  the sections of $p:E\to W$ can be counted with
weights determined by an element of $H^2(E, s_\partial (\partial
W);K^*)$. This will be discussed elsewhere.

The  proof of Theorem  \ref{t2}  is based on  techniques of quantum
topology and specifically utilizes 2-dimensional Homotopy Quantum
Field Theory (HQFT). HQFTs are versions of   familiar Topological
Quantum Field Theories (TQFTs) for manifolds and cobordisms endowed
with maps into a certain space. HQFTs were introduced in my
unpublished preprint \cite{Tu2}. Two-dimensional HQFTs with simply
connected target were independently introduced in \cite{BT}. The
bulk of the paper is devoted to a study of 2-dimensional HQFTs whose
target is an Eilenberg-MacLane space of type $K(G,1)$, where $G$ is
a group. This  study is strictly limited here to notions and
  results   needed for the proof of Theorem  \ref{t2}.
Our exposition is self-contained and  does not require preliminary knowledge of TQFTs or HQFTs.

Here is an outline of the proof of  Theorem \ref{t2}. We rewrite the
number $\vert {\mathcal S}(p, s_\partial)\vert$
  in terms of a 2-dimensional HQFT   associated with
$p_\#:\pi'\to \pi$. Analyzing the underlying  algebra  of this HQFT,
we split it   as a direct sum of HQFTs numerated by elements of
$\Irr (\Phi)$. Using  the technique of transfers, we show that only
the elements of $I_0\subset \Irr (\Phi)$ may contribute non-zero terms
to   $\vert {\mathcal S}(p,
s_\partial)\vert$. These terms are   the summands in (\ref{t2-}).



The  Frobenius-Mednykh formula       for the number of
homomorphisms      from $\pi_1(W) $       to a finite group was
rediscovered in the context of TQFTs by Dijkgraaf and
   Witten \cite{DW} and by  Freed and
     Quinn \cite{FQ}.   The present paper extends their work. It  is   plausible that Theorem \ref{t2} admits a purely
algebraic proof   generalizing the original methods of Frobenius and
Mednykh. The author plans to study an algebraic approach to
Theorem \ref{t2} elsewhere.



Throughout the paper, the symbol $K$ denotes an algebraically closed
field of characteristic zero.   The symbols $G$ and $G'$ denote
(discrete) groups.

My work on this paper was partially supported by the NSF grant
DMS-0707078.

\section {Representations and  cohomology classes}\label{section2}

We define  here (in a   general algebraic setting) the cohomology
classes $\zeta^{\rho,\gamma}$.

Throughout this section we fix a group epimorphism $q:G'\to G$
with kernel $\Gamma$.

\subsection{Representations of $\Gamma$.}\label{e21}
  By a representation
of $\Gamma$, we mean a group homomorphism $\rho:\Gamma \to GL_n (K)$
with $n\geq 1 $. Two  representations $\rho :\Gamma\to GL_{n} (K)$
and $\rho':\Gamma\to GL_{n'}(K)$ of $\Gamma$ are {\it equivalent}
  if $ n =n'$ and there is $M\in GL_{n }(K)$ such that  $\rho'(h) =M^{-1}\rho (h)  \,  M$
  for all $h\in \Gamma$. A  representation $\rho:\Gamma\to GL_n (K)$  is {\it irreducible}
if the induced action of $\Gamma$ on $K^n$ preserves no linear
subspace of $K^n$ except  $0$ and $K^n$. Let $\Irr (\Gamma)$ be the
set of equivalence classes of irreducible representations
of~$\Gamma$. Given    $a\in G'$ and an irreducible  representation
$\rho:\Gamma\to GL_n(K)$, the formula $h\mapsto  \rho ( a^{-1} h
  a ): \Gamma \to GL_n(K)$ defines an irreducible representation  of $\Gamma$
  denoted
$  a\rho $. This defines a  left  action of $G'$
  on~$\Irr (\Gamma)$. The  action of $\Gamma\subset G'$ is trivial  because
  $a\rho  =  \rho ( {a  })^{-1}\,  \rho  \, \rho
( {a })$ for any $a \in \Gamma$.
  We obtain thus a left action of    $G=G'/\Gamma$ on~$\Irr (\Gamma)$.
The   fixed point set of this action  is denoted $I_0(q)$.

\subsection{ The cohomology  class $\zeta^\rho$.}\label{e22}
For    $\rho\in I_0(q)$, we define  a cohomology class $\zeta^\rho
\in H^2(G ; K^*)$. Set $n=\dim\, \rho\geq 1$. For each $\alpha\in G
$, choose $\widetilde \alpha\in q^{-1} (\alpha)$  so that  $
\widetilde 1=1 $.  Since $\rho\in I_0(q)$,    there is a matrix
$M_\alpha\in GL_n(K)$ such that $  {\widetilde \alpha}  \rho  =
M_\alpha^{-1} \, \rho  \, M_\alpha$. By the Schur lemma,
  $M_\alpha$ is
  unique  up to multiplication by  elements of $K^*$. We call
  $M_\alpha$ the conjugating matrix (corresponding to $\widetilde \alpha$) and fix it
  for all $\alpha\neq 1$.
For $\alpha=1$,  set $M_\alpha=E_n$. For  $\alpha, \beta\in G $, set
$ L_{\alpha, \beta}= ({ \widetilde {\alpha\beta}})^{-1}\,
{\widetilde \alpha}\, {\widetilde \beta}\in  \Gamma  $.

\begin{lemma}\label{l1}  For any  $\alpha,
\beta\in G$, there is a unique  $\zeta_{\alpha, \beta} \in K^*$ such
that  \begin{equation}\label{eeq2.3}\zeta_{\alpha, \beta}
\,M_\alpha\,  M_\beta= M_{\alpha \beta} \,  \rho(L_{\alpha, \beta})
. \end{equation}  The family
$\{\zeta_{\alpha,\beta}\}_{\alpha,\beta}$ is a normalized 2-cocycle
on $G$. Its cohomology class  $\zeta^\rho\in H^2(G; K^*)$    depends
only
  on the equivalence
class of $\rho$ and does not depend on the choice of the matrices
$\{M_\alpha\}_{\alpha}$  and the lifts $\{\widetilde
\alpha\}_{\alpha }$.
\end{lemma}

\begin{proof}
  For  $\alpha,
\beta\in G $ and $h\in \Gamma$, $$M_{\alpha \beta}^{-1} \,\rho (h)\,
M_{\alpha \beta} =  (\widetilde {\alpha\beta}  \rho) (h)=
\rho({\widetilde {\alpha\beta}}^{-1}\, h \, {\widetilde
{\alpha\beta}})=  \rho (L_{\alpha, \beta}\, {\widetilde \beta
}^{-1}\, {\widetilde  \alpha }^{-1}\, h \, {\widetilde \alpha }\,
{\widetilde  \beta }\, L_{\alpha, \beta}^{-1} )$$
$$=  \rho (L_{\alpha, \beta})
\,M_\beta^{-1} \, M_\alpha^{-1}  \,\rho (h)  \,M_\alpha \, M_\beta
\,  \rho (L_{\alpha, \beta})^{-1}.$$ Since  $\rho$ is
irreducible, there is a unique $\zeta_{\alpha, \beta} \in K^*$
satisfying (\ref{eeq2.3}).

  Any element of $G' $
expands uniquely as $    \widetilde \alpha g$ with  $\alpha \in G $
and $g \in \Gamma$. The formula $\overline \rho  ({  \widetilde
\alpha g})=  M_\alpha\, \rho (g)  $ defines a mapping $\overline
\rho : G'\to  GL_n(K)  $. Clearly,  $\overline \rho \vert_{ \Gamma
}= \rho $ and  $\overline \rho  ({\widetilde \alpha})=M_\alpha$ for
all $\alpha\in G$.  Also, for all $a\in
 G' $ and  $h\in \Gamma$,
\begin{equation}\label{eeq2.5} \overline \rho (a )\,   \rho  (h)=
\overline \rho ( ah)\, .
\end{equation}  To see this,  expand
  $a={\widetilde \alpha  g}$  with
$\alpha\in G$ and $g \in \Gamma$. We have
$$  \overline \rho ({  \widetilde \alpha g}) \,   \rho
({h})=  M_\alpha\, \rho (g)\,
\rho(h)=
M_\alpha\, \rho (g h)  =  \overline
\rho ( { \widetilde \alpha  g  h}
)
  \, .$$
Formula (\ref{eeq2.5}) implies that for any $\alpha, \beta \in G$,
\begin{equation}\label{eeq2.7}  \overline\rho  ({\widetilde \alpha} \, {\widetilde \beta})=
    \overline\rho ({\widetilde {\alpha \beta}}) \,  \rho (L_{\alpha, \beta})
= M_{\alpha \beta}\,  \rho (L_{\alpha, \beta})
=\zeta_{\alpha,\beta} \, M_\alpha\, M_\beta=\zeta_{\alpha,\beta} \,
\overline\rho  ({\widetilde \alpha})\,  \overline\rho  ({\widetilde \beta})\, .
\end{equation}
We claim that more generally, for all $a,b \in G'$,
\begin{equation}\label{eeq2.8}\overline \rho({a }\, b)
=\zeta_{q(a), q(b)}\, \, \overline \rho(a) \, \overline
\rho(b). \end{equation}
To see this, expand  $a=  \widetilde \alpha g$, $b= \widetilde \beta
h$ with $\alpha=q(a)$, $ \beta =q(b)\in G$ and $g,h\in \Gamma$.
Using
  (\ref{eeq2.5}) and  (\ref{eeq2.7}), we obtain
$$\overline \rho(a \, b)=
  \overline
\rho({\widetilde \alpha} \,{\widetilde \beta} \,{\widetilde
\beta}^{-1}\, g\, {\widetilde\beta}\,  h ) =
  \overline \rho({\widetilde \alpha}\,
{\widetilde \beta})\,   \rho  ({\widetilde \beta}^{-1}\, g\,
{\widetilde\beta}) \,   \rho  ( h ) $$ $$=  \zeta_{\alpha,\beta}
\, M_\alpha \, M_\beta \, (\widetilde \beta \rho) (g)\, \rho ( h)
 =\zeta_{\alpha,\beta} \,    M_\alpha\, \rho ( g)  \, M_\beta\,
\rho (h) =\zeta_{\alpha,\beta} \,\overline \rho(a) \, \overline
\rho(b)\, .$$

Formula (\ref{eeq2.8}) implies that
$\{\zeta_{\alpha,\beta}\}_{\alpha,\beta}$ is a 2-cocycle on $G$. It
is  normalized in the sense that $\zeta_{1,1}=1$. That the
cohomology class $\zeta^\rho$ of this cocyle does not depend on the
choice of $\{M_\alpha\}_{\alpha}$ follows directly from
(\ref{eeq2.3}). To prove the independence of $\zeta^\rho$  of  the
lifts $\{\widetilde \alpha\}_{\alpha
    }$, suppose that each
$\widetilde \alpha$ is traded for  $\widetilde \alpha'= \widetilde
\alpha\, g_\alpha$ with $g_\alpha\in \Gamma$.  Then  $M_\alpha$ is
traded for
  $M'_\alpha= M_\alpha\, \rho(g_\alpha)$.  For    $\alpha,
\beta\in G $,
$$L'_{\alpha, \beta}=(\widetilde
{\alpha\beta}')^{-1}\, {\widetilde \alpha'}\, {{{\widetilde
\beta}'}}  =        {g_{\alpha\beta}^{-1}}\, {\widetilde
{\alpha\beta}}^{-1} \,    {\widetilde \alpha}\,  {g_\alpha}\,
{\widetilde \beta} \,    {g_\beta}
=  {g_{\alpha\beta}^{-1}}\, L_{
\alpha, \beta} \, ({\widetilde \beta}^{-1}\, {g_{\alpha }} \,
{\widetilde \beta}) \, {g_{ \beta}}  \, . $$
  Substituting these
new  values of $M,L$ in (\ref{eeq2.3}) and using that $ \rho
({\widetilde \beta}^{-1}\, {g_{\alpha }} \, {\widetilde \beta})=
M_\beta^{-1}  \rho (  g_{\alpha } ) M_\beta$, we obtain that the
lifts $\widetilde \alpha'$ lead  to the same cocycle $
\{\zeta_{\alpha,\beta}\} $. When
  $\rho $ is conjugated by  $M\in GL_n(K)$,  the
matrices $\{M_\alpha\}_\alpha$ are replaced by $\{M^{-1} M_\alpha
M\}_\alpha$ and  the cocycle $\{\zeta_{\alpha, \beta}\}_{\alpha,
\beta}$ is  also preserved.
\end{proof}

\subsection{The function $t_\rho$ and
the cohomology class    $ \zeta^{\rho,  \gamma}$.}\label{e23-9} A
representation
  $\rho\in I_0(q)$ determines a function $t_\rho: G'\to \{0,1\}$ as follows.  For any $a\in
G'$,  we have $a\rho=M^{-1}\rho M$ with    $M\in GL_n(K)$, where
$n=\dim\, \rho$. Since $M$ is unique up to multiplication by
elements of $K^*$, so is $\Tr\,M\in K$. We say that $a $ is {\it
$\rho$-adequate} if $\Tr \, M\neq 0$.  Let $t_\rho: G'\to \{0,1\}$
be the function
  sending all $\rho$-adequate elements of $G'$ to 1 and all other elements of $G'$ to $0$.
      The set of
$\rho$-adequate elements of $G'$ and the function $t_\rho$  depend
only on the equivalence class of~$\rho$.  A subset of $G'$ is  {\it
$\rho$-adequate} if all its elements are $\rho$-adequate.

We say that a set $\gamma\subset G'$ is    {\it $q$-free} if
$\gamma$ freely generates a (free) subgroup $\langle  \gamma
\rangle$ of $G'$ and the restriction of $q$ to $\langle  \gamma
\rangle$ is injective. We derive
  from a $q$-free set $\gamma$ and any
  $\rho\in I_0(q)$
  a  cohomology class    $ \zeta^{\rho,  \gamma} \in
H^2(G, H ; K^*)$, where $H= q(\langle \gamma  \rangle)=\langle
q(\gamma) \rangle \subset G$.  If $\gamma$ is not $\rho$-adequate,
set $\zeta^{\rho, \gamma}=0 $. Suppose that $ \gamma$ is
$\rho$-adequate.   Then for any   $a\in \gamma$, there is a unique
matrix $\nu(a) \in GL_n(K)$ such that $a\rho=\nu(a)^{-1}\rho \nu(a)$
and $\Tr \, \nu(a)=1$.  The mapping $\gamma\to GL_n(K)$, $a\mapsto
\nu(a)$ induces a group  homomorphism  $ \langle \gamma \rangle \to
GL_{n}(K)$ denoted $\nu$.   We define a normalized 2-cocycle
$\{\zeta_{\alpha,\beta}\}_{\alpha,\beta}$  on $G$  as in Section
\ref{e22} using the following choices: for all $\alpha\in H$, let
 $\widetilde \alpha $ be the only element of $q^{-1}(\alpha) \cap \langle \gamma \rangle$   and   $M_\alpha= \nu
 (\widetilde \alpha )$  (the lifts of elements of $G-H$
to $G'$ and the associated matrices $M$ are chosen in an arbitrary
way). It follows from the definitions  that $\zeta_{\alpha,\beta}=
1$ for all $\alpha, \beta\in H$. The proof of Lemma \ref{l1} shows
that the cocycle $\{\zeta_{\alpha,\beta}\}_{\alpha,\beta}$ is well
defined up to
  coboundaries of 1-cochains $G\to K^*$ sending~$H$
to~$1$.
  The cohomology class  $ \zeta^{\rho,  \gamma} \in
H^2(G, H ; K^*)$  of this cocycle  depends only on  the equivalence
class of~$\rho$. The natural homomorphism $H^2(G, H ; K^*) \to H^2(G
; K^*)$ carries  $\zeta^{\rho, \gamma} $ to $\zeta^\rho$ if~$\gamma$
is $\rho$-adequate and to $0$ otherwise. For example, the empty
subset of $G'$
  is $q$-free and $\zeta^{\rho, \emptyset}=\zeta^{\rho } \in
H^2(G, \{1\} ; K^*) = H^2(G  ; K^*)$.

\subsection{Theorems \ref{t1} and \ref{t2}  re-examined.}\label{e24}
We  can   now explain all notation  in  Theorems \ref{t1} and
\ref{t2}. Applying Section \ref{e21}  to \begin{equation}\label{idi}
G'=\pi',\,\,\, G=\pi,\,\,\, q=p_\#:\pi'\to \pi, \,\,\,\Gamma=\Phi\,
,
\end{equation} we obtain an action of $\pi$ on $\Irr (\Phi)$. Every
$\rho$ in the fixed point set of this action $I_0 =I_0(p_\#)$ yields
a function $t_\rho:\pi'\to \{0,1\}$ by Section \ref{e23-9}. The set
$\gamma=\{\gamma_k \}_{k=1}^m\subset \pi'$ defined before the
statement of Theorem \ref{t1} is  $q$-free because
   $p_\#(\gamma)$ generates a free subgroup    of $\pi$ of rank
   $m=\vert \gamma\vert$. Here it is essential that
the genus of $W$ is positive. Therefore each $\rho\in I_0 $
determines $\zeta^{\rho, \gamma} \in H^2(\pi, \langle
q(\gamma)\rangle ;K^*) =H^2(W,
\partial W;K^*)$. This completes the statement of Theorems \ref{t1} and
\ref{t2}.

\section {Homotopy Quantum Field Theory}\label{section3}

Throughout this section we fix an integer $d\geq 0$ and a
 connected CW-space $X$ with base point $x$.

\subsection  {Preliminaries.}\label{ee3.1} By a
  manifold, we mean    an oriented  topological manifold.
  A manifold  $M$ is
  {\it pointed}  if   every    component  of $M$    is provided with
a base point.  The set of base points of $M$ is denoted $M_\bullet$.

A {\it $d$-dimensional $X$-manifold} is a pair (a pointed closed
$d$-dimensional  manifold $M$, a  map $g :M\to X$  such that
$g(M_\bullet)=x$).    A disjoint union of $X$-manifolds is an
$X$-manifold in the obvious way. An {\it $X$-homeomorphism  of
$X$-manifolds} $ (M,g)\to (M',g')$ is an orientation preserving
homeomorphism $f:M\to M'$ such that $g =g'f$ and
$g(M_\bullet)=M'_\bullet$. An empty set  is viewed as a pointed
manifold and an $X$-manifold of any given dimension.

By a  {\it $(d+1)$-dimensional cobordism}, we  mean  a triple
$(W,M_0,M_1)$ where $W$ is a compact $(d+1)$-dimensional  manifold
and  $M_0,M_1$ are disjoint pointed
 closed $d$-dimensional  submanifolds of $\partial W$ such that $\partial W=(-M_0)\amalg
 M_1$. The manifold $W$   is not supposed to be
  pointed. As usual, $-M$ is $M$ with reversed orientation.


An    {\it $X$-cobordism}  is a cobordism $(W,M_0,M_1)$ endowed with
a map  $g:W\to X$  carrying   the base points of  $M_0, M_1$ to
  $x$. Both
  bases  $M_0$ and  $M_1$ of $W$ are
considered as
  $X$-manifolds with    maps  to $X$ obtained by restricting
$g$. An {\it $X$-homeomorphism of $X$-cobordisms}
\begin{equation}\label{ee3.1.1} F:(W,M_0,M_1,g)\to
(W',M'_0,M'_1,g')\end{equation}  is an orientation preserving and
base point preserving homeomorphism of triples $(W,M_0,M_1)\to
(W',M'_0,M'_1)$ such that $g =g'F $. The standard operations on
cobordisms (disjoint union, gluing along bases, etc.) apply in  this
setting in the obvious way. For brevity, we shall often  omit  the
maps from the notation for $X$-manifolds and $X$-cobordisms.

\subsection  {Axioms of HQFTs}\label{ee3.2}  We  adapt  Atiyah's
axioms of a TQFT  to the present setting.  A {\it
$(d+1)$-dimensional  Homotopy Quantum Field Theory
  with target $X$} or, shorter, a {\it   $(d+1)$-dimensional  $X$-HQFT} assigns a finite-dimensional vector space
$A_M$ over $K$ to any $d$-dimensional $X$-manifold $M$, an
 isomorphism $f_{\#}:A_M\to A_{M'}$ to any $X$-homeomorphism of
$d$-dimensional $X$-manifolds $f:M\to M'$, and a   homomorphism
$\tau(W): A_{M_0}\to A_{M_1} $ to any $(d+1)$-dimensional
$X$-cobordism    $(W,M_0,M_1)$. The following seven axioms should be
met.

(1) For any $X$-homeomorphisms of $d$-dimensional $X$-manifolds
$f:M\to M', f':M'\to M''$, we have $(f'f)_{\#}=f'_{\#}f_{\#}$.

 (2) For any $X$-homeomorphism  (\ref{ee3.1.1}), the following
 diagram is commutative:
 $$
 \CD
 A_{(M_0,g\vert_{M_0})}  @>(F\vert_{M_0})_{\#} >> A_{(M'_0,g'\vert_{M'_0} )}  \\
   @V\tau (W,\, g) VV    @VV\tau (W',\, g')V        \\
   A_{(M_1,g\vert_{M_1})}  @>(F\vert_{M_1})_{\#} >> A_{(M'_1,g'\vert_{M'_1} )}\, .
 \endCD  $$

      (3) If an $X$-cobordism    $(W,M_0,M_1)$  is obtained from two
      $(d+1)$-dimensional $X$-cobordisms $(W_0,M_0,N)$ and $(W_1,N',M_1)$
      by gluing along an $X$-homeomorphism $f:N\to N'$, then
      $$\tau(W)= \tau(W_1)\circ f_{\#} \circ \tau(W_0):A_{M_0}\to A_{M_1}.$$

(4) $  A_{\emptyset}=K$ and for any disjoint $d$-dimensional
$X$-manifolds $M,N$, there is a natural  isomorphism $A_{M\amalg
N}=A_M\otimes A_N$ (for a detailed formulation of    naturality, see
\cite{Tu1}, p.\ 121). Here and below $\otimes=\otimes_K  $.

(5) If a $(d+1)$-dimensional $X$-cobordism $W$    is a disjoint
union of  $X$-cobordisms $W_1,W_2$, then $\tau(W)=\tau(W_1) \otimes
\tau(W_2)$.

(6) For any  $d$-dimensional $X$-manifold $(M,g:M\to X )$,
  $$\tau(M\times [0,1], M\times \{0\}, M\times \{1\},
g\circ \pr :M\times [0,1] \to  X)=\id: {A_M}\to {A_M}\, ,$$
  where $\pr:M\times [0,1] \to M$ is the
  projection.  Here   we identify $A_{M\times \{t\}}
=A_M$  for $t=0,1$ via the canonical homeomorphism  $ {M\times
\{t\}}\approx M $.

(7) For any  $(d+1)$-dimensional $X$-cobordism $(W, g:W\to X) $, the
homomorphism $\tau(W)$ is preserved under  homotopies of  $g$
constant on $\partial W$.

   If  $X=\{x\}$,  then all references to maps to
  $X$     are redundant, and $X$-HQFTs are the familiar TQFTs.  Note a few
   properties of HQFTs generalizing the standard properties of TQFTs,  see \cite{Tu1}. Given a
    $(d+1)$-dimensional $X$-HQFT $(A, \tau)$, we associate with
  any  $d$-dimensional
$X$-manifold $M $    the bilinear form
$$\eta_M=\tau(M\times [0,1],M\times \{0\} \cup (-M)\times
\{1\},\emptyset) :  A_M\otimes A_{-M}\to K\, .$$ This form is   nondegenerate,
natural with respect to $X$-homeomorphisms, multiplicative with
respect to disjoint union,
  and
symmetric in the sense that $\eta_{-M}=\eta_M\circ   {\sigma}$ where
$\sigma$ is the standard flip $ A_{-M}\otimes A_{M} \to A_M\otimes
A_{-M}$. By definition, $-\emptyset=\emptyset$ and
$\eta_\emptyset:K\otimes K\to K$ is the multiplication in $K$.

For a compact $(d+1)$-dimensional manifold $W$ with pointed boundary
and a map  $(W, (\partial W)_\bullet) \to (X,x)$, the HQFT $(A,
\tau)$ yields the   vector space $A_{\partial W}$ and two
homomorphisms $\tau(W,\emptyset,\partial W) :K=A_\emptyset\to
A_{\partial W}$ and   $\tau(-W,\partial W, \emptyset) :A_{\partial
W}\to A_\emptyset=K$. Let $\tau(W)\in A_{\partial W}$ be the  image
of the unity $1_K$ under the first   homomorphism. The second
homomorphism   is nothing but the image of
 $\tau(-W)  $ under the isomorphism
$ A_{-\partial W} \approx \Hom_K(A_{\partial W}, K) $ induced by
$\eta_{\partial W}$. The vector $\tau(W)$ is invariant under
homotopy of the given map $W\to X$,   natural with respect to
$X$-homeomorphisms, and multiplicative under disjoint union.

 Two $(d+1)$-dimensional $X$-HQFTs $(A,\tau)$ and $
(A',\tau')$   are isomorphic if there is
    a family of $K$-isomorphisms
$\{\rho_M:A_M\to A'_M\}_M$, where $M$ runs over all  $d$-dimensional
$X$-manifolds, such that:  $\rho_{\emptyset}=\id_K $  and
$\rho_{M\amalg  N} =\rho_M\otimes \rho_N$ for any
  disjoint  $X$-manifolds $M,N$;    the natural square
diagrams associated with $X$-homeomorphisms    and  $X$-cobordisms
are commutative.


\subsection  {Example}\label{ee3.3}
Each $(d+1)$-dimensional $K^*$-valued singular cocycle $\theta$ on
$X$ gives rise to a
  $(d+1)$-dimensional $X$-HQFT $(A,\tau)$. Let
$M=(M,g:M\to X)$ be a $d$-dimensional $X$-manifold.
  Then  $A_M$ is
a one-dimensional $K$-vector space defined as follows. A
$d$-dimensional singular cycle $a\in  C_d(M
 ;\ZZ)$ is a {\it fundamental cycle of~$M$} if  it
represents the sum of the fundamental homology classes of the
components of $M$. Every such  $a $ determines a generating vector
$\langle a\rangle\in A_M=K \langle a\rangle$. If $a,b$ are two
fundamental cycles of~$M$, then $a-b=\partial c$ for
  $c\in C_{d+1}(M; \ZZ)$. We impose the equality $
\langle b\rangle=g^*(\theta) (c) \langle a\rangle$, where
$g^*(\theta)$ is the singular
  cocycle  on $M$ obtained by pushing
back $\theta$ along~$g $.  It is easy to check that $g^*(\theta)
(c)$ does not depend on the choice of $c$ and that $A_M$ is a
well-defined one-dimensional  vector space. If $M=\emptyset$, then
$a=0$ and
 by definition $A_M=K$ and $\langle a\rangle=1_K\in K$. An $X$-homeomorphism
of   $X$-manifolds $f:M\to M'$ induces an isomorphism $f_{\#}:
A_M\to A_{M'}$  by $f_{\#}(\langle a\rangle)= \langle f_*(a) \rangle
$ for any fundamental cycle $a
 $ of $M$.

Given a $(d+1)$-dimensional  $X$-cobordism    $(W,M_0,M_1, g:W\to
X)$, pick a singular chain $B\in C_{d+1}(W ;\ZZ) $ such that
$\partial B=b_1-b_0$, where $b_0,b_1$ are fundamental cycles of
$M_0, M_1$, respectively. We also require   $B$ to be fundamental in
the sense that its   image    in $C_{d+1}(W, \partial W;\ZZ)$ is a
fundamental cycle of $W$. We define $\tau(W): A_{M_0}\to A_{M_1} $
by
 $\tau(W) (\langle b_0\rangle)= g^*(\theta) (B)  \, \langle
b_1\rangle$. It is an exercise in singular homology to show  that
$\tau(W)$ is well defined and     the axioms of an
  HQFT are met, cf.\  \cite{Tu2}. This HQFT is
denoted $(A^{\theta}, \tau^{\theta})$. Its  isomorphism class
  depends only on the cohomology class of $\theta$.

For  $\partial W=\emptyset$, we have $\tau(W,g)=g^*(\theta)
([W])=\theta(g_*([W]))$. We can  view   $(A^{\theta},
\tau^{\theta})$ as a device  extending to cobordisms the
  evaluation of
$\theta$ on   cycles in $X$. Another such device (in the smooth
category) is provided by the  Cheeger-Simons differential
characters. For a comparison, see \cite{Turn}.

\subsection  {Operations on HQFTs}\label{ee3.3+}
We  need three operations on   HQFTs:    direct sum, rescaling, and
transfer.  The direct sum of  $(d+1)$-dimensional $X$-HQFTs
$(A^1,\tau^1) $, $ (A^2,\tau^2)$    is defined as follows. Set
$(A^1\oplus A^2)_M=A^1_M\oplus A^2_M$ for any connected
$d$-dimensional $X$-manifold  $M$ and  extend this to nonconnected
$M$ via Axiom (4) above. The action of $X$-homeomorphisms is defined
by applying $\oplus$ and $\otimes$ to the actions provided by
$(A^1,\tau^1)$, $ (A^2,\tau^2)$. For a  $d$-dimensional $X$-manifold
$M$ and $k=1,2$, we have a natural embedding $i^k_M:A^k_M\to
(A^1\oplus A^2)_M$  and a natural projection $p^k_M:(A^1\oplus
A^2)_M \to A^k_M, $  such that $p^k_M i^k_M=\id$. For a  connected
$(d+1)$-dimensional $X$-cobordism $(W,M,N)$, set $$(\tau^1\oplus
\tau^2)(W)=i^1_N\tau^1(W)p^1_M + i^2_N \tau^2(W)p^2_M: (A^1\oplus
A^2)_M \to (A^1\oplus A^2)_N\, . $$ This extends to nonconnected
$X$-cobordisms  via Axiom (5) and gives a $(d+1)$-dimensional
$X$-HQFT $(A^1\oplus A^2, \tau^1\oplus \tau^2)$.

One way to rescale a $(d+1)$-dimensional $X$-HQFT consists in
multiplying  the homomorphism   associated with each $X$-cobordism $
(W,M,N)$   by $k^{\chi (W) -\chi (M)}  $ with   fixed $k\in K^*$.
For $d=1$, we shall use a slightly subtler version of this
transformation. Given
  $k\in K^*$, the {\it $k$-rescaling} transforms
  a 2-dimensional    $X$-HQFT    into  the same HQFT  except that
  the homomorphism
associated with each  $X$-cobordism $ (W,M,N)$ is multiplied by
$k^{(\chi (W) + b_0(M)- b_0(N))/2} $. Here $b_0(M)=\dim H_0(M;\RR)$
and we use the  inclusion $\chi (W) + b_0(M)- b_0(N)\in 2\ZZ$.

The transfer  of HQFTs  is defined in the following setting. Let
  $p:\widetilde X\to X$ be a      finite-sheeted (unramified) covering. Consider the path-connected pointed
space $Y=\widetilde X/p^{-1}(x)$ and let $q:\widetilde X\to Y$ be
the projection.  The transfer derives from a $(d+1)$-dimensional
HQFT $(A,\tau)$   with target $Y$
   a $(d+1)$-dimensional  HQFT $(\widetilde A, \widetilde
\tau)$ with target $X$  as follows. For  a $d$-dimensional
$X$-manifold $(M,g:M\to X)$, consider all lifts of $g$ to
$\widetilde X$, i.e., all maps $\widetilde  g:M\to \widetilde X$
such that $p \widetilde  g=g$. The set of such $\widetilde  g$ is
finite (possibly, empty). Each pair $(M, q\widetilde  g)$ is an
$Y$-manifold. Set
$$\widetilde  A_M= \oplus_{\widetilde  g, \,p \widetilde  g=g} \,\,A_{(M, q\widetilde
g)}.$$ Given an  $X$-homeomorphism of  $X$-manifolds
$f:(M,g)\to (M',g')$, a  lift $\widetilde  g:M\to \widetilde X$ of
$g$ induces a lift $ \widetilde  g f^{-1}:M'\to \widetilde X$ of
$g'$. The HQFT $(A, \tau)$ provides an isomorphism $f_{\#}: A_{(M,
q\widetilde g)}\to A_{(M' , q\widetilde g f^{-1})}$. The direct sum
of the latter over all $\widetilde  g$ is the  isomorphism $ f_{\#}:
\widetilde A_{(M, g)}\to \widetilde A_{(M', g')}$ determined by
$(\widetilde A, \widetilde \tau)$.

Given a $(d+1)$-dimensional    $X$-cobordism    $(W,M_0,M_1, g:W\to
X)$, there is a finite set  of maps $\widetilde g:W\to \widetilde X$
such that $p \widetilde  g=g$. Each such  $\widetilde  g$
restricted to $M_0,M_1$ yields certain lifts, $ \widetilde g_0,
\widetilde  g_1$, of the maps $ g_0= g \vert_{M_0}:{M_0}\to X$ and $
g_1= g \vert_{M_1}:M_1\to X$, respectively. We define $\widetilde
\tau(W,g): \widetilde A_{M_0}\to \widetilde A_{M_1} $ to be the sum
over all $\widetilde g$ of the  homomorphisms $\tau(W,q\widetilde
g): A_{(M_0, q\widetilde  g_0)}\to A_{(M_1, q\widetilde g_1)} $.
Then $(\widetilde A, \widetilde \tau)$ is    a $(d+1)$-dimensional
$X$-HQFT.

For example, for $d\geq 1$,    any $\theta\in H^{d+1}(\widetilde
X;K^*)= H^{d+1}(Y;K^*) $ determines  a  $Y$-HQFT
  $(A^\theta, \tau^\theta)$. Its transfer  $(\widetilde
A^\theta, \widetilde \tau^\theta)$  is an $X$-HQFT. For a
$d$-dimensional $X$-manifold $(M, g:M\to X)$, the dimension of $
\widetilde  A^{ \theta}_M$ is  equal to the number of lifts of~$g$
to $\widetilde X$. For  a  closed  $(d+1)$-dimensional $X$-manifold
$(W, g:W\to X)$, we have $\widetilde  \tau^{  \theta}(W,g)=
\sum_{\widetilde g } \widetilde  g^*(\theta) ([W])$, where
$\widetilde g$ runs over all lifts of $g$ to $\widetilde X$.

\subsection  {Aspherical targets}\label{ee3.4+}
 A $(d+1)$-dimensional HQFT
$(A,\tau)$ with aspherical target $X$ may be reformulated in terms
of homotopy classes of maps to $X$ rather than maps themselves. By
homotopy  we mean homotopy in the class of maps sending the base
points   to $x\in X$.  Here is the key observation: for any pointed
closed $d$-dimensional  manifold $M$ and   maps $g_0, g_1:M\to X$
related by a homotopy $H:(M\times [0,1], M_\bullet\times [0,1]) \to
(X,x)$, the pair $(M\times [0,1], H)$ is an $X$-cobordism, and the
associated homomorphism $\tau (M\times [0,1], H):A_{M, g_0}\to
A_{M,g_1}$  does not depend on the choice of $H$.  This    follows
from Axiom (7) of an HQFT since any two such $H$ are homotopic,
which is a consequence of the asphericity of~$X$. Composing $H$ with
the inverse homotopy and using Axiom (6), we obtain  that  $\tau
(M\times [0,1], H)$ is an isomorphism. We identify the vector spaces
$A_{M, g}$ (with $g:M\to X$ in the given homotopy class)   along
these isomorphisms. The resulting vector space depends only on the
homotopy class  and has the same dimension as $A_{M, g}$. The action
of $X$-homeomorphisms and of $X$-cobordisms is compatible with these
identifications.   As a consequence, dealing with  HQFTs with
aspherical target $X$, we can safely replace maps of manifolds and
cobordisms to $X$ by  pointed  homotopy classes of maps.

From now on, we will consider only $X$-HQFTs with aspherical  $X$
and will modify the notions of $X$-manifolds and $X$-cobordisms
accordingly (i.e., view the maps to $X$ up to pointed homotopy). By
$X$-homeomorphisms
  we  will mean  orientation preserving and base point
 preserving homeomorphisms   commuting with the given homotopy classes of maps to $X$.

 \begin{lemma}\label{l1908}  If $f:M\to N$ is   an $X$-homeomorphism   of $d$-dimensional
   $X$-manifolds and if a homeomorphism $f': (M,M_\bullet)\to
(N, N_\bullet)$ is   isotopic to $f$ in the class of homeomorphisms
$(M,M_\bullet)\to (N, N_\bullet)$, then  $f'$ is an
$X$-homeomorphism and $f'_\#=f_\#:A_M\to A_N$ for any
$(d+1)$-dimensional $X$-HQFT $(A,\tau)$.
\end{lemma}

The first claim of this lemma follows from the definitions. The
second claim follows from Axioms (2) and (6) of an HQFT.

An aspherical space $X$  with base point $x$  is an
Eilenberg-MacLane space of type $K(G,1)$ for   $G=\pi_1(X,x)$.   By
Section \ref{ee3.3}, every $\theta\in H^{d+1}(G;K^*)=H^{d+1}(X;K^*)$
gives rise to a $(d+1)$-dimensional $X$-HQFT  $(A^{\theta},
\tau^{\theta})$.  More generally, consider a subgroup $H\subset
\orlin$ of finite index   and the associated covering $p:\widetilde
X\to X$. Pick
  $\widetilde x\in p^{-1}(x)$ and identify $\pi_1(\widetilde X,\widetilde x)=H$
via $p_\#$.   Any $\theta\in H^{d+1}(H;K^*)=H^{d+1}(\widetilde
X;K^*)$ with $d\geq 1$ yields a    $(d+1)$-dimensional $X$-HQFT
through
  transfer.  For $d=1$ and    $k\in K^*$, we can   $k$-rescale this HQFT. The resulting $X$-HQFT
is denoted    $(A^{G, H,
  \theta} , \tau^{G, H, \theta, k} ) $. It can also be obtained by
  first $k$-rescaling the  HQFT $(A^{\theta},
\tau^{\theta})$ and then transferring to $X$.


\section  {Biangular $\orlin$-algebras and 2-dimensional HQFTs}

For  a group $G$, we introduce   a class  of $G$-graded algebras
  giving rise to   2-dimensional HQFTs with target $X=K(G,1)$.
 This     generalizes the   construction of
  2-dimensional TQFTs ($\orlin=1$) from associative algebras in    \cite{BP}  and  \cite{FHK}.

\subsection  {Biangular $\orlin$-algebras}\label{BGa}
A {\it $\orlin$-graded algebra}  or, briefly, a {\it
$\orlin$-algebra}  is an associative algebra $B$ over $K$ endowed
with a splitting $B=\bigoplus_{\alpha\in \orlin} B_{\alpha}$  such
that all $\rup_\alpha$ are finite-dimensional, $B_{\alpha} B_{\beta}
\subset B_{\alpha \beta}$ for any $\alpha, \beta\in \orlin$, and $B$
has a (right and left) unit $1_B\in B_1$ where $1$ is the neutral
element of $\orlin$. An {\it inner product}  on a $\orlin$-algebra
$B$ is  a symmetric $K$-bilinear  form $\eta : B \otimes B \to K $
such that

(1)  for all $\alpha \in \orlin$, the restriction of $\eta$ on
$B_\alpha\otimes B_{\alpha^{-1}}$ is nondegenerate;

(2) $\eta(B_{\alpha}\otimes B_{\beta})=0$ if $\alpha \beta\neq 1$,

(3) $\eta(a,b)=\eta(ab, 1_B)$ for any $a,b\in B$.

For $c\in B$, the symbol $\mu_c$ denotes  the homomorphism $B\to B,
x\mapsto cx$. A $\orlin$-algebra $\rup=\oplus_{\alpha\in \orlin}
\rup_\alpha$ is {\it biangular}  if it  has an inner product $\eta :
B \otimes B \to K $ such that for any $\alpha, \beta\in G$ and $a\in
\rup_\alpha, b\in \rup_{\alpha^{-1}}$,
\begin{equation}\label{innn} \eta  (a,b) =
  \Tr(\mu_{ab}\vert_{\rup_\beta}:\rup_\beta\to
\rup_\beta)\, .
\end{equation}
The form $\eta $ (if it exists) is unique and  denoted  $\eta_B$.
The conditions above imply that for any $c\in B_1$, the trace of
$\mu_c:B_\beta\to B_\beta$ does not depend on $\beta$ and is equal
to $\eta_B(c,1_B)=\eta_B(1_B,c)$. In particular, $\dim
(\rup_\beta)=\dim
  (\rup_1)=\eta_B(1_{ \rup}, 1_{ \rup})$ for all $\beta\in G$. Note
that Conditions (2) and (3) for $\eta_B$ and the symmetry of
$\eta_B$ may be deduced from (\ref{innn}).

Given a biangular $G$-algebra $B$ and $\alpha\in G$, we can view the
vector spaces $B_{\alpha}$ and $ B_{\alpha^{-1}}$ as dual to each
other via $\eta_B$.
 Pick a
basis $\{p_i^\alpha \}_{i }$ of $B_\alpha$, consider the dual basis
$\{q_i^\alpha \}_{i }$ of $B_{\alpha^{-1}}$, and set
$b_\alpha=\sum_i p_i^\alpha \otimes q_i^\alpha  \in B_\alpha\otimes
B_{\alpha^{-1}}$.  The vector $b_\alpha$  does not depend on the
choice of  $\{p_i^\alpha \}_{i }$. The system of vectors
$\{b_\alpha\}_\alpha$ is symmetric in the sense that
  $b_{\alpha^{-1}}$ is obtained from $b_\alpha$ by permutation of
the tensor factors for all $\alpha\in G$. Note that
for all $\alpha\in G$,
\begin{equation}\label{firstide}
\sum_{i } p^\alpha_i q^\alpha_i=1_B\, .\end{equation} This follows
from the non-degeneracy of $\eta_B$ and the fact that for any $c\in
\rup_1$,
$$  \eta(c , \sum_i
p_i^{\alpha} q_i^{\alpha} ) =\sum_i\eta( c \, p_i^{\alpha},
q_i^{\alpha} )  = \Tr(\mu_c\vert_{\rup_\alpha} :{\rup_\alpha} \to
{\rup_\alpha})=\eta_B(c,1_B)\, .$$ Also, for any $a\in B_{\alpha^{-1}}$ and $b\in B_\alpha$,
\begin{equation}\label{secide}
\sum_i \eta( a,p_i^{\alpha}) \, \eta (b,q_i^{\alpha} )
 =
\eta( a, \sum_i \eta (b,q_i^{\alpha} ) \, p_i^{\alpha}) = \eta(a,b)
\, . \end{equation}

  For example, if $q:G'\to G$ is a group epimorphism with finite
  kernel $\Gamma$, then the group algebra $B=K[G']$ is a
biangular  $G$-algebra, where $B_\alpha=K[q^{-1}(\alpha)]\subset B$
for all
  $\alpha\in G$. Observe that
  $\eta_B(a,b)=\vert \Gamma \vert$ if $a,b\in G'$ satisfy $ab=1$  and $\eta_B(a,b)=0$ for all  other $a,b\in
  G'$.
  For $\alpha\in G$, we have $b_\alpha=\vert \Gamma \vert^{-1}\sum_{a\in q^{-1}(\alpha)}
  a\otimes a^{-1}$. More
  general examples of biangular  $G$-algebras may be derived from the twisted group algebras of
  $G'$ associated with $K^*$-valued 2-cocycles on $G'$.

\subsection  {$G$-systems}\label{Aps1}
Recall  the standard combinatorial description of    maps from a
CW-complex $T$ to
  $X =K(G,1)$.  By  {\it vertices}\index{CW-complex!vertex of},
{\it edges}\index{CW-complex!edge  of}, and  {\it
faces}\index{CW-complex!face  of} of $T$, we  mean 0-cells, 1-cells,
and    2-cells of $T$, respectively. Denote the set of vertices of
$T$ by $\vrt (T)$ and    the set of  oriented  edges of $T$ by
$\Edg(T)$. Each  $e\in \Edg(T)$  leads from an {\it initial vertex}
$i_e $  to a {\it terminal vertex} $t_e $ (they may coincide).
  The orientation reversal defines  a free involution
$e\mapsto e^{-1}$ on $\Edg(T)$.

A face $\Delta$ of $T$ is   adjoined to the 1-skeleton $T^{(1)}$ of
$T$ along a (continuous) map $f_\Delta:S^1\to T^{(1)}$. We call $T$
{\it regular} if for any    face $\Delta$ of $T$, the set
$f_{\Delta}^{-1} (\vrt (T))\subset S^1$ is  a finite non-empty set
which splits $S^1$ into arcs mapped by $f_{\Delta}$ homeomorphically
onto
  edges of~$T$. These arcs in $S^1$  are called the {\it
sides}  of $\Delta$.  The image $f_\Delta(e)$ of a side $e$  of
$\Delta$ is an edge of $T$ called the {\it underlying edge}  of $e$.
We shall often make no difference between sides and their
underlying edges. An orientation of $\Delta$ induces an orientation
and a cyclic o $e_1, e_2, \ldots ,e_n$ on the set of sides
of~$\Delta$. Here $n\geq 1$ is the number of sides and
$t_{e_{r}}=i_{e_{r+1}}$  for  all $r \,(\mod \,n)$. The
corresponding cyclically ordered oriented edges of $T$ form the
boundary of $\Delta$.

Let $T$ be a regular  CW-complex and   $V$ be a subset of $\vrt
(T)$.
  A {\it
$\orlin$-system}\index{$G$-system}
  on  $(T,V)$ is a mapping  $\Edg(T)\to \orlin,
  e\mapsto g_e$  such
that

(i)  $g_{e^{-1}}=(g_e)^{-1}$  for any $e\in \Edg(T)$;

(ii) if  ordered oriented edges $e_1,e_2,\ldots ,e_n$ of $T$ with
$n\geq 1$ form the boundary of a face  of $T$,  then $g_{e_1}
g_{e_2} \cdot\cdot\cdot g_{e_n}=1$.


Two $\orlin$-systems $g,g'$ on $(T,V)$ are {\it homotopic} if there
is a map $v: \vrt (T)\to \orlin$ such that $v (V)=1$ and $ g'_e=v
(i_e) \,g_e\, (v (t_e))^{-1}$ for all $e\in \Edg (T)$. Homotopy  is
an equivalence relation on the set of $\orlin$-systems on $(T,V)$.

A $\orlin$-system $g $ on $(T,V)$ gives rise to a map $\vert g
\vert: \vert T \vert  \to X$, where $\vert T \vert$ is the
underlying topological space of $T$. This map    sends $\vrt (T)
 $   to the base
point  $x\in X $ and sends each  $e\in \Edg(T)$ into a loop in $X$
representing $g_e\in \orlin=\pi_1(X,x)$. The  map $g \mapsto \vert g
\vert$ induces a bijection between the homotopy classes of
$\orlin$-systems on $(T, V)$ and the homotopy classes  of maps
$(\vert T \vert , V) \to (X,x)$.

  \subsection  {State sums on closed surfaces.}\label{Ssocs}
  Fix a biangular
$\orlin$-algebra $ \rup=\oplus_{\alpha \in G} \rup_\alpha
  $. Let $W$ be
a closed  $X$-surface, i.e., a closed  oriented surface endowed with
a  map  $W\to X=K(\orlin,1)$.  We define    $ \tau_{ \rup}(W) \in K$
as follows.

Pick a regular CW-decomposition $T$ of $W$ (for example, a
triangulation). By a {\it flag}  of $T$, we mean a pair (a face
$\Delta$ of $T$, a side $e$ of $\Delta$). The flag
  $(\Delta,e)$  induces an orientation on $e$ such that
$\Delta$ lies on the right of $e$. This means that the pair (a
vector looking from a point of $e$ into $\Delta$, the oriented edge
of~$T$ underlying~$e$) is positive  with respect to the given
orientation of $W$.

Let  $g $ be a  $\orlin$-system on $T$ representing the  homotopy
class of the given map  $\vert T \vert=W\to X$ (here $V=\emptyset$).
With each flag $(\Delta,e) $ of  $T$  we associate the $K$-module ${
\rup}(\Delta,e,g)={ \rup}_{g_e}$ where  $e$ is oriented so that
$\Delta$ lies on its right.

Every  edge $e$ of $T$ appears in two flags, $(\Delta_1,e) ,
(\Delta_2,e) $, and inherits from them opposite orientations. Since
the corresponding values of $g$  are mutually inverse,  Section
\ref{BGa} yields a
  vector $b_{g_e}\in
{ \rup}(\Delta_1,e,g)  \otimes { \rup}(\Delta_2,e,g)$.  Set
$$b_g=\otimes_e \, b_{g_e} \,\in \, \otimes_{(\Delta,e) } \, {
\rup}(\Delta,e,g) $$ where  on the left-hand side  $e$ runs over all
unoriented edges of $T$ and on the right-hand side  $(\Delta,e) $
runs over all flags of $T$.

Let $\Delta$ be  a  face of $T$ with $n $ sides. We orient and
cyclically order  the sides $e_1, e_2, \ldots ,e_n$ of $\Delta$ so
that $\Delta$ lies on their right. Then $g_{e_1}g_{e_2} \cdots
g_{e_n}=1$. The  form \begin{equation}\label{form} {
\rup}(\Delta,e_1,g) \otimes { \rup}(\Delta,e_2,g)\otimes
\cdot\cdot\cdot \otimes { \rup}(\Delta,e_n,g) \to K\, ,
\end{equation} defined by
$$ a_1\otimes a_2 \otimes \ldots \otimes a_n \mapsto \eta_B(a_1 a_2
\cdot\cdot\cdot a_n,1_{ \rup})$$  is invariant under cyclic
permutations. The tensor product of these
  forms over all  faces of $T$ is a homomorphism
$d_g: \otimes_{(\Delta,e) } { \rup}(\Delta,e,g) \to K$.
  Set  $\langle g
\rangle=d_g(b_g)\in K$.

\begin{lemma}\label{lele-1} $\tau_B(W)=\langle g \rangle \in K  $
   depends only on the
homotopy class of the  given map $W\to X$.
  \end{lemma}

\begin{proof} We need to prove that $\langle g \rangle$ does not depend on the choice of
  $g$ in its homotopy class and does not depend on the choice of
  $T$. It is convenient to switch to
the dual language of skeletons of $W$. A {\it skeleton} of $W$ is a
finite graph on $W$ whose complementary regions are   open  2-disks.
A skeleton may have loop edges (i.e., edges with coinciding
endpoints) and multiple edges (i.e., different edges with the same
endpoints). A regular CW-decomposition  $T$ of $W$ determines a
skeleton $S_T$ of $W$. It is obtained by placing a  \lq\lq central"
point  in each face of $T$ and connecting the centers of any two
faces adjacent to the same 1-cell $e$ of $T$ by an edge meeting $e$
transversely in one point and disjoint from the rest of $T^{(1)}$.
This establishes an equivalence between   regular CW-decompositions
and   skeletons. The  definitions of a $G$-system $g$ on $T$ and of
$\langle g\rangle$ can be easily rewritten in terms of labelings of
the oriented edges of skeletons.

We define two local moves on  a  skeleton $S$ of $W$.    The {\it
contraction move}  contracts an edge  of $S$  with distinct
endpoints.  The corresponding move on CW-decompositions erases an
open edge adjacent to two distinct faces.  The  {\it  loop move}  on
${S} $ removes a   loop edge   bounding
  a   disk in $W- {S}$.   The corresponding
move on CW-decompositions erases a vertex of valency 1 and the
incident open edge.   The contraction moves, the  loop moves, and
their inverses are called {\it basic moves}. Using the theory of
spines of surfaces, one shows that any two skeletons of $W$ can be
related by a finite sequence of basic moves  and ambient isotopy in
$W$. The invariance of $\langle g \rangle$ under the loop move and
the contraction move follows from (\ref{firstide}) and
(\ref{secide}), respectively.

It remains to prove the invariance of  $\langle g \rangle$ under a
homotopy move $g\mapsto g'$.   In terms of   the dual  skeleton $S$,
the move is determined by a   map $v:\pi_0(W-S)\to G$. We have
$g'(f)= v (U) g(f) (v(V))^{-1}$ where $f$ is any oriented edge  of
$S$ and $U,V$ are the components of $W- {S}$ lying on the right and
on the left of $f$, respectively (possibly, $U=V$). It is enough to
consider the case where $v$ takes value $1\in \orlin$ on all
components of $W- {S}$ except a single  component $U$, where this
value is $v_0\in G$. We say   that $g'$ is obtained from $g$ by a
$v_0$-move at $U$.  Pick a small embedded loop $e\subset U$ based at
a vertex of $S$ adjacent to $U$. The loop $e$ splits $U$ into a
small  disk bounded by $e$ and a complementary open 2-disk $D$. Then
${S}_1={S} \cup e$ is a skeleton of $W$ obtained from ${S}$ by an
inverse loop move. The $G$-systems
 $g,g'$ on $S$ extend
 to   $\orlin$-systems $g_1,g'_1$ on  ${S}_1$
 related by the $v_0$-move  at the component $ D$
of $W- {S}_1$ (the element  $g_1(e)\in G$ may be chosen arbitrarily;
this choice and $v_0$ determine  $g'_1(e)\in G$).
 Now we modify ${S}_1$
keeping one   endpoint of $e$  and sliding the second endpoint along
the sides of~$U$.  We do this until the second endpoint traverses
the entire boundary of $U$ and comes back to its original position
from the other side. During this deformation $e$ becomes an embedded
arc except at the beginning and   the end when $e$ is an embedded
loop.
 At the end of the deformation, the  skeleton ${S}_1={S} \cup e $
is transformed into a  skeleton ${S}_2$
 isotopic to ${S}_1$.
 Note that when
 the moving endpoint  of $e$  traverses  a vertex of ${S}$ adjacent
to $U$, the  skeleton ${S} \cup e$  is transformed via a composition
of a contraction move with an inverse contraction move. Under these
two moves, the $\orlin$-systems $g_1,g'_1$  are transformed in a
canonical way  keeping their values on all edges except the
contracted ones. Throughout the deformation, these two
$\orlin$-systems
    remain  related by the
$v_0$-move at the image of $  D$  under the deformation. At the end
of the deformation the $\orlin$-systems $g_1,g'_1$  are transformed
into $\orlin$-systems, $g_2,g'_2$ on ${S}_2$  related by the
$v_0$-move at the image of $  D$ under the deformation. This image
is a small 2-disk bounded by $e$. Applying the loop move, we
transform $g_2,g'_2$ into one and the same $\orlin$-system $g_3$ on
${S}$. This gives  six sequences of basic moves:
 $g \mapsto g_1\mapsto g_2\mapsto g_3 \mapsto g'_2\mapsto g'_1\mapsto
 g'$.
  By
the invariance of $\langle g\rangle$  under the basic moves,
$\langle g \rangle=\langle g' \rangle$.
\end{proof}

  \subsection  {A
pseudo-HQFT}\label{statesu} The invariant $ \tau_B$ of closed
$X$-surfaces derived above from  a biangular $\orlin$-algebra $B$ is
extended here to a
  2-dimensional \lq\lq pseudo-HQFT" $(A=A_B, \tau=\tau_B)$ with target $X=K(G,1)$.
The prefix \lq\lq pseudo" appears here for two reasons. First of
all, the pseudo-HQFT $(A, \tau)$ applies to 1-dimensional
$X$-manifolds with a certain additional structure called
trivialization. Secondly, $(A, \tau)$ fails  to satisfy Axiom (6) of
Section \ref{ee3.2}. This pseudo-HQFT   will be transformed into a
genuine HQFT in the next subsection.

By $X$-curves, we shall mean   1-dimensional $X$-manifolds. A {\it
trivialized $X$-curve} is a triple $(M,T,g)$, where $M$ is   a
pointed closed    1-dimensional manifold, $T$ is a CW-decomposition
of $M$ such that
  $M_\bullet\subset \vrt (T)$  and $g$ is a
$\orlin$-system  on~$T$.  This $\orlin$-system   gives rise to a
homotopy class of maps $\vert g \vert: (M, M_{\bullet}  )\to (X,x)$,
where $x$ is the base point of $X$. This makes $M$ into an $X$-curve
and allows us to consider trivializations as additional structures
on $X$-curves.

We associate with a  trivialized    $X$-curve $M=(M,T,g)$ the
finite-dimensional vector space $A_{M} =\otimes_e \, {{
\rup}_{g_e}}$, where $e$ runs over the  edges of $T$    endowed with
the orientation  induced by that of $M$.
 By
definition, $M=\emptyset$ is trivialized and $A_{\emptyset}=K$. Note
that  $A_{M\amalg N}=A_{M} \otimes A_{N}$ for all $M,N$.

 An $X$-homeomorphism of trivialized  $X$-curves   is a homeomorphism of the underlying manifolds
preserving the base points, the orientation, the trivialization, and
the $G$-system. Such homeomorphisms act on the associated vector
spaces in the obvious tautological way.

  Consider  a 2-dimensional $X$-cobordism
$(W,M_0,M_1 )$ such that both  $X$-curves  $M_0$ and $M_1$ are
  trivialized.  We define a $K$-homomorphism
$\tau(W):A_{M_0} \to A_{M_1}$ as follows. Pick a regular
CW-decomposition $T$ of $W$ extending the given CW-decompositions of
$ M_0 $ and $M_1$.  Consider  a $\orlin$-system $g$ on $ T $ such
that  $g$ extends
  the  $\orlin$-systems on $M_0$ and $  M_1$
provided by the trivialization and  the induced map $\vert g \vert:
(W, (\partial W)_{\bullet} )\to (X,x)$ is in the given homotopy
class. Flags $(\Delta,e) $ of  $T$  and the vector  spaces ${
\rup}(\Delta,e,g)={ \rup}_{g_e}$ are defined as in Section
\ref{Ssocs}. As in Section \ref{Ssocs}, we have a homomorphism $ d_{
g}: \otimes_{(\Delta,e) } { \rup}(\Delta,e, g) \to K$, where
$(\Delta,e) $ runs over all flags of $T$, and a  vector
$$b_g=\otimes_e \, b_{g_e}\in \otimes_{(\Delta,e\subset \Int W) } {
\rup}(\Delta,e,g)
$$ where   $e$ runs  over all edges
of $T$ lying in $\Int \, W
 $ and    $(\Delta,e) $ runs over all flags of  $T$
such that $e\subset \Int \, W$.   Contracting $d_g$ with $  b_g$,
we obtain a homomorphism \begin{equation}\label{ggg} \langle g
\rangle: \otimes_{ e \subset M_0\cup M_1 } { \rup}(\Delta_e,e, g)
\to K\, .\end{equation}

Since $\partial W= M_1\cup (-M_0)$, the  face   of $T$ adjacent to
$e \subset M_r$ lies on the left of $e$  for $r=1$ and  on the right
of $e$  for $r=0$. Therefore
$${ \rup}(\Delta_e,e,  g)=
\begin{cases} { \rup}_{g_e}~ { \rm  {if}}\,\,\, e\subset M_0, \\
  { \rup}_{g_e^{-1}}=({ \rup}_{g_e})^* ~    {\rm  {if}}\,\,\, e\subset
M_1.\end{cases}
$$
Here we identify ${ \rup}_{g_e^{-1}}$ with the dual of $ {
\rup}_{g_e}$ using the inner product $\eta_B$ on $B$. Thus, the
homomorphism (\ref{ggg}) is adjoint to  a homomorphism
$$ A_{M_0}=\bigotimes_{e \subset M_0}
{ \rup}_{g_e} \to \bigotimes_{e \subset M_1}  { \rup}_{
g_e}=A_{M_1}\,$$ denoted $\tau(W)$.

\begin{lemma}\label{lele1} The homomorphism
$\tau(W)$ does not depend on the choice of  $T$ and $g$. (It depends
only on the trivializations of the $X$-curves $M_0, M_1$ and the
homotopy class of the given map $(W, (\partial W)_\bullet)\to
(X,x)$).\end{lemma}

 The proof goes along the same lines as the proof of
 Lemma \ref{lele-1}. The only difference is that a skeleton  of $W$
  meets $\partial W$ along a set of 1-valent
 vertices that are the centers of the 1-cells
  of $M_0\cup M_1$

\begin{lemma}\label{lele2}  Let $M_0,M_1,N$ be
trivialized   $X$-curves. If a 2-dimensional $X$-cobordism
$(W,M_0,M_1)$ is obtained  by gluing
  two $X$-cobordisms $(W_0,M_0, N)$ and
$(W_1,N, M_1)$ along $N$, then  $\tau(W)= \tau(W_1)\circ
\tau(W_0):A_{M_0} \to A_{M_1}$. \end{lemma}

\begin{proof}  For $r=0,1$, pick  a regular
CW-decomposition $T_r$ of $W_r$ extending the given
CW-decompositions of the bases.  Let  $g_r$ be a $\orlin$-system on
$ T_r
 $ representing the given map $W_r\to X$.  Gluing $T_0$ and $T_1$ along $N$
we obtain a regular CW-decomposition, $T$, of $W$.   Let $g$ be the
unique $\orlin$-system on $ T $ extending  $g_0$ and~$g_1$.
Formula (\ref{secide}) implies that $\langle g \rangle =\langle g_1
\rangle \circ \langle g_0 \rangle: A_{M_0} \to A_{M_1} $.
\end{proof}

To sum up, the rule $M\mapsto  A_M$, $W\mapsto \tau(W)$ defines a
sort of HQFT for trivialized $X$-curves and for $X$-cobordisms with
trivialized bases. This \lq\lq pseudo-HQFT" $(A, \tau)$ satisfies
all axioms of Section \ref{ee3.2} except possibly Axiom (6).

\subsection  {A 2-dimensional HQFT}\label{statesu++}  We now
  derive from the pseudo-HQFT  $(A=A_B, \tau=\tau_B)$   a
genuine   HQFT, cf.\ [Tu2], Section VII.3. Observe that for any
trivializations $t_0, t_1$ of an $X$-curve $M$, the cylinder
$W=M\times [0,1]$ (mapped to $X$ via the composition of the
projection $W\to M $ with the given map $M\to X$) is an
$X$-cobordism between the trivialized $X$-curves $(M, t_0)$ and $(M,
t_1)$. Set
$$p(t_0,t_1)=\tau(W):A_{(M, t_0)}\to A_{(M, t_1)}\, .$$
By Lemma \ref{lele2}, $p(t_0,t_2)=p(t_1,t_2)\, p(t_0,t_1)$ for any
trivializations $t_0,t_1, t_2$ of~$M$.  In particular,  $p(t_0,t_0)$
is a projector onto a subspace $A^{\circ}_{(M, t_0)}$, of $A_{(M,
t_0)}$. Moreover,
  $p(t_0,t_1)$ maps
$A^{\circ}_{(M, t_0)}$ isomorphically onto $A^{\circ}_{(M, t_1)}$.
We identify the vector spaces $\{A^{\circ}_{(M, t)}\}_{t}$ (where
$t$ runs over all trivializations of $M$) along these isomorphisms
and   obtain a vector space  $A^{\circ}_M$ independent of
  $t$. By definition,  for any trivialization  $t$ of   $M$, we have a canonical embedding
  $A^\circ_M\subset A^{}_{(M,
t)}$  and a canonical projection $P_M:A^{}_{(M, t)} \to A^\circ_M$
which is   the identity on $A^\circ_M$.

 To define the action
$f_{\#}:A^{\circ}_M \to A^{\circ}_{M'}$ of an $X$-homeomorphism
$f:M\to M'$,
  pick a trivialization $t$ of $M$
  and consider the   trivialization $t'= f(t)$ of $M'$. Then
  $f_{\#}$
  is  the composition of the identification
isomorphism  $A^{\circ}_M\cong  A^{\circ}_{(M,t)}$ with the
isomorphism $A^{\circ}_{(M,t)}\cong  A^{\circ}_{(M',t')}$ induced by
$f$ and with the identification isomorphism $ A^{\circ}_{(M',t')}
  \cong A^{\circ}_{M'}$.  This composition is independent of the
  choice of $t$.

A 2-dimensional $X$-cobordism $(W,M_0,M_1)$    splits as a union of
collars of $M_0, M_1$  with a copy of $W$. Therefore for any
  trivializations $t_0, t_1$ of $M_0, M_1$, we have
$\tau(W)=P_{M_1}\tau(W) P_{M_0}:A_{(M_0, t_0)}\to A_{(M_1, t_1)}$.
Let $\tau^{\circ} (W):A^{\circ}_{M_0} \to A^{\circ}_{M_1}$  be the
restriction of $\tau(W)$ to   $A^{\circ}_{M_0}\subset A_{(M_0,
t_0)}$. This homomorphism   does not depend on $t_0, t_1$ and enters
the
  commutative diagram
\begin{equation}\label{CDDD}\CD A^{}_{(M_0, t_0)} @>P_{M_0}>>   A^{\circ}_{M_0} @>{\rm inclusion}>> A^{}_{(M_0, t_0)}\\
  @V\tau(W)VV    @VV\tau^\circ(W)V @VV\tau(W)V\\
  A^{}_{(M_1, t_1)} @>P_{M_1}>>   A^{\circ}_{M_1} @>{\rm inclusion}>> A^{}_{(M_1, t_1)} \, . \endCD  \end{equation}

\begin{theor}\label{t24} The   rule $M\mapsto  A^{\circ}_M$,
   $W\mapsto \tau^{\circ}(W)$ defines
 an  $X$-HQFT.  \end{theor}

 This theorem follows from the results  of Section
\ref{statesu}. To stress the role of   $B$, we shall   write
$A^\circ_B$ for $A^\circ$ and $\tau^\circ_B$ for $\tau^\circ$. If
$W$ is a closed $X$-surface, then $\tau^\circ_B (W) =\tau_B(W)$.

The vector space $A^\circ_M$ associated with a connected $X$-curve
$M$ may be described in terms of the $G$-algebra $B$ as follows. The
$X$-curve $M$ yields a loop in $(X, x)$   representing an element
$\alpha=\alpha(M)$ of $G=\pi_1(X,x)$. The $X$-curve $M$ admits a
canonical trivialization $t $ having only one vertex (located at the
base point of~$M$)  and   one edge. Then
 $A_{(M,t)}=B_{\alpha}$ and $A^\circ_M$ is the image of
the projector $P_M:B_{\alpha}\to   B_{\alpha} $.  We shall compute
$P_M$ in   Section \ref{thegcenterb}. Note that each $\gamma\in
B_{\alpha}$ yields a vector $[\gamma]=P_M(\gamma)$ in  $ A^\circ_M
$.

\section {Reduction  of Theorem    \ref{t2} to a lemma}\label{sectione23}

 Consider  a  group epimorphism  $q:G'\to G$ with finite kernel $\Gamma$.
The   biangular $G$-algebra $B=K[G']$ derived from $q$ in
     Section \ref{BGa} determines a  2-dimensional HQFT
     $(A^{\circ}=A^\circ_B, \tau^\circ=\tau^\circ_B)$ with target
     $X=K(G,1)$.
  We   compute $(A^{\circ} , \tau^\circ )$  in two different
  ways
  and deduce Theorem    \ref{t2}.

\subsection  {Lemmas}\label{twolemmas}
 Let $W $ be a compact connected oriented surface
with $m \geq 0$ pointed boundary components $C_1,..., C_m$ endowed
with orientation induced by that of~$W$. For $k=1, \ldots , m$, fix
an embedded  path $c_k:[0,1]\to W$ leading from a base point $w\in
\Int\,  W$
  to  the base point  of $  C_k $.
We  assume that these $m$ paths  meet only at~$w$, and set $c=\cup_k
\, c_k ([0,1])\subset W$. Transporting $ C_k $  along $ c_k $, we
obtain  a loop in $W$ based at $w$ and representing some $x_k\in
\pi=\pi_1(W,w)$.

Let $X=K(G,1)$ with base point $x$. Fix a  homomorphism $g:\pi\to G$
and   consider    a map $\widetilde g:(W, c)\to (X,x)$    such that
$\widetilde g_\#=g:  \pi\to G=\pi_1(X,x)$.
 For
  $k=1, \ldots , m$,  the
  $X$-curve
   $(C_k, \widetilde g\vert_{C_k})$ represents $g(x_k)\in G$.

Fix a set $\gamma=\{\gamma_1,..., \gamma_m\}\subset G'$ such that
$\gamma_k\in  q^{-1}(g(x_k))$ for all $k$. By Section \ref{statesu++}, the vector $\gamma_k\in   B_{g(x_k)}$
 projects to a vector
  $[\gamma_k]\in A^\circ_{C_k}$.
 Set
 $$[\gamma]= \otimes_{k=1}^m [\gamma_k] \in A^\circ_{\partial
 W}= \otimes_{k=1}^m A^\circ_{C_k}  \, .$$
 If $m=0$, then $\gamma=\emptyset$ and by definition
 $[\gamma]=1_K\in K=A^\circ_{\emptyset}=A^\circ_{\partial
 W}$.

\begin{lemma}\label{firstkeyl}
The homomorphism  $\tau^\circ(-W,\widetilde g ):A^\circ_{\partial
W}\to K$ associated with the $X$-cobordism $(-W, \partial W,
\emptyset, \widetilde g)$ is computed  on $[\gamma]$   by
$$\tau^\circ(-W,\widetilde g) ([\gamma])=\vert \Gamma\vert^{\chi(W)-1}\, \times \, \vert {\mathcal S}_* (g,
\gamma)\vert  \, ,$$
where  ${\mathcal S}_*(g, \gamma)$ is  the (finite)  set of all
pairs (a homomorphism $g':\pi\to G'$, a family $\{a_k\in
\Gamma\}_{k=1}^m$)
 such that  $qg'=g$ and $
g'(x_k) = a_k \gamma_k a_k^{-1}$ for all $k$).
\end{lemma}

 \begin{proof}  Take a CW-decomposition
 $T$ of $W$ having   the base points of $C_1,..., C_m$ and $w$ as the only
 vertices,
 having $C_1,..., C_m$ and
  $c_1,..., c_m$
 among edges, and having only one face. Then $\chi(W)=m+2-r $, where $r $ is the number
 of edges   of $T$.
 The map $\widetilde g$ is represented by a $G$-system on $T$   assigning $1$ to $c_k$ and  $g(x_k)$ to
 $C_k$ for all $k$. The computation of   $\eta_B$ and
  $\{b_\alpha\}_{\alpha\in G}$ at the end of Section \ref{BGa} implies that  $\tau^\circ(-W,\widetilde g) ([\gamma])= \vert
 \Gamma\vert \times \vert
 \Gamma\vert^{-(r -m)} \times$ (the number of $G'$-systems on $T$
 which are lifts of our  $G$-system and which assign $\gamma_k$
 to $C_k$ for all $k$). The latter number is equal to $\vert {\mathcal S}_* (g ,
\gamma)\vert$.  \end{proof}

 \begin{lemma}\label{firstkey1++} Suppose that the genus of $W$ is positive and let  $H=\langle x_1,..., x_m\rangle$ be the  subgroup of $ \pi$  (freely)  generated by $x_1,..., x_m$. 
 If $g:\pi\to G$ is an epimorphism whose restriction to $H$  is injective, then the set 
  $\gamma= \{\gamma_1,..., \gamma_m\}\subset G'$ is
 $q$-free in the sense of Section \ref{section2} and 
\begin{equation}\label{t3-++-34}\tau^\circ(-W, \widetilde g) ([\gamma])
  =
\sum_{\rho\in I_0 (q) } ( \dim \, \rho)^{\chi(W)} \, g^*(\zeta^{\rho, \gamma })  ([W, \partial W])\,
  \prod_{k=1}^m  t_\rho (\gamma_k)  \,  .\end{equation}
\end{lemma}

The first claim of the lemma  
follows from the equality  $\langle q(\gamma)\rangle=g(H)$. Note that   $\zeta^{\rho, \gamma }\in H^2(G, g(H);K^*)$ and  $g^*(\zeta^{\rho, \gamma }) \in H^2(\pi,
H;K^*)$. We have
  $$H_2(\pi,H ;\ZZ)=  H_2(W, c\cup \partial W ;\ZZ)=  H_2(W, \partial W ;\ZZ)=\ZZ\cdot [W, \partial
  W]$$  so that we can evaluate $g^*(\zeta^{\rho, \gamma }) $ on $[W, \partial
  W]$. For the definition of $I_0(q)$, see Section \ref{e21}.

   Lemma \ref{firstkey1++}  will be proven in Section \ref{Tcoatb---+}. Combining Lemmas  \ref{firstkeyl}
   and \ref{firstkey1++}, we obtain
that for any $q$-free system $\gamma=\{\gamma_k\in
q^{-1}(g(x_k))\}_{k=1}^m$,
\begin{equation}\label{t3-++-}\vert {\mathcal S}_* (g, \gamma)\vert
  =  \vert \Gamma\vert
\sum_{\rho\in I_0 (q) } (\vert \Gamma\vert/\dim \, \rho)^{-\chi(W)} \, g^*(\zeta^{\rho, \gamma })  ([W, \partial W])\,
  \prod_{k=1}^m  t_\rho (\gamma_k)  \,  .\end{equation}

 \subsection{Proof of Theorem  \ref{t2}}\label{e24zzx}
We substitute (\ref{idi}) and $g=\id:\pi\to \pi $  in
(\ref{t3-++-}).
  Since the genus of $W$ is positive,
  $x_1,..., x_m \in  \pi $ generate a free group  of rank $m$. Therefore the
set
  $\gamma\subset \pi'$ constructed before the statement of  Theorem  \ref{t1} is     $q$-free.
Theorem  \ref{t2} now  follows from  (\ref{t3-++-}) provided we show
that
  $\vert {\mathcal S} (p, s_\partial)\vert= \vert {\mathcal S}_* (\id, \gamma)\vert$.

A pointed section $s $ of $p:E\to W$ extending $s_\partial:\partial
W\to E$ induces a homomorphism $s_\#:\pi \to \pi'$ such that
$qs_\#=\id$.  For $k=1,..., m$,  the path $ \widetilde c_k:[0,1]\to
E$ used in the definition of $\gamma_k $ in Section \ref{int} and
  the path  $s  c_k$ have the same endpoints and both project to $c_k
$. Let $a_k\in \Gamma=\Phi$ be represented by the  loop $  (s c_k)
\widetilde c_k^{-1}:[0,1]\to E$.
  Clearly,  $ (s_\#, \{a_k
\}_{k=1}^m) \in {\mathcal S}_* (\id, \gamma)$. The resulting mapping
$ {\mathcal S} (p, s_\partial) \to {\mathcal S}_* (\id, \gamma)$ is
bijective, as can be easily shown using the same CW-decomposition of
$W$  as in the proof of Lemma \ref{firstkeyl}. The  injectivity
results from the identification of   bubble equivalent sections  in
${\mathcal S} (p, s_\partial)$. The surjectivity is obtained by
 constructing an appropriate section over the 1-skeleton of $W$
and extending it  over the face using the injectivity of the
inclusion homomorphism $\pi_1(F,e)\to \pi'$, where $F=p^{-1}(w)$.
Thus, $\vert {\mathcal S} (p, s_\partial)\vert= \vert {\mathcal S}_*
(\id, \gamma)\vert$.

\section{Crossed $G$-algebras}\label{crocro}

We introduce    crossed  $G$-algebras which will be our main tools
in the study of
  2-dimensional HQFTs with target $K(G,1)$.

 \subsection{Basics}\label{crocroa}   We   use
 terminology of  Section \ref{BGa}.
   A {\it crossed $G$-algebra} is a triple (a  $\orlin$-algebra  $L=\oplus_{\alpha\in \orlin}\,
   L_\alpha$, an inner product $\eta$ on $L$, a homomorphism
   $\alpha\mapsto \varphi_\alpha $ from  $G$ to the group of algebra
   automorphisms of $L$) such that   for all $\alpha, \beta \in
   \orlin$,

(1) $\varphi_{\alpha} (L_{\beta})\subset
   L_{\alpha\beta\alpha^{-1}}$ and $\varphi_{\alpha}\vert_{L_{\alpha}}=\id$;

   (2)  for any $a\in L_\alpha, b\in L $, we have
           $\varphi_{\alpha}(b)a=ab$;

     (3)   $\eta   (\varphi_{\alpha} (a),
     \varphi_{\alpha}(b))=\eta(a,b)$ for all $a,b\in L$;

   (4)  for any
   $c\in   L_{\alpha\beta \alpha^{-1}
   \beta^{-1}}$,  the homomorphism $\mu_c:L\to L, x\mapsto
   cx$ satisfies
\begin{equation}\label{traces} \Tr\, (\mu_c\, \varphi_{\beta}:L_{\alpha}\to L_{\alpha})=
   \Tr\, ( \varphi_{\alpha^{-1}} \mu_c:L_{\beta}\to L_{\beta})\, .
\end{equation}

      Axiom  (2)
    implies that    $L_1\subset L$  lies in the center  of~$L$.
 In particular, $L_1$ is a commutative associative
    $K$-algebra with unit.  The
     group $\orlin$ acts on $L_1$ by
    algebra automorphisms.  This action
    determines the dimensions of all $L_\alpha$:    applying
   (4)   to $\beta=1$  and
    $c=1_L\in L_1$,  we obtain
    $$ \Dim\,  L_{\alpha} =  \Tr\, (  \id:L_{\alpha}\to L_{\alpha})
    =\Tr\, (  \varphi_{1}:L_{\alpha}\to L_{\alpha})
    =\Tr\, (  \varphi_{\alpha^{-1}} :L_{1}\to L_{1})\, .$$

Isomorphisms of crossed $G$-algebras are algebra isomorphisms
preserving the $G$-grading and the inner product, and commuting with
the action of $G$.

 The {\it
direct sum}  $L\oplus L'$ of two crossed $\orlin$-algebras $L,L'$ is
the crossed  $\orlin$-algebra  defined by $(L\oplus L')_{\alpha}=
L_{\alpha}\oplus  L'_{\alpha}$ for $\alpha\in \orlin$ with
coordinate-wise multiplication, inner product, and action of $G$.
For  $k\in K^*$, the {\it $k$-rescaling}  transforms a crossed
$\orlin$-algebra $(L,\eta, \varphi)$
  into the
    crossed  $\orlin$-algebra $(L,k\eta, \varphi)$.

 \subsection{Example}\label{exxama}
   Let
$\theta=\{\theta_{\alpha,\beta} \}_{\alpha,\beta\in \orlin}$
 be a normalized $K^*$-valued 2-cocycle of   $\orlin$.
The  cocycle identity  $ \theta_{\alpha,\beta}\,
\theta_{\alpha\beta,\gamma}=\theta_{\alpha,\beta\gamma}\, \theta_{
\beta,\gamma}  $ for   $\alpha,\beta, \gamma\in \orlin$ and the
normalization condition $\theta_{1,1}=1$ imply
  that    $\theta_{1,\alpha}=\theta_{\alpha,1}=1$ and
                             $\theta_{\alpha,\alpha^{-1}}=\theta_{\alpha^{-1}, \alpha}$ for all
                             $\alpha$.

 We
define  a  crossed $\orlin$-algebra $L=L^{\theta}$ as follows. For
$\alpha \in \orlin$, let $L_{\alpha}= K l_{\alpha}$   be the
one-dimensional vector space
  generated by a vector $l_{\alpha}$.
Multiplication in $L=\oplus_\alpha L_\alpha$ is defined by $
l_{\alpha} l_{\beta}= \theta_{\alpha,\beta} l_{\alpha\beta}$.  This
makes $L$ into an associative algebra with unit      $1_L=l_1\in
L_1$. The  inner product   on
 $L$ is   defined   by      $\eta(l_{\alpha}, l_{\alpha^{-1}})=
 \theta_{\alpha,\alpha^{-1}}$ for all $\alpha\in G$.  The action   of $\alpha\in G$
 is defined on the basis $\{l_\beta\}_\beta$ of $L$ by
$$\varphi_\alpha (l_\beta)=
  \theta_{\alpha,\alpha^{-1}}^{-1}\, \theta_{\alpha,\beta}\,
  \theta_{\alpha \beta, \alpha^{-1}} \, l_{\alpha \beta \alpha^{-1}}.$$
It is clear that $\varphi_\alpha$ is the only $K$-linear
automorphism of $L$ satisfying Axioms (1) and (2). The other axioms
of a crossed $G$-algebra are verified by direct
 computations using solely the cocycle identity and the normalization condition.

 Note that the isomorphism class of
  $L^{\theta}$ depends only on the  cohomology class of~$\theta$. For $k\in K^*$, denote $
L^{ \theta, k}$ the crossed $G$-algebra obtained from $ L^{
 \theta}$ by $k$-rescaling.   If $\theta=1$, then this yields  a
structure of a crossed
 $G$-algebra on the group ring $K[G]$.

   \subsection{Transfer}\label{transfcro}
Let $H$ be  a subgroup   of   $\orlin$ of finite index.  A
    crossed $H$-algebra $(L=\oplus_{\alpha\in H} L_\alpha,\eta, \varphi)$ gives rise  to a crossed
    $\orlin$-algebra
$(\widetilde L=\oplus_{\alpha\in \orlin} \widetilde
     L_\alpha, \widetilde \eta, \widetilde \varphi)$ called its {\it transfer} and defined as follows.
  Let $H\backslash \orlin$ be the set of right cosets of $H$ in $G$.   For
  each
     $i\in H\backslash \orlin$, fix a representative $\omega_i\in
   G$ so that $i=H\omega_i$. (It is convenient but not necessary
     to take $\omega_i=1 $ for  $i=H$.) For $\alpha\in \orlin$, set
     $$N(\alpha)= \{i\in H\backslash \orlin\, \vert \,
      \omega_i\,  \alpha \,  \omega_i^{-1} \in
     H\} \quad \quad
  {\text {and}} \quad \quad \widetilde  L_{\alpha}=\oplus_{  i  \in N(\alpha)} L_{\omega_i
     \alpha
     \omega_i^{-1}}\, .$$   In
     particular, if
     ${\alpha}$ is not conjugate  to   elements of
     $H$, then
     $\widetilde  L_{\alpha}=0$.

     We provide $\widetilde  L=\oplus_{\alpha}\widetilde  L_{\alpha}$ with
     coordinate-wise multiplication. Thus, for $\alpha,\beta \in \orlin$,  the multiplication $\widetilde  L_{\alpha}
     \times
     \widetilde  L_{\beta} \to \widetilde
     L_{\alpha\beta}$  restricted to $L_{\omega_i \alpha
     \omega_i^{-1}}\times
     L_{\omega_j \beta \omega_j^{-1}}$  is  $0$, if
     $i\neq j$, and is induced by multiplication   in $L$
     $$L_{\omega_i \alpha \omega_i^{-1}}\times L_{\omega_i \beta
     \omega_i^{-1}} \to
     L_{\omega_i \alpha \beta  \omega_i^{-1}},$$  if $i=j\in
     N(\alpha)\cap N(\beta)$.
    By definition, the algebra $\widetilde  L_{1} $
     is a direct sum of   $[\orlin:H]$ copies of $L_1$ numerated by the elements  of $H\backslash G$.
     The corresponding sum of     copies of   $1_L\in  L_{1}$ is the unit
     of $
     \widetilde  L_{1}$. The
  inner product
   $\widetilde
     \eta$ on $\widetilde  L
      $   is defined by
$$\widetilde  \eta\vert_{\widetilde  L_{\alpha}\otimes \widetilde  L_{\alpha^{-1}}}=
     \bigoplus_{i\in N(\alpha)=N(\alpha^{-1})}
     \eta\vert_{L_{\omega_i\alpha \omega_i^{-1}} \otimes L_{\omega_i
     \alpha^{-1}\omega_i^{-1}}} \, . $$

We now define   $\widetilde \varphi$. The group  $\orlin$ acts on
   $H\backslash \orlin$ by $  {\alpha} (j) = j\alpha^{-1}$ for  $\alpha\in
   \orlin$, $ j
   \in
   H\backslash \orlin$.   The equality $H \omega_{\alpha(j)} = H\omega_j
   \alpha^{-1}$ implies that  $\alpha_j=\omega_{\alpha(j)} \alpha \omega_j^{-1} \in
   H$. Note that the
     fixed point set of the  map
   $H\backslash \orlin \to H\backslash \orlin$, $ j \mapsto \alpha(j)$ is
   $N(\alpha)$.   Therefore for any $\beta \in G$, this map
sends   $N(\beta)$ bijectively  onto $ N(\alpha \beta
   \alpha^{-1})$.
For   $j\in N(\beta)$, consider
 the
    homomorphism \begin{equation}\label{hoho}\varphi_{\alpha_j}: L_{\omega_j
   \beta \omega_j^{-1}} \to L_{\alpha_j\omega_j
   \beta \omega_j^{-1}\alpha_j^{-1}}= L_{\omega_{\alpha(j)} \alpha \beta \alpha^{-1}
   \omega_{\alpha(j)}^{-1}}. \end{equation}
   We define
    $\widetilde  \varphi_{\alpha}:\widetilde  L_{\beta} \to
   \widetilde  L_{\alpha \beta \alpha^{-1}}$ to be the direct
   sum of these  homomorphisms   over all $j\in N(\beta)$.
The identity $(\alpha \alpha')_j={\alpha}_{{\alpha'}(j)}  \alpha'_j$
   for any  $ \alpha, \alpha'\in \orlin$  implies that
   $\widetilde  \varphi_{\alpha\alpha'}=\widetilde  \varphi_{\alpha}\, \widetilde
   \varphi_{\alpha'}$. Hence
    $\widetilde  \varphi$ is an action of $\orlin$  on $\widetilde  L$.

     \begin{lemma}\label{l2-jiiif}
      The triple $( \widetilde  L, \widetilde \eta, \widetilde \varphi)$ is a crossed
      $G$-algebra. \end{lemma}

     \begin{proof}  That $\widetilde \eta$ is an inner product preserved by $\widetilde \varphi$ follows
     directly from the definitions and properties of $L$.
    We check
      Axioms (1), (2), (4) of Section \ref{crocroa}.

       For   $j\in N(\alpha)$, we have  $\alpha(j)=j$ and
       $ {\alpha}_j=\omega_j \alpha \omega_j^{-1}$. By Axiom (1) for $L$, the
       homomorphism
       (\ref{hoho})  with  $\beta=\alpha $ is the identity. This yields   Axiom
       (1)
       for~$\widetilde  L $.

  Let
         $a\in L_{\omega_i \alpha \omega_i^{-1}}\subset \widetilde  L_{\alpha}, b\in
         L_{\omega_j \beta \omega_j^{-1}}
         \subset \widetilde  L_{\beta}$
          with
         $i\in N(\alpha), j\in N(\beta)$.  Then
         $$\widetilde  \varphi_{\alpha} (b)=\varphi_{{\alpha}_j}(b)
         \in
         L_{\omega_{{\alpha}(j)}  { \alpha} \beta \alpha^{-1}
           \omega_{{ \alpha}(j)}^{-1}}.$$
 If   $i\neq j$
         then ${ \alpha}(j)\neq { \alpha}(i)=i$
        and
         $\widetilde  \varphi_{\alpha} (b) a=0=ab$.
      If $i=j$, then
         ${ \alpha}_j=\alpha_i=\omega_i { \alpha} \omega_i^{-1}$.
         By Axiom    (2)    for $L$,
         $$\widetilde  \varphi_{\alpha} (b) a= \varphi_{\alpha_j} (b)
         a
          =\varphi_{\omega_i {
         \alpha}
         \omega_i^{-1}} (b) a=ab.$$

  We now check  Formula (\ref{traces})  with $L, \varphi$ replaced by $\widetilde L, \widetilde \varphi$. Since both sides
   are linear functions of $c$,
    it suffices to treat the case where $$c\in
    L_{\omega_i\alpha\beta \alpha^{-1} \beta^{-1}\omega_i^{-1}} \subset
    \widetilde  L_{\alpha\beta \alpha^{-1}
    \beta^{-1}}$$
    with $i\in N( {\alpha\beta \alpha^{-1}
    \beta^{-1}})$. A direct  application of   definitions shows that both
    sides of
    the desired formula are equal to $0$ unless $i\in N(\alpha)\cap N(\beta)$.
    If $i\in N(\alpha)\cap N(\beta)$ then
    the trace of $ \mu_c\, \widetilde \varphi_{\beta}:\widetilde  L_{\alpha}\to \widetilde
    L_{\alpha}
 $
  is equal to  the trace of the endomorphism
     $\mu_c\, \varphi_{\omega_i\beta \omega_i^{-1}}$ of $\widetilde
    L_{\omega_i\alpha
    \omega_i^{-1}}$. The trace of
    $  \widetilde \varphi_{\alpha^{-1}} \mu_c:\widetilde  L_{\beta}\to \widetilde
    L_{\beta} $
   is equal to  the trace of the endomorphism
    $\varphi_{\omega_i \alpha^{-1}\omega_i^{-1}} \mu_c$ of
    $\widetilde  L_{\omega_i\beta \omega_i^{-1}}$. The equality of these two
    traces
    follows from Axiom
  (4)  for   $L$. \end{proof}

It follows from Lemma \ref{trac} below that the isomorphism class of
    $(\widetilde  L, \widetilde
  \eta, \widetilde \varphi)$   does not depend on the choice of  the
  representatives $\{\omega_i\}_i$.   A direct algebraic proof is left to the reader as    an
  exercise.

 We can apply transfer to the crossed $H$-algebra $L^\theta$ derived from a   $K^*$-valued normalized 2-cocycle
$\theta $ on $H$.  The resulting crossed
  $\orlin$-algebra is
denoted    $ L^{\orlin, H,\theta}$. For $k\in K^*$, denote $
L^{\orlin, H,\theta, k}$ the crossed $G$-algebra obtained from $
L^{\orlin, H,\theta}$ by $k$-rescaling. It is clear that transfer
and  rescaling commute so that $ L^{\orlin, H,\theta, k}$ is the
transfer of the crossed $H$-algebra $ L^{\theta, k}$.

  \subsection{Semisimple   $G$-algebras}\label{sssns}
  A  crossed $\orlin$-algebra  $ L=\oplus_{\alpha\in \orlin}\,
  L_{\alpha} $
    is   {\it semisimple}  if
    $L_1$ is  a direct sum of
    several  mutually annihilating (and therefore mutually orthogonal)  copies of   $K$.  For instance, the
  crossed  $\orlin$-algebra $L^\theta$
   in Example \ref{exxama} is semisimple   since
  $L_1=K $. Direct sums, rescalings,
   and transfers of  semisimple crossed  group-algebras are
  semisimple.

  If a crossed $\orlin$-algebra  $L$ is semisimple, then $L_1$
   has a basis $I$
  such that $i j=\delta^i_j i$
       for all   $i,j\in I$, where $\delta^i_j$ is the Kronecker delta. Such a basis   is
       unique
       and denoted
  $\mid (L)$. Its elements   are called the {\it  basic
    idempotents}  of $L$. Algebra
    automorphisms  of $L_1 $ preserve  $\mid (L)$ set-wise. In particular, the
    action
     of $\orlin$ on~$L $ restricts to an action of $G$ on $\mid(L)$. We
 call $L$   {\it simple} if (it is semisimple and) the action of
     $G$ on $\mid (L)$ is transitive.
     The  crossed  $\orlin$-algebra  $L^\theta$
           in Example \ref{exxama} is simple because $\mid(L^\theta)$ is a one point set.
The
 crossed  $\orlin$-algebra   $ L=L^{\orlin, H,\theta, k}$ constructed at the end
 of
     Section \ref{transfcro} is simple
  because $   L_{1} $ is a direct sum of
  copies of   $ L^\theta_{1}=K$ numerated by elements of
     $H\backslash \orlin$ and the action of $ \orlin$ on $
     L_{1} $ permutes these copies of $K$ via the natural (transitive)
     action of $\orlin$ on $H\backslash G$.

 Consider again a semisimple crossed $G$-algebra $L$ and set $I=\mid(L)$.
  The equality $1_L=\sum_{i\in I} i$      and the fact
 that
   $I$  lies in the center of $L$
     imply that $L$ splits as a direct sum of   mutually annihilating
subalgebras $\{iL\}_{i\in I}$:
     $$L=\bigoplus_{i\in I} iL\quad \quad {\text {where}}\quad
     \quad iL=Li=\bigoplus_{\alpha\in \orlin} iL_\alpha.
     $$ The  subalgebras $\{iL\}_{i\in I}$ of $L$  are
  permuted via the
     action
      of $\orlin$ on $L$.  Therefore   for any orbit  $\kappa\subset  I$ of the action  of
      $\orlin$ on $ I$, the vector space  $\kappa  L =\oplus_{i\in \kappa} \,
      iL$ is a $G$-invariant subalgebra of $L$.
      Restricting the action of $G$ and the inner product on $L$ to $\kappa L$ we
      transform the latter into a   simple  crossed $G$-algebra.
      We conclude that each semisimple crossed $G$-algebra
      canonically splits as a direct sum of  simple crossed
      $G$-algebras.

      We now  classify  simple crossed
$\orlin$-algebras.   Denote by $\mathcal A (\orlin)$ the set of
isomorphism classes of pairs (a   simple crossed $\orlin$-algebra $L
$, a basic idempotent $i \in \mid(L)$). (Two such pairs are
isomorphic if there is an isomorphism of the crossed $G$-algebras
preserving the distinguished basic idempotent).  The group $\orlin$
acts on $\mathcal A (\orlin)$ by $\alpha(L, i )= (L,
  \varphi_{\alpha} (i ))$ for $\alpha\in \orlin$. Let
$\mathcal B (\orlin)$ be the set of triples (a subgroup $H\subset
\orlin$ of finite index, a cohomology class $\theta\in H^2(H;K^*)$,
an element $k$ of $K^*$). The group $\orlin$ acts on $\mathcal B
(\orlin)$ by  $\alpha(H,\theta, k)= (\alpha H\alpha^{-1},
\alpha_*(\theta), k)$ where $\alpha\in \orlin$ and
$\alpha_*:H^2(H;K^*)\to H^2(\alpha H\alpha^{-1};K^*)$ is the
isomorphism induced by the conjugation by $\alpha$.

\begin{lemma}\label{lelecrocro2}  There is a canonical
$\orlin$-equivariant bijection $\mathcal A (\orlin)\to \mathcal B
(\orlin)$. \end{lemma}

\begin{proof}  Let $ L=\oplus_{\alpha\in \orlin} \,
L_{\alpha} $ be a   simple crossed $\orlin$-algebra and $i \in
\mid(L)$. Let $\orlin_i= \{\alpha \in \orlin\, \vert\,
\varphi_\alpha(i)=i\}$ be the stabilizer of  $i $.
 We first compute the dimension of $iL_\alpha
\subset L_\alpha$ for all  $\alpha\in \orlin$.  Applying
(\ref{traces}) to $\beta=1$ and $c=i\in L_1$, we obtain
$$\dim (iL_\alpha)=\Tr(\mu_i:L_\alpha\to L_\alpha)  =\Tr(  \varphi_{\alpha^{-1}}\mu_i:L_1\to L_1).
$$
  The endomorphism $
\varphi_{\alpha^{-1}}\mu_i $ of $L_1$  carries  $i$ to
$\varphi_{\alpha^{-1}}(i)\in I$ and    carries all other elements of
the basis $I $ to $0$. Hence
$$\dim (iL_\alpha)=\begin{cases} 1~ {  \rm {if}}\,\,\, \alpha \in \orlin_i, \\
  0~   {  \rm {if}}\,\,\, \alpha \in G- \orlin_i.\end{cases}
 $$     For  every  $\alpha \in \orlin_i-
\{1\}$, pick a non-zero vector $s_\alpha\in i L_\alpha$. For
$\alpha=1$,  set $s_\alpha=i \in i  L_1$. Then for any
$\alpha,\beta\in \orlin_i$, we have $s_\alpha s_\beta
=\nabla_{\alpha,\beta} \, s_{\alpha  \beta}$ with
$\nabla_{\alpha,\beta} \in K$.  We claim that $\nabla_{\alpha,\beta}
\neq 0$.
 Indeed,
the inner product $\eta$ in $L$ satisfies   $\eta(s_\beta
s_{\beta^{-1}}, 1_L)= \eta(s_\beta, s_{\beta^{-1}})\neq 0$. Hence
$s_\beta s_{\beta^{-1}}=\nabla s_1=\nabla  i $ with
  $\nabla \in K^*$ and
\begin{equation}\label{nonzi} (s_\alpha s_\beta) s_{\beta^{-1}}= s_\alpha (s_\beta
s_{\beta^{-1}})= \nabla  s_\alpha  i = \nabla  s_\alpha \neq 0.  \end{equation} Therefore $s_\alpha
s_\beta\neq 0$ and  $\nabla_{\alpha,\beta} \neq 0$.

The associativity of multiplication in $L$ and the choice $s_1=i$
ensure  that $ \{\nabla_{\alpha,\beta}\}_{\alpha,\beta}$ is a
 normalized $K^*$-valued 2-cocycle on $\orlin_i$. Let $\nabla_i \in H^2(\orlin_i;K^*)$ be its
cohomology class. Under a different choice of $\{s_\alpha\}_{\alpha
 }$ we obtain   the same
  $\nabla_i$.

 The formula $(L,i)\mapsto (G_i, \nabla_i, \eta(i,i))$ defines  a $\orlin$-equivariant mapping
 $\mathcal A (\orlin)\to \mathcal B (\orlin)$.   In the opposite direction, any triple  $(H, \theta, k)\in \mathcal B (\orlin)$ determines
  a simple crossed  $\orlin$-algebra
$\widetilde L=L^{\orlin, H,\theta, k}$. The elements of
$\mid(\widetilde L)$ bijectively correspond to the elements of
$H\backslash \orlin$.
  Let
$i_\theta\in  \mid(\widetilde L)$  correspond  to $H\backslash H \in
H\backslash \orlin$. The formula $(H,\theta, k)\mapsto (\widetilde
L, i_\theta)  $ defines a mapping $ \mathcal B (\orlin)\to \mathcal
A (\orlin) $.

 We claim that these mappings  $ \mathcal B (\orlin)\to  \mathcal A (\orlin) $ and $ \mathcal A (\orlin)\to
\mathcal B (\orlin) $   are mutually inverse bijections. Let us
verify that  the composition  $ \mathcal B (\orlin)\to \mathcal A
(\orlin) \to \mathcal B (\orlin) $ is the identity. Consider  the
pair $(\widetilde  L, i_\theta )$ derived as above from $(H, \theta,
k)\in \mathcal B (\orlin)$.    The stabilizer of
$i_\theta=H\backslash H \in H \backslash \orlin= \mid(\widetilde L)$
with respect to the natural action of $\orlin$    is $H$. It follows
from the definition of $\widetilde L$ that $i_\theta\widetilde
L_\alpha=L^\theta_\alpha$ for all $\alpha\in H$. Set
 $s_\alpha=l_\alpha\in L^\theta_\alpha= i_\theta\widetilde
 L_\alpha$,
where $l_\alpha$ is the vector used in the definition of $
L^\theta$. Now, it is obvious that  the mapping $ \mathcal A
(\orlin)\to \mathcal B (\orlin) $ carries    $(\widetilde
L,i_\theta)$ to $(H,  \theta, k )$.

To accomplish the proof, it suffices  to show that the mapping $
\mathcal A (\orlin)\to  \mathcal B (\orlin) $  is injective. It is
enough  to prove that a      simple crossed $\orlin$-algebra
 $L $  with distinguished    basic idempotent $i_0\in
L_1$ can be uniquely reconstructed from the triple (the stabilizer
$H\subset \orlin$ of $i_0$, the  $H$-algebra $i_0 L$,   the element
$k=\eta(i_0,i_0)$ of $K$). Recall that $L=\oplus_{i\in I}\, iL$,
where $I=\mid (L)$. Since the action of $\orlin$ on $I$ is
transitive, for each $i\in I$ there is
  $\omega_i\in \orlin$ such that $\varphi_{\omega_i}(i_0)=i$. We   take
$\omega_{i_0}=1$. The isomorphism $\varphi_{\omega_i}:L\to L$ maps
$i_0 L$ bijectively onto $iL$. Pick a non-zero vector
    $s_\alpha$   in $i_0 L_\alpha\cong  K$  for all $\alpha\in H$.
 Then the set $\{\varphi_{\omega_i}
(s_{\alpha})\, \vert \, i\in I, \alpha \in H\}$ is a  basis of~$L$.
Multiplication in $L$ is computed in this basis by
$$\varphi_{\omega_i} (s_{\alpha}) \,\varphi_{\omega_{i'}}
(s_{\alpha'})=
\begin{cases} \varphi_{\omega_i} (s_{\alpha} s_{\alpha'})~ {
\rm {if}}\,\,\,i=i',
\\
  0~ {  \rm {if}}\,\,\,i \neq i'.\end{cases}$$
  The inner product
$\eta$ on $L$ is determined by the algebra structure of $L$ and the
equalities $\eta(i,1_L)=\eta(i,i)=k$  for all $i\in I $. It remains
to recover the action $\varphi$ of $\orlin$ on $L$. Given $\alpha\in
H$, the homomorphism $\varphi_{\alpha}:i_0L\to i_0L$ is fully
determined by the condition $\varphi_{\alpha}(s_{\beta})
s_{\alpha}=s_{\alpha} s_{\beta}$ for all $\beta\in H$ (we use  that
$s_{\alpha\beta
 \alpha^{-1}} s_\alpha\neq 0$, cf.\ (\ref{nonzi})). Each
$\beta\in \orlin$ expands as a product $\omega_i\alpha$ where
$i=\varphi_\beta (i_0)$ and  $\alpha\in H$. This gives
$\varphi_{\beta}=\varphi_{\omega_i}\varphi_{\alpha}$ and computes
the restriction  of $\varphi_{\beta}$ to $i_0L$.  Knowing these
restrictions for all $\beta\in \orlin$, we can   recover
   $\varphi_\beta:L\to L$  by $\varphi_\beta (\varphi_{\omega_i}
   (s_{\alpha}))=\varphi_{\beta \omega_i} (s_{\alpha})
 $. \end{proof}

\section {The underlying $\orlin$-algebra}

We study   a  two-dimensional  $X$-HQFT $(A,\tau)$ where
$X=K(\orlin,1)$ with base point $x $. We derive from $(A,\tau)$ an
\lq\lq underlying" crossed
  $\orlin$-algebra.

\subsection{Conventions} Given an   oriented  surface $W$ and a  component $M$ of $ \partial
W$, we    write $M_{+}$ (resp.\ $M_{-}$) for $M$   with orientation
induced from that in $ W$  (resp.\ in $- W$). When all components
$M$ of $ \partial W$ are labeled with signs $\varepsilon(M)=\pm$, we
have  in the category of oriented manifolds
 $\partial W=\sum_M  \varepsilon (M)\, M_{ \varepsilon (M)}$. We view
 such
 $W$ as a cobordism whose bottom (resp.\ top) base is
formed by the boundary components labeled with $-$ (resp.\ with
$+$).

\subsection{Disks  with holes} We describe the structures of $X$-cobordisms on annuli  and disks
with two holes.  Set $C=S^1\times [0,1]$. Fix  once and for all an
orientation in $C$. Set   $C^0=S^1\times 0 \subset \partial C$ and
$C^1=S^1\times 1\subset \partial C$.   It is convenient   to think
of $C$ as of an annulus in  $\Bbb R^2$ with clockwise orientation
such that $C^0$ (resp.\ $C^1$) is its internal (resp.\ external)
boundary component. We provide $C^0,C^1$ with base points $ s \times
0,  s \times 1$, respectively, where $s\in S^1$.
 Given     $\varepsilon, \mu=\pm 1$,    denote by
$C_{\varepsilon,\mu}$ the oriented   annulus  $C$ with  oriented
pointed boundary $ C^0_{\varepsilon}\cup C^1_{\mu} $.
  For example, if $C$ lies in $\Bbb R^2$  as
  above, then both components of $\partial C_{-+}$ are oriented clockwise.

The homotopy class of a map $g:C_{\varepsilon,\mu} \to X$ is
determined by the elements  $\alpha,\beta$ of $ \orlin$, represented
by the loops $g\vert_{C^0_{\varepsilon}}$ and~$g\vert_{s \times
[0,1]}$, respectively. Here     $[0,1]$ is oriented from 0 to 1.
Note that    $g\vert_{C^1_\mu}$ represents   $ \beta^{-1}
\alpha^{-\varepsilon \mu} \beta \in \orlin$. The $X$-surface
($C_{\varepsilon,\mu}$, the map   $C_{\varepsilon,\mu}\to X$
corresponding to $\alpha,\beta \in \orlin$) is denoted
$C_{\varepsilon,\mu} (\alpha; \beta )$.


 Let $D$ be a  2-disk with two holes. Let $Y,Z,T$ be the boundary
components of $D$
   with base points $y,z,t$, respectively.  Fix   proper embedded
arcs $ty$ and $tz$ in $D$ oriented  from  $t$ to $y,z$  and meeting
solely
 at~$t$.  We provide $D$ with the orientation obtained by
rotating $ty$ towards $tz$ in $D$ around    $t$. One can view $D$ as
  a disk with two holes in $\Bbb R^2$   with clockwise orientation
such that $Y,Z$   are  the internal   boundary components, $T$ is
the external boundary component, and $y,z,t$ are the bottom points
of $Y$, $Z $, $T$. Given $\varepsilon , \mu,\nu=\pm$,   let
$D_{\varepsilon,\mu,\nu}$ be $D$ with oriented pointed boundary $
Y_{\varepsilon}\cup Z_{\mu}\cup T_{\nu} $.     For example, if $D$
lies in $\Bbb R^2$  as
  above, then all components of $\partial  D_{--+}$ are oriented clockwise.

  To a map $g:D_{\varepsilon,\mu,\nu}\to X$ we assign  the
homotopy classes of the loops $g\vert_{Y_{\varepsilon}}$, $
g\vert_{Z_{\mu}},g\vert_{ty}, g\vert_{tz}$. This gives  a bijection
from the set of homotopy classes of maps $ D_{\varepsilon,\mu,\nu}
\to  X$ onto $\orlin^4$. For a  tuple $(\alpha,\beta,\rho,
\delta)\in \orlin^4 $, denote  $D_{\varepsilon,\mu,\nu} (\alpha,
\beta;\rho, \delta)$ the disk with holes $D_{\varepsilon,\mu,\nu}$
endowed with the map $g$ to $X$ corresponding to this tuple. Note
that the loops $g\vert_{Y_{\varepsilon}}, g\vert_{Z_{\mu}},
g\vert_{T_{\nu}}$ represent    $\alpha, \beta,  \gamma= (\rho
\alpha^{-\varepsilon} \rho^{-1} \delta \beta^{-\mu}
\delta^{-1})^{\nu}$, respectively. As an exercise, the reader may
construct an $X$-homeomorphism
\begin{equation}\label{cyc}D_{\varepsilon,\mu,\nu} (\alpha, \beta;\rho,
\delta)\approx  D_{\mu,\nu,\varepsilon} ( \beta, \gamma;\rho^{-1}
\delta, \rho^{-1}) \,  \end{equation} cyclically permuting the
components $Y,Z,T$ of $\partial D$.

\subsection{Algebra $L $}\label{algebral}   A  connected   $X$-curve $M$ (i.e., a connected 1-dimensional $X$-manifold)  represents  an element
   $\alpha=\alpha(M)$
  of  $\pi_1(X,x)=\orlin$, cf.\ Section \ref{statesu++}. We shall sometimes
   write $(M,\alpha)$ for $M$.
The vector space  $A_M$  depends only on $\alpha(M)$     up to
canonical isomorphisms. This follows from Lemma~\ref{l1908} because
any two oriented pointed circles are related by a (unique up to
isotopy) orientation preserving and base point preserving
homeomorphism. In this way,  for every $\alpha\in \orlin$,  the HQFT
$(A,\tau)$ gives a finite dimensional vector space. It is
denoted~$L_{\alpha}$.

We provide
 $L=\oplus_{\alpha\in \orlin} L_{\alpha}$
with multiplication and inner product as follows. For $\alpha,
\beta\in G$, the disk with two holes $D_{--+} (\alpha, \beta;1, 1)$
is an $X$-cobordism between the $X$-curves $(Y_-,\alpha) \amalg
(Z_-,\beta)$ and $(T_+,\alpha \beta)$. Consider the associated
homomorphism  $\tau (D_{--+} (\alpha, \beta;1, 1)):
L_{\alpha}\otimes L_{\beta}\to L_{\alpha \beta} $. We
   define    multiplication in $L$ by
$$ab=\tau
(D_{--+} (\alpha, \beta;1, 1)) (a\otimes b) \in L_{\alpha \beta}\, ,
 $$
for $a\in  L_{\alpha}, b\in
   L_{\beta}$.
The annulus $C_{--} (\alpha; 1)$ is an $X$-cobordism between
   $(C^0_-,\alpha) \amalg (C^1_-,{\alpha^{-1}})$ and $\emptyset$.   Set
  $$\eta_\alpha= \tau (C_{--} (\alpha;
1)):  L_{\alpha} \otimes   L_{\alpha^{-1}} \to K. $$  The properties
of HQFTs stated in Section \ref{ee3.2} imply that $\eta_\alpha $ is
non-degenerate for all $\alpha$ and $\eta_{\alpha^{-1}}$ is obtained
from $\eta_\alpha$ by the permutation of the tensor factors. Let
  $\eta: L\otimes L \to K$   be the  symmetric  bilinear form whose
restriction to $L_\alpha \otimes L_\beta$ with $\alpha, \beta \in
\orlin$ is zero if $\alpha \beta \neq 1$ and is $\eta_\alpha$ if
$\alpha \beta =1$.

The group $\orlin$ acts on $L$
  as follows. For $\alpha,\beta \in \orlin$, the annulus
$C_{-+} (\alpha; \beta^{-1})$ is an $X$-cobordism between
$(C^0_-,\alpha) $ and $(C^1_+,\beta\alpha \beta^{-1})$.   Set
  $$\varphi_{\beta}=\tau (C_{-+} (\alpha;
\beta^{-1})):  L_{\alpha} \to  L_{\beta\alpha \beta^{-1}}. $$
Axioms (3) and (6) of an HQFT
 yield that $\varphi_{ \beta \gamma}=  \varphi_{\beta} \varphi_{\gamma} $
 for   $ \beta,  \gamma\in G$
and  $\varphi_1=\id$.

\begin{lemma}\label{oplik}    The triple $(L,\eta,\varphi)$
   is a crossed   $\orlin$-algebra. \end{lemma}

\begin{proof}  Let us prove that $(ab)c=a(bc)$ for any $a\in L_\alpha,
b\in L_\beta, c\in L_\gamma$ with $\alpha, \beta, \gamma \in
\orlin$. Consider the $X$-cobordisms $W_0= D_{--+}(\alpha, \beta;1,
1) \amalg C_{-+} (\gamma;1)$ and $W_1=  D_{--+}(\alpha \beta,
\gamma;1, 1)$.   Here $W_0$ is an $X$-cobordism between $(Y_-,\alpha
) \amalg (Z_-, \beta) \amalg (C^0_-,\gamma)$ and $ (T_+,\alpha\beta)
\amalg (C^1_+,\gamma)$. By Axioms (4)    and (6) of an HQFT  and the
definition of multiplication in $L$, the homomorphism $\tau(W_0):
L_\alpha \otimes L_\beta \otimes L_\gamma \to L_{\alpha \beta}
\otimes L_\gamma$ carries  $a\otimes b \otimes c$ into $ab \otimes
c$. The homomorpism $\tau(W_1):   L_{\alpha \beta} \otimes
L_\gamma\to L_{\alpha \beta \gamma}$ carries  $a b \otimes c$ into
$(ab) c$. The gluing of $W_0$ to $W_1$ along  an $X$-homeomorphism
$(T_+,\alpha\beta)\amalg (C^1_+,\gamma) \approx
(Y_-,\alpha\beta)\amalg (Z_-,\gamma)$ yields an $X$-cobordism  $W$
and
$$\tau(W) (a\otimes b \otimes c)= \tau(W_1)\, \tau (W_0) ( a\otimes
b \otimes c) =  \tau(W_1)   ( a b \otimes c) =(ab)c.$$ The same
$X$-cobordism $W$ can be also obtained  by gluing the $X$-cobordisms
$ C_{-+} (\alpha;1)  \amalg  D_{--+}(\beta, \gamma;1, 1)$ and $
D_{--+}(\alpha, \beta  \gamma;1, 1)$ along an $X$-homeomorphism
$$(C^1_+,\alpha) \amalg (T_+, \beta\gamma)  \approx (Y_-,\alpha)\amalg
(Z_-, \beta\gamma)\, .$$  Therefore $\tau(W) (a\otimes b \otimes c)
=a (bc)$. Thus, $(ab)c=a(bc)$.

The unit of  $L$ is constructed as follows. Let $B_+$ be an oriented
2-disk whose boundary is pointed and   endowed with induced
orientation. There is only one homotopy class of maps $B_+\to X$.
The corresponding  homomorphism $\tau(B_+):K\to L_1$ carries $1_K\in
K$ into an element of $L_1$, denoted $1_L$.
 Let us prove that $1_L$ is a right unit of $L$ (the proof that it is a
left unit is similar). Let $a\in L_\alpha$ with $\alpha \in \orlin$.
Consider the $X$-cobordisms $W_0=
 C_{-+} (\alpha;1) \amalg B_+$ and $W_1=  D_{--+}(\alpha , 1;1, 1)$.
The gluing of $W_0$ to $W_1$ along  an $X$-homeomorphism $
(C^1_+,\alpha) \amalg (\partial B_+, 1) \approx (Y_-,\alpha )\amalg
(Z_-,1)$ yields an $X$-cobordism $X$-homeomorphic to $C_{-+}
(\alpha; 1)$. By Axioms (6)    and (3) of an HQFT, $$a =\tau(C_{-+}
(\alpha; 1)) (a )= \tau(W_1)\, \tau (W_0) ( a ) = \tau(W_1)   ( a
\otimes 1_L) =a1_L.$$

To show that $\eta$ is an inner product on $L$, we need only to
prove that $\eta(ab,c)=\eta(a,bc)$ for any $a \in L_{\alpha}$, $
b\in L_{\beta}$, $c \in L_{\gamma}$ with $\alpha, \beta, \gamma\in
\orlin$.
 If $\alpha \beta \gamma\neq1$, then $\eta(ab,c)=0=\eta(a,bc)$.
Assume that $\alpha \beta \gamma=1$.  Gluing    $C_{--}
(\alpha\beta;1)$ to   $D_{--+}(\alpha,\beta;1,1)$ along
$(C^0_-,\alpha\beta)\approx (T_+,\alpha\beta)$, we obtain
$D_{---}(\alpha,\beta;1,1)$. Hence
$$\eta(ab,c)= \tau (D_{---}(\alpha,\beta;1 ,1 )) (a\otimes b\otimes
c)\, .$$  Similarly, gluing   $C_{--}
(\alpha;1)$ to   $D_{--+}(\beta,\gamma;1,1)$ along
$(C^1_-,\alpha^{-1})\approx (T_+,\beta\gamma)$, we obtain
$D_{---}(\beta,\gamma;1,1)$. Hence
$$\eta(a,bc)= \tau (D_{---}(\beta,\gamma;1,1)) ( b\otimes c\otimes a)\, .$$
 By  (\ref{cyc}) (where $\rho=\delta=1$) and Axiom (2) of an
HQFT, $\eta(ab,c)=\eta(a,bc)$.

  Let us prove that $\varphi_{\rho}:L\to L$ is an algebra homomorphism for all $\rho\in G$.
 Pick  $a \in L_{\alpha}, b\in L_{\beta}$ with $\alpha, \beta
 \in \orlin$.  Note   that
 for  any  $  \rho,\delta\in \orlin$,  the
homomorphism
$$\tau (D_{--+}(\alpha,\beta;\rho,\delta)): L_{\alpha}\otimes
L_{\beta}\to  L_{\rho\alpha\rho^{-1} \delta\beta\delta^{-1}}$$
carries $a \otimes b$ to $\varphi_\rho (a)\, \varphi_\delta(b)  $.
This is so because $D_{--+}(\alpha,\beta;\rho,\delta)$ can be
obtained by gluing
 $C_{-+}(\alpha; \rho^{-1})\amalg  C_{-+}(\beta; \delta^{-1})$ to   $D_{--+}(\alpha,\beta;1,1)$.
Similarly, gluing
 $C_{-+} ( \alpha \beta ; \rho^{-1})$
  to   $D_{--+}( \alpha, \beta;1,1)$  along
$(C^0_-, \alpha\beta)\approx (T_+,\alpha\beta)$, we obtain
$D_{--+}(\alpha,\beta; \rho ,\rho )$. Therefore
 $$\varphi_{\rho}(ab)= \tau ( D_{--+}(\alpha,\beta; \rho ,\rho ))
( a\otimes b)=\varphi_{\rho}(a)
\,\varphi_{\rho} (b)\, .$$

Let us verify Axioms (1) -- (4) of a crossed $G$-algebra.

 (1)    The Dehn twist along the circle  $S^1\times
(1/2)\subset  C_{-+} (\alpha; 1)$  yields an  $X$-homeomorphism
$C_{-+} (\alpha; 1)\to C_{-+}(\alpha; \alpha^{-1})$. Axiom (2) of an
HQFT implies that
 $\varphi_{\alpha}\vert_{L_{\alpha}}= \tau(C_{-+}(\alpha;
\alpha^{-1}))=\tau(C_{-+}(\alpha;1))= \id$.

    (2) Consider a self-homeomorphism $f$ of the disk with two holes $D $
which is the identity on $T$ and    permutes   $(Y,y)$ and $(Z,z)$.
 We choose $f$
so that    $f(tz)= ty$ and   $f(ty)$ is   an  arc leading from $t$
to $z$ and homotopic to the product of four  arcs $ty,
\partial Y, (ty)^{-1}, tz$.   An easy
  computation   shows that $f$ is an $X$-homeomorphism
  $D_{--+}( \alpha, \beta;1,1) \to D_{--+}( \beta, \alpha;1,
\beta^{-1})$. Axiom  (2) of an HQFT   implies that the homomorphism
 $\tau (D_{--+}( \alpha, \beta;1,1)): L_{\alpha}\otimes L_{\beta}\to
L_{ \alpha \beta} $ is obtained from the homomorphism  $\tau
(D_{--+}( \beta, \alpha;1,\beta^{-1})) : L_{\beta}\otimes
L_{\alpha}\to  L_{ \alpha \beta}  $   by composing with the flip $
L_{\alpha}\otimes L_{\beta}\to L_{\beta}\otimes L_{\alpha}$.
Therefore for any $a\in L_\alpha, b\in L_\beta$
$$ab=
\tau (D_{--+}( \alpha, \beta;1,1))(a\otimes b)=\tau (D_{--+}( \beta,
\alpha;1,\beta^{-1})) (b\otimes a)= b \,\varphi_{\beta^{-1}}(a).$$
Replacing $a$ with $\varphi_\alpha(b)$ and $b$ with $a$, we obtain
$\varphi_\alpha(b)a =a \varphi_{\alpha^{-1}}
(\varphi_\alpha(b))=ab$.

(3)  The   identity $\eta (\varphi_{\beta}(a),\varphi_{\beta}(b))=
\eta (a,b)$ with $a\in L_\alpha, b\in L_{\alpha^{-1}}$ follows from
the fact that the annulus $C_{--}(\alpha;1)$ used to define
$\eta_\alpha$ may be obtained by gluing the annuli
$C_{-+}(\alpha;\beta^{-1})$, $C_{--}(\beta \alpha\beta^{-1};1)$, and
$C_{-+}(\alpha^{-1};\beta^{-1})$.

    (4)   Fix an orientation of   $S^1$ and a point $s\in S^1$.
     Let $P$ be the punctured torus obtained from
$  S^1\times S^1$  by removing a small open   2-disk   disjoint from
$S^1\times \{s\}$ and $  \{s\} \times S^1$. We   assume that
 the circle $\partial P$ meets  $S^1\times \{s\}$ and $  \{s\}  \times S^1$ precisely at the point $(s,s)$ and we take
 $(s,s)$ as the base point of $\partial P$.  Provide $P $ with orientation induced from the product orientation
in $S^1\times S^1$.

Fix  a map $g:P\to X=K(\orlin,1)$ such that $g(s, s)=x $ and the
restrictions of~$g$ to  $S^1\times \{s\}$ and $   \{s\}  \times S^1$
represent $\alpha,\beta \in \orlin$, respectively.  (The
orientations on   $S^1\times \{s\},   \{s\}  \times S^1$ are induced
by that of $S^1$.) Then the loop $ g\vert_{\partial P}: (\partial
P)_- =-\partial P\to X $ represents $\alpha\beta \alpha^{-1}
\beta^{-1}$.  We view  $(P,g)$  as an $X$-cobordism between $
((\partial P)_-,g\vert_{\partial P})$ and $\emptyset$. We compute
the associated homomorphism $\tau (P,g): L_{\alpha\beta \alpha^{-1}
\beta^{-1}}\to K$ in two different ways.
 Observe first that $(P,g)$ can be  obtained from
  $D_{--+}  (\alpha\beta \alpha^{-1} \beta^{-1},\alpha;1,\beta)$ by
gluing the boundary circles  $(Z_-,\alpha)$ and $(T_+,\alpha)$ along
an $X$-homeomorphism. (These circles    give   $S^1\times
\{s\}\subset P$.) The axioms of an HQFT imply that     $\tau (P,g) $
carries $c\in L_{\alpha\beta \alpha^{-1} \beta^{-1}}$ to the  trace
of the homomorphism $$L_\alpha\to L_\alpha, \,\,\,  d\mapsto  \tau
(D_{--+}  (\alpha\beta \alpha^{-1} \beta^{-1},\alpha;1,\beta))
(c\otimes d)=c\, \varphi_\beta (d)\, .$$  The same pair $(P,g)$ is
obtained from
  $D_{--+}  (\alpha\beta \alpha^{-1}
\beta^{-1},\beta;\alpha^{-1},\alpha^{-1})$ by gluing    along
$(Z_-,\beta)\approx (T_+,\beta)$. (The circles $Z_-$ and $T_+$  give
  $  \{s\}  \times S^1\subset P$.) Therefore $\tau (P,g) $
carries $c\in L_{\alpha\beta \alpha^{-1} \beta^{-1}}$ to the  trace
of the homomorphism
$$L_\beta \to L_\beta,\,\,\, d\mapsto \tau (D_{--+}   (\alpha\beta \alpha^{-1}
\beta^{-1},\beta;\alpha^{-1},\alpha^{-1}))  (c\otimes d)=
 \varphi_{\alpha^{-1}}(cd).$$ Thus,   $\Tr
(\mu_c\, \varphi_{\beta}:L_{\alpha}\to L_{\alpha})=\tau (P,g) (c)=
\Tr ( \varphi_{\alpha^{-1}} \mu_c:L_{\beta}\to
L_{\beta})$.\end{proof}

\subsection{Example}\label{theoopers--}
By  Section  \ref{ee3.3}, any $a\in H^2(G;K^*)$ determines a
2-dimensional $X$-HQFT $(A^a, \tau^a)$, where $X=K(G,1)$ with base
point $x$. By Section  \ref{exxama},   $a$ determines an isomorphism
class of crossed $\orlin$-algebras $L^a$.

 \begin{theor}\label{liftii+98}
 For any $a\in H^2(G;K^*)$, the underlying
crossed   $\orlin$-algebra of $(A^a, \tau^a)$  is isomorphic to
$L^{-a}$.
\end{theor}

\begin{proof} Set $S^1=\{z\in \CC\, \vert\, \vert z\vert=1\}$ with clockwise
 orientation and base point $ -i$.
  For  each $\alpha\in \orlin=\pi_1(X,x)$, fix a loop $  u_\alpha: S^1\to X$ carrying $-i$ to $x$ and representing $\alpha$. We choose $u_1$ to be the
 constant loop at $x$. Let $p:[0,1]\to S^1$ be the map carrying $t\in [0,1]$ to $-i\, \exp(-2\pi i t)\in S^1$.
   Let
$\Delta\subset \Bbb R^3$ be the standard 2-simplex   with the
vertices $v_0=(1,0,0)$, $v_1=(0,1,0)$, $v_2=(0,0,1)$.    For
$\alpha, \beta\in \orlin$,    pick a map $  f_{\alpha,
\beta}:\Delta \to X$ such that for all $t\in [0,1]$, $$ f_{\alpha,
\beta} ((1-t)
v_0+t v_1)= u_{\alpha } p(t),\, f_{\alpha,
\beta} ((1-t) v_1+t v_2)= u_{\beta} p(t)
,\, $$
$$f_{\alpha,
\beta} ((1-t) v_0+t v_2)= u_{\alpha\beta} p(t) \, .
$$  We   represent   $ a  \in H^2(G;K^*)= H^2(X;K^*)$  by a $K^*$-valued
singular 2-cocyle $\Theta $ on~$X$.  For any
$\alpha, \beta\in \orlin$,
let $\theta_{\alpha, \beta}= \Theta (f_{\alpha,\beta})
\in K^*$ be the evaluation of $\Theta$ on the singular simplex $f_{\alpha, \beta}$.
Observe that  $\theta_{\alpha, \beta}$  does not depend on the
choice of $f_{\alpha,\beta} $. Indeed, if
$f'_{\alpha,\beta}:\Delta \to X$ is another map as above, then the
formal difference $f_{\alpha,\beta}-f'_{\alpha,\beta}$  is a
singular 2-cycle in $X$. The homology class of this 2-cycle is
trivial because it can be realized by a   map $S^2\to X$ and
$\pi_2(X)=0$. Therefore $\Theta (f_{\alpha,\beta})=\Theta
(f'_{\alpha,\beta})$. A similar argument shows that $\theta=\{
\theta_{\alpha, \beta}\}_{\alpha, \beta\in G}$ is a 2-cocycle. Multiplying, if necessary,   $\Theta$ by a coboundary, we can ensure that
  $\theta_{1, 1}=1$. Both  $\Theta$ and
$\theta$  represent   $ a  \in H^2(G;K^*)$ and  $
\theta_-=\{ \theta^{-1}_{\alpha, \beta}\}_{\alpha, \beta\in G}$
represents $-a$.

  Let $L$ be the crossed  $\orlin$-algebra
underlying the
  HQFT $(A^\Theta, \tau^\Theta)$. We   prove that $L$ is isomorphic to  the crossed
    $\orlin$-algebra~$ L^{\theta_-}$ determined by the cocycle $\theta_-$.

  For   $\alpha\in \orlin$, the  pair $(S^1,    u_\alpha)$ is an $X$-curve denoted
    $M_\alpha$.
 The singular 1-simplex $p:[0,1]\to S^1$ is a fundamental cycle
  of $M_\alpha$. By definition of   $(A^\Theta, \tau^\Theta)$
    and  $L$,  this
     cycle   determines a  generating vector
    $p_\alpha=\langle p\rangle $ in   $L_\alpha=A^\Theta_{M_\alpha}\cong
    K$. We claim that
   $   p_\alpha p_\beta  =\theta^{-1}_{\alpha, \beta} \,   p_{\alpha\beta} $
    for all $\alpha, \beta \in \orlin$. To compute $   p_\alpha p_\beta$,
    we
  apply $\tau^\Theta$ to
    the disk with holes $D=D_{--+} (\alpha, \beta; 1,1)$ viewed as an $X$-cobordism between
     $M_\alpha\amalg M_\beta$ and $M_{\alpha \beta}$.   By
     definition, $   p_\alpha p_\beta  =k \,   p_{\alpha\beta} $
     for $k=g^*(\Theta)(B)$, where $g:D\to X$ is the map determined
     by the tuple $(\alpha, \beta; 1,1)$ and $B\in C_2(D;\ZZ)$ is a fundamental singular 2-chain  in $D$ such that $\partial B=
     p_{\alpha  \beta}   -p_\alpha- p_\beta$. Recall the segments $ty, tz\subset D$.
      Clearly, $D'=D/ty \cup tz$ is an oriented triangle with oriented edges and
     identified
      vertices.  Deforming   $g$
     in its homotopy class, we may assume that $g(ty \cup tz)=x$.
    Then
  $g$ expands as the composition of the projection $q:D\to D'$
     with a map $g':D'\to X$, and $k$   is
     the evaluation of $(g')^*(\Theta)$ on $q_*(B)$. Instead of  $q_*(B)$ we can use  more general   2-chains in $D'$.
       Namely, let
     $\partial D'=q(\partial D)$ be the union of the sides of $D'$.       Then $k= (g')^*(\Theta) (B')$ for any 2-chain $B'$ in $D'$
     such that $\partial B'=
     qp_{\alpha\beta}   -qp_\alpha- qp _\beta$ and the image of $B$ in
     $C_2(D',\partial D';\ZZ)$ represents the generator of $
     H_2(D',\partial D';\ZZ)=\ZZ$ determined by the orientation of $D'$. There is an obvious projection $f:\Delta
     \to D'$ identifying the vertices of $\Delta$ and such that
     $g'f=f_{\alpha, \beta}$ (up to homotopy   ${\rm {rel}}\, \partial \Delta$). The singular chain $B'=-f$
     satisfies all the conditions above. Therefore
     $$k= (g')^*(\Theta) (B')=((g')^*(\Theta) (f))^{-1}=(\Theta(g'f))^{-1}=
     (\Theta (f_{\alpha, \beta}))^{-1}=\theta^{-1}_{\alpha,
     \beta}\, .$$

   By definition of $L^{\theta_-}$,  the formula $  p_\alpha  \mapsto \ell_\alpha$ for
    $\alpha\in \orlin$ defines   an isomorphism of $\orlin$-algebras $L\cong L^{\theta_-}$.
   It remains to compare the  inner products and  the actions of $G$. The inner
   product,
     $\eta^-$, on $L^{\theta_-}$ satisfies $\eta^- (1_L,1_L)=1$.
  By Remark \ref{Eqcat++} below, the inner product $\eta$ on $L$ satisfies
  $\eta(1_L,1_L)=\tau^\Theta (S^2) =1 $.
Since $L_1=L^{\theta_-}_1=K$, we have $\eta^-=\eta $. The
isomorphism    $L\cong L^{\theta_-} $
  commutes with the action of $G$ because there is only one action of $G$ on $L^{\theta_-} $ satisfying  the
  axioms of a crossed $G$-algebra, cf.\   Section  \ref{exxama}.
 \end{proof}

\section{Properties of the underlying crossed $G$-algebras}\label{propunderalgas}

As above, $X=K(\orlin,1)$ with base point $x $.

\subsection{Functoriality}\label{Eqcat}  An isomorphism of   2-dimensional $X$-HQFTs induces an
isomorphism of the underlying crossed $G$-algebras in the obvious
way. This defines    a functor from the category of 2-dimensional
$X$-HQFTs and their isomorphisms to the category of crossed
$G$-algebras and their isomorphisms. This functor is   an
equivalence of categories, see \cite{Tu2}. We will need only the
following weaker claim.

 \begin{lemma}\label{liftii}
Any isomorphism of the  underlying crossed $G$-algebras of
2-dimen\-sional $X$-HQFTs is induced by an isomorphism of the HQFTs
themselves.
\end{lemma}

\begin{proof}
We first    show how to reconstruct a 2-dimensional $X$-HQFT
   $(A,\tau)$
  from its underlying crossed
$\orlin$-algebra $L=(L,\eta,\varphi)$. The discussion at the
beginning of  Section \ref{algebral}  shows that $A$ is determined
by $L$. We need only to reconstruct $\tau$.

It is clear that every 2-dimensional  $X$-cobordism $W$ can be
obtained by gluing several   $X$-cobordisms whose underlying
surfaces are disks with $\leq 2$ holes. Axioms of an HQFT imply that
$\tau(W)$ is determined by the values of $\tau$ on such
$X$-cobordisms and the restrictions $\{\eta_\alpha:L_\alpha\otimes
L_{\alpha^{-1}} \to K\}_{\alpha\in G}$ of $\eta$. It remains to show
that the values of $\tau$ on the  disks with $\leq 2$ holes are
determined by $L$.

We begin by computing   $\tau $ for annuli. An  $X$-annulus  is
$X$-homeomorphic  to  $C_{-+} (\alpha; \beta)$,  $C_{--} (\alpha;
\beta)$, or   $C_{++} (\alpha; \beta)$ with  $\alpha,\beta \in
\orlin$. By definition,     $ \tau (C_{-+} (\alpha; \beta))
=\varphi_{\beta^{-1}}\vert_{L_\alpha}$.
  The annulus $C_{--}(\alpha; \beta)$  can be obtained
by  gluing   $C_{-+}(\alpha; \beta)$ and
$C_{--}(\beta^{-1}\alpha\beta; 1)$ along
$(C^1_+,\beta^{-1}\alpha\beta)\approx
(C^0_-,\beta^{-1}\alpha\beta)$. Axiom (3) of an HQFT  and the
definition of $\eta$  imply that   $\tau (C_{--}(\alpha; \beta)):
L_\alpha \otimes L_{\beta^{-1}\alpha^{-1}\beta} \to  K $  carries
$a\otimes b$ to $\eta (\varphi_{\beta^{-1}} (a),b)$ for any  $a\in
L_\alpha $, $b\in L_{\beta^{-1}\alpha^{-1}\beta}$. To compute  the
vector $$\tau (C_{++}(\alpha; \beta))\in
 \Hom_K (K, L_\alpha\otimes
L_{\beta^{-1}\alpha^{-1}\beta})=L_\alpha\otimes
L_{\beta^{-1}\alpha^{-1}\beta },$$  we expand it  as a finite sum
$\sum_i p_i\otimes q_i$, where $p_i\in  L_\alpha$ and $q_i\in
L_{\beta^{-1}\alpha^{-1}\beta }$. The
 gluing of $C_{--}(\alpha^{-1};1)$ to
$C_{++}(\alpha; \beta)$ along  $(C^1_-,\alpha)\approx
(C^0_+,\alpha)$ yields
  $C_{-+}(\alpha^{-1};\beta)$. The axioms of an HQFT   yield that $\sum_i\eta_\alpha (a, p_i)
\,q_i=\varphi_{\beta^{-1}}(a)$ for all $a\in L_{ \alpha^{-1}}$.
 Since   $\eta_\alpha$   is non-degenerate,
this   uniquely determines     $\tau (C_{++}(\alpha; \beta))=\sum_i
p_i\otimes q_i$.

There are two $X$-disks:   $B_+$ as in the proof of Lemma
\ref{oplik}   and $B_-$    obtained from $B_+$ by reversing
 the orientation of   the boundary.    The vector $\tau (B_+)=1_L
 $
  is    determined by $L$.
The $X$-disk $B_-$ may be obtained by   gluing   $B_+$ and $C_{--}
(1;1)$  along $\partial B_+\approx C^0_-$. Therefore    $\tau (B_-):
L_1 \to K $ carries  any $a\in L_{1}$ to $\eta (1_L,a)$.

An  $X$-disk with 2 holes  $D$   splits
  along    disjoint  loops parallel to its boundary
components into  three $X$-annuli and a smaller $X$-disk  with 2
holes~$D'$. Choosing appropriate orientations of the    loops and an
appropriate map $D\to X$ in the given homotopy class, we can ensure
that $D'$ is $X$-homeomorphic to $D_{--+} (\alpha, \beta;1, 1)$ for
some  $\alpha, \beta\in \orlin$. The homomorphism $\tau (D')$ and
the values of $\tau $ on the annuli are then determined by $L$. The
axioms of an HQFT allow us to recover $\tau (D) $. We conclude that
$(A,\tau)$ can be entirely reconstructed from $L$.

We now prove the claim of the theorem. Let $(A,\tau)$ and
$(A',\tau')$ be 2-dimensional $X$-HQFTs with underlying crossed
$G$-algebras $L, L'$, respectively. An isomorphism   $\rho: L\to L'$
defines  in the obvious way an isomorphism $ A_M\to A'_{M}$ for any
connected
  $X$-curve $M$. This
extends to non-connected $M$ via~$\otimes$. We claim that these
   isomorphisms $\{ A_M\to A'_M\}_M$ make
the natural square diagrams associated with $X$-homeomorphisms and
$X$-cobordisms   commutative. The part concerning the homeomorphisms
is obvious. Each $X$-cobordism   can be obtained by gluing
$X$-surfaces of type $B_+ $,  $D_{--+} (\alpha, \beta;1, 1)$, and
  annuli. It suffices to check the commutativity of the
  corresponding
diagrams. For $B_+$ and $D_{--+} (\alpha, \beta;1,  1)$, the
commutativity follows from the assumption that $\rho $  is an
algebra isomorphism. For   annuli, the commutativity follows from
the computations above because $\rho $ preserves the inner product
and commutes with the action of $\orlin$.
\end{proof}

\subsection{Remark}\label{Eqcat++} The computation of $\tau(B_\pm)$  in the proof of
Lemma \ref{liftii} allows us to compute $\tau(S^2) \in K$    in
terms of the underlying crossed $G$-algebra $(L,\eta, \varphi)$ of
  $(A, \tau)$. The unique homotopy class of maps
$S^2\to X$ turns   $S^2$ into an $X$-manifold. This $X$-manifold can
be  obtained from $B_+$ and $B_-$ by gluing along the boundary.
Therefore the homomorphism $K\to K$,  $k\mapsto \tau(S^2)\, k $ is
the composition of $\tau(B_+):K\to L_1$ with $\tau(B_-):L_1\to K$.
Thus, $\tau(S^2)=\eta(1_L,1_L)$.

\subsection{Transformations}\label{Eqcat++-}  The following   lemmas  show   that
under the passage to the underlying  algebra, the direct sums,
rescalings, and transfers of   HQFTs correspond to the direct sums,
 rescalings, and transfers of algebras.

 \begin{lemma}\label{splitd}
If $(A^{{1}}, \tau^{{1}})$ and $(A^{{2}}, \tau^{{2}})$  are
$X$-HQFTs with underlying crossed $G$-algebras $L^{{1}}$ and $
L^{{2}}$, respectively,  then the underlying crossed $G$-algebra of
the  direct sum $(A^{{1}}, \tau^{{1}})\oplus (A^{{2}}, \tau^{{2}})$
 is $L^{{1}}\oplus
L^{{2}}$. Moreover, if the underlying crossed $G$-algebra $L$ of a
2-dimensional $X$-HQFT $(A,\tau)$ splits as a direct sum of two
crossed $G$-algebras $L=L^{{1}}\oplus L^{{2}}$, then $(A,\tau)$
  splits as a direct sum of two $X$-HQFTs with underlying
crossed $G$-algebras $L^{{1}}$ and $ L^{{2}}$.
\end{lemma}

\begin{proof} The first claim follows from the definitions. We prove
the second claim. The splitting   $L_\alpha=L^{{1}}_\alpha\oplus
L^{{2}}_\alpha$ for all $\alpha\in G$  induces a splitting
$A_M=A^{{1}}_M\oplus A^2_M$ for any connected $X$-curve $M$. For a
non-connected $X$-curve $M$, we define $A^1_M$, $ A^2_M$  via Axiom
(4) of an HQFT.  This yields
 a
natural embedding $i^k_M:A^k_M\to A_M$  and a natural projection
$p^k_M:A_M \to A^k_M  $,  where $k=1,2$. Observe now that given a
connected $2$-dimensional $X$-cobordism $(W,M,N)$, the homomorphism
$\tau(W): A_M \to A_N$ expands uniquely as   $ i^1_N\tau^1(W)p^1_M
+i^2_N\tau^2(W)p^2_M $ for some   $\tau^k(W)\in \Hom (A^k_M,
A^k_N)$,  where $k=1,2$. For  cobordisms
 of type $D_{--+}$, $C_{--}$,   $C_{-+}$, and $B_+$ this follows from the
splitting $L=L^1\oplus L^2$. For   other cobordisms,
   this is obtained by splitting them  into disks with $\leq 2$ holes  as in the proof of Lemma \ref{liftii}.
 It is easy to check that $(A^{{1}}, \tau^{{1}})$ and $(A^{{2}},
 \tau^{{2}})$ are $X$-HQFTs and $(A,\tau)$ is their direct sum.
\end{proof}

 \begin{lemma}\label{split+resc}
If  an  HQFT $(A', \tau')$ is obtained from a
2-dimensional HQFT $(A,\tau)$ by $k$-rescaling with $k\in K^*$, then the underlying crossed $G$-algebra of
$(A', \tau')$ is obtained from the underlying crossed $G$-algebra of
$(A, \tau)$ by  $k$-rescaling. \end{lemma}

\begin{proof} For a cobordism  $(W,M,N)$ of type
$D_{--+}$, $C_{--}$, and $C_{-+}$, the number $\chi(W)+b_0(M)-b_0(N)$ is equal to $-1+2-1=0$, $0+2-1=1$, and $0+1-1=0$, respectively.
\end{proof}

\begin{lemma}\label{trac}  Let $H\subset G$ be a subgroup   of finite
index, $ a\in H^{2}(H; K^*)$, and $k\in K^*$.   The underlying
crossed $\orlin$-algebra   of the $X$-HQFT $(A^{G,H, a}, \tau^{G,H,
a, k})$ is isomorphic to the transfer $L^{G,H, -a,k}$ of the crossed
$H$-algebra $L^{-a,k}$.
\end{lemma}

\begin{proof} By the previous lemma, it is enough to consider the case $k=1$.
Let $p:\widetilde X \to X$ be the covering corresponding to
$H\subset G$. Set $Y=\widetilde X/p^{-1}(x)$ and consider the
$Y$-HQFT $(A^a, \tau^a)$ determined by $ a\in H^{2}(H; K^*) =
H^{2}(\widetilde X; K^*)=H^{2}(Y; K^*)$. By definition, the HQFT
$(A'=A^{G,H, a}, \tau'=\tau^{G,H, a,1})$   is   the
 transfer of    $(A^a, \tau^a)$ to~$X$.  Let $L\cong L^{-a}$ be the underlying  crossed $H$-algebra
 of $(A^a, \tau^a)$.  We compute the underlying  crossed
 $G$-algebra $(L', \eta', \varphi')$
 of $(A',
\tau')$.   Pick a base point  $\widetilde x \in p^{-1}(x)$ of
$\widetilde X$. For each
  $y\in p^{-1}(x)$ fix a path   in $\widetilde X$ leading from $\widetilde x$  to $y$.
  Let $\omega_y\in G=\pi_1(X,x)$ be the homotopy class of the
   loop obtained by projecting this path to $X$. Then $\{\omega_y\}_{y\in p^{-1}(x)}$ is a set of representatives
   of the right $H$-cosets in~$G$. For an   $X$-curve $M=(M,
 g:M\to X)$, the vector space $A'_M $ is a direct sum of copies of
 $K$ numerated by
the lifts of $g$
  to  $\widetilde X$. If $M$ is connected and $g:M\to X$ represents
   $\alpha\in G$, then the lifts    of $g$ to $\widetilde X$ are numerated by $y\in p^{-1}(x)$ such that  $\omega_y\alpha
  \omega_y^{-1}\in H$. Thus, $L'_M$ is the direct sum of copies of
  $K$ numerated by such $y\in p^{-1}(x)$. Comparing with the definition of the transfer
  $(\widetilde L, \widetilde \eta, \widetilde \varphi)$ of $L$ (determined by the
   same set of representatives $\{\omega_y\}_{y }$),
  we obtain  that $L' =\widetilde L $ as $G$-graded vector
  spaces. Multiplication in $L'$ is defined using the maps $D_{--+}\to X$
  carrying the segments  $ty,tz\subset D_{--+}$ to the base point $x$.
A  lift  of such a map  to $\widetilde X$   carries these  segments
  to the same point of $p^{-1}(x)$ and induces multiplication in the
  corresponding copy of $K$. Therefore,  multiplication in  $L'$
  is a direct sum of multiplications in the copies of $K$ indexed by
  the same point of $p^{-1}(x)$. Hence $L' =\widetilde L $ as
  $G$-graded algebras. That $\eta'=\widetilde \eta$ is proven
  similarly. The definition of the transfer   implies that for any $\alpha\in G$, both
  $\varphi'_\alpha$ and $\widetilde \varphi_\alpha$ carry   $K=L_{\omega_j \beta
  \omega_j^{-1}}\subset L'_\beta=\widetilde L_\beta$ to $K=L_{\omega_{\alpha(j)} \alpha \beta \alpha^{-1}
   \omega_{\alpha(j)}^{-1}}\subset L'_{\alpha\beta\alpha^{-1}}=\widetilde L_{\alpha\beta\alpha^{-1}}$, where we use
    the  notation of Section
   \ref{transfcro}. The equality $\varphi'_\alpha=\widetilde
   \varphi_\alpha$ follows then from the uniqueness
   of the action of $G$ on $L$   satisfying the axioms   of a crossed $G$-algebra (Section \ref{exxama}).
 \end{proof}

  \section{The $G$-center of a biangular $G$-algebra}\label{thegcenterb}

  Let $B=\oplus_{\alpha \in G}\,
  B_\alpha$ be   a biangular $G$-algebra.  We compute    the crossed $G$-algebra underlying the
  HQFT $(A^\circ_B, \tau^\circ_B)$ derived from $B$.

  We begin with algebraic
  preliminaries. For   $\alpha\in G$,   define a  $K$-linear homomorphism
$\psi_\alpha:\rup \to \rup$ by
 $\psi_\alpha(a)=\sum_{i } p_i^{\alpha} a q_i^{\alpha}   $,
where $\sum_i p_i^{\alpha}\otimes q_i^{\alpha}=b_\alpha \in
\rup_{\alpha} \otimes \rup_{\alpha^{-1}}$ is the vector derived from
the inner product $\eta=\eta_B$ in Section \ref{BGa}. Here $i$ runs
over a finite set of indices~${J}_\alpha$
 and $p_i^{\alpha} \in \rup_{\alpha}, q_i^{\alpha} \in
\rup_{\alpha^{-1}}$. Clearly,  $\psi_\alpha (
  \rup_\beta)\subset  \rup_{\alpha \beta \alpha^{-1}}$ for all $\beta\in G$.
  By
(\ref{firstide}),  $\psi_\alpha(1_{ \rup})=1_{ \rup}$. The symmetry
of $\eta$ implies that
$$\psi_{\alpha^{-1}}(a)=
 \sum_{j\in {J}_{\alpha^{-1}}}  p_j^{\alpha^{-1}}a q_j^{\alpha^{-1}}
=\sum_{i\in {J}_\alpha} q_i^{\alpha} a p_i^{\alpha}\, .
 $$
 The following
lemma  exhibits the  main
  properties of the homomorphisms  $ \{\psi_\alpha  \}_\alpha$.

\begin{lemma}\label{firstbuiang} For any $\alpha, \beta \in \orlin$ and $ a,b\in \rup$,
\begin{equation}\label{1.3.a}  \eta (\psi_\alpha (a), b)= \eta(a, \psi_{\alpha^{-1}} (b)),
 \end{equation}
\begin{equation}\label{1.3.b}\psi_\alpha (a\psi_{\beta} (b))= \psi_{ \alpha}(a)
\,\psi_{\alpha\beta} (b),  \end{equation}
\begin{equation}\label{1.3.c}\psi_\alpha \, \psi_\beta (b)=   \psi_{\alpha
\beta} (b).  \end{equation} If   $b\in \rup_\beta$,
then for any $\alpha \in \orlin$,
\begin{equation}\label{1.3.d}\psi_{ \alpha \beta}(b)=\psi_\alpha (b)
 \end{equation} and  for any $\alpha \in \orlin$, $a \in
\rup$,
\begin{equation}\label{1.3.e}\psi_\alpha(a) \,b= b\, \psi_{\beta^{-1} \alpha} (a).
 \end{equation} For any  $\alpha,\beta \in \orlin$ and $c\in
\rup_{\alpha\beta \alpha^{-1} \beta^{-1}}$,
\begin{equation}\label{1.3.f}\Tr\, (\mu_c\, \psi_{\beta}:\rup_{\alpha}\to \rup_{\alpha})=
\Tr\, ( \psi_{\alpha^{-1}} \mu_c:\rup_{\beta}\to \rup_{\beta})\, ,
\end{equation} where  $\mu_c$  is  the homomorphism $B\to B, a\mapsto
  ca$.
\end{lemma}

 \begin{proof}  We first check   that   for any $\beta\in G$ and $b\in \rup_{\beta^{-1}}$,
\begin{equation}\label{1.3.h}\sum_{i\in {J}_\alpha} p_i^\alpha \otimes q_i^\alpha b= \sum_{j\in
{J}_{\beta \alpha}} b p_j^{\beta \alpha} \otimes  q_j^{\beta
\alpha}. \end{equation} Both sides belong to $\rup_\alpha \otimes
\rup_{\alpha^{-1}\beta^{-1}}$ and it suffices to prove that they
determine equal functionals on the dual space $\rup_{\alpha^{-1}}
\otimes \rup_{ \beta\alpha}$. Pick $x\in \rup_{\alpha^{-1}}$ and
$y\in \rup_{ \beta\alpha}$. By (\ref{secide}), the left-hand side of
(\ref{1.3.h}) evaluated on $x\otimes y$ gives
$$ \sum_{i } \eta(p_i^\alpha,x)\, \eta( q_i^\alpha b, y)=\sum_{i }
\eta(x,p_i^\alpha)\, \eta( q_i^\alpha, b
y)
 =\eta(x,by).$$
Note that the expression $\eta(ab,c)$ with $a,b,c\in \rup$  is
invariant under cyclic permutations of  $ a,b,c $; indeed,
$\eta(ab,c)=\eta (a,bc)=\eta(bc,a)$.  Therefore the right-hand side
of (\ref{1.3.h}) evaluated on $x\otimes y$ gives
$$ \sum_{j } \eta(b p_j^{\beta \alpha},x) \,\eta( q_j^{\beta \alpha},
y)= \sum_{j } \eta( xb,p_j^{\beta \alpha}) \, \eta( q_j^{\beta
\alpha}, y)=\eta(xb,y)=\eta(x,by).$$ This proves (\ref{1.3.h}).
Formula (\ref{1.3.h}) implies that for any $a\in \rup$, $b\in
\rup_{\beta^{-1}}$,
\begin{equation}\label{1.3.i} \psi_\alpha(a) b=\sum_{i\in {J}_\alpha} p_i^\alpha a q_i^\alpha b= \sum_{j\in {J}_{\beta
\alpha}} b p_j^{\beta  \alpha} a  q_j^{\beta  \alpha}=b \psi_{\beta
\alpha} (a) \, .\end{equation} This implies
  (\ref{1.3.e}).  By the cyclic symmetry,
$$ \eta (\psi_\alpha (a), b)=\eta(\sum_i p_i^{\alpha} a
q_i^{\alpha},b)=
\eta(a, \sum_i  q_i^{\alpha} b  p_i^{\alpha})=\eta(a,
\psi_{\alpha^{-1}} (b)).$$ This proves (\ref{1.3.a}). Using (\ref{1.3.h}), we obtain
$$\psi_\alpha (a\psi_{\beta} (b))= \sum_{i\in {J}_\alpha } p_i^\alpha a \psi_{\beta} (b) q_i^\alpha
= \sum_{i\in {J}_\alpha}
p_i^\alpha
a q_i^\alpha  \psi_{\alpha\beta} (b)
 =\psi_{ \alpha}(a)
\,\psi_{\alpha\beta} (b).$$ This proves   (\ref{1.3.b}).
Substituting $a=1_{ \rup}$ in (\ref{1.3.b}), we obtain
(\ref{1.3.c}).

Formula (\ref{1.3.h}) implies   that for any $b\in
\rup_{\beta^{-1}}$,
$$\psi_{\alpha^{-1}} (b)=\sum_{i\in {J}_\alpha} q_i^\alpha b p_i^\alpha =
\sum_{j\in {J}_{\beta \alpha }} q_j^{\beta  \alpha} b p_j^{\beta
\alpha}=\psi_{(\beta \alpha)^{-1} } (b)=\psi_{\alpha^{-1}\beta^{-1}}
(b).$$ This   is equivalent to (\ref{1.3.d}).   We now check (\ref{1.3.f}):
$$\Tr\, (\mu_c\, \psi_{\beta}:\rup_{\alpha}\to
\rup_{\alpha})=\sum_{i\in {J}_\alpha} \eta( \mu_c \psi_\beta
(p^\alpha_i) , q^\alpha_i)=\sum_{i\in {J}_\alpha, j\in {J}_\beta}
\eta( cp_j^\beta
 p^\alpha_i q_j^\beta, q^\alpha_i)$$
$$= \sum_{i\in {J}_\alpha, j\in {J}_\beta} \eta( q^\alpha_i cp_j^\beta
 p^\alpha_i , q_j^\beta  )
 = \sum_{  j\in {J}_\beta} \eta( \psi_{\alpha^{-1}} \mu_c (p_j^\beta
), q_j^\beta  )
 =
\Tr\, ( \psi_{\alpha^{-1}} \mu_c:\rup_{\beta}\to \rup_{\beta}).  $$
 \end{proof}

\begin{lemma}\label{firstbuiang+}    Set
$L=\oplus_{\alpha\in \orlin}\,  L_\alpha$ where
$L_\alpha=\psi_1(\rup_\alpha)\subset \rup_\alpha$.  Then
  $L$ is a subalgebra of ${ \rup}$ and $\psi_\alpha (L)=L$ for
all $\alpha\in \orlin$. The triple $(L, \eta\vert_L, \{ \psi_\alpha
\vert_L\}_{\alpha\in \orlin})$ is a semisimple crossed $\orlin$-algebra.
\end{lemma}

 \begin{proof} Formula (\ref{1.3.b}) with  $ \alpha=\beta=1$ shows that $L=\Im\,
\psi_1$ is a subalgebra of ${ \rup}$ with unit $1_{ \rup}=\psi_1(1_{
\rup}) $.  Formula (\ref{1.3.c})   implies that $\psi_1^2=\psi_1$.
By (\ref{1.3.a}), the projector $\psi_1:B\to B$ is self-adjoint with
respect to~$\eta=\eta_B$. Hence ${ \rup}$ splits as an orthogonal sum of
$L$ and $\Ker\, \psi_1$. Therefore   $\eta \vert_L$ is an inner
product on $L$.

  Formula (\ref{1.3.c}) with $\alpha=1$
shows that $\psi_\beta (\rup)\subset L$  for all $\beta \in \orlin$.
Set $\varphi_\beta=\psi_\beta \vert_L:L\to L$. Formula (\ref{1.3.c})
with $\alpha=\beta^{-1}$  implies that $\varphi_\beta$ and
$\varphi_{\beta^{-1}}$ are mutually inverse automorphisms of $L$.
That they preserve multiplication  follows from
  (\ref{1.3.b}) with $\beta=1$. By (\ref{1.3.c}) and (\ref{1.3.a}),   $\varphi_{\alpha\beta} = \varphi_\alpha  \,
\varphi_\beta  $ and   $\varphi_\alpha$ preserves $\eta\vert_L$.
Formula (\ref{1.3.d}) with $\alpha=1$ yields $\varphi_\beta
\vert_{L_\beta}=\id$. Formula (\ref{1.3.e}) implies that
$\varphi_\alpha(b) a=ab$ for   $a\in L_\alpha$, $b\in L$. Finally,
for any $\alpha, \beta \in \orlin$ and $c\in L_{\alpha\beta
\alpha^{-1} \beta^{-1}}$,
$$\Tr\, (\mu_c\, \varphi_{\beta}:L_{\alpha}\to L_{\alpha})=\Tr\,
(\mu_c\,
\psi_{\beta}:L_{\alpha}\to
L_{\alpha})=\Tr\, (\mu_c\, \psi_{\beta}:\rup_{\alpha}\to
\rup_{\alpha})$$
$$=\Tr\, ( \psi_{\alpha^{-1}} \mu_c:\rup_{\beta}\to \rup_{\beta})=\Tr\,
( \psi_{\alpha^{-1}} \mu_c:L_{\beta}\to L_{\beta})= \Tr\, (
\varphi_{\alpha^{-1}} \mu_c:L_{\beta}\to L_{\beta})\, ,$$ where the
second and fourth equalities follow from the inclusion $\psi_\beta
(\rup)\subset L$ and the third equality is (\ref{1.3.f}). Thus $L$
satisfies all axioms of  a crossed  $\orlin$-algebra.

The radical of a finite-dimensional $K$-algebra
$C$ can be defined as the annihilator of the symmetric bilinear form
 $C\times C\to K$ carrying a pair $(c_1,c_2)\in C\times C$ to the trace
  of the homomorphism $C\to C$, $ c\mapsto c_1c_2c$. By the definition of a
   biangular $G$-algebra, the radical of  $B_1 $ is trivial.
Hence  $\rup_1$ is a direct sum of matrix rings
over $K$. If $B_1$ is  a matrix ring,  $\rup_1=\Mat_n(K)$ with
$n\geq 1$, then $\eta=\eta_B$ carries  a pair  $(a_1, a_2)\in \rup_1\times B_1 $   to $n\,
\Tr(a_1a_2)$, where $\Tr$ is the standard matrix trace on $\Mat_n(K)$.
 The  vector $b_1\in \rup_1\otimes \rup_1$ determined by  $\eta \vert_{B_1}$    is equal to $n^{-1}
\sum_{i,j=1}^n e_{i,j} \otimes e_{j,i}$ where $e_{i,j}$ is the
elementary $(n\times n)$-matrix whose $(i,j)$-term is $1$ and all
other terms are zero. The homomorphism $\psi_1:\rup_1\to \rup_1$
carries any $a\in \rup_1$ to $n^{-1} \sum_{i,j=1}^n e_{i,j} a
e_{j,i}=n^{-1} \Tr(a) E_n$ where $E_n$ is the unit $(n\times
n)$-matrix. Thus, $\psi_1$ is    a projection of $\Mat_n(K)$ onto
its 1-dimensional center. If $\rup_1$ is a direct sum of $N$ matrix
rings then $\psi_1$ is a projection of $\rup_1$ onto its
 center $K^N$. Hence $L_1=\psi_1(\rup_1)=K^N$. Thus   $L $ is semisimple.
  \end{proof}

The crossed   $\orlin$-algebra  $L \subset B$ is called
 the {\it $G$-center}  of $\rup$.

 \begin{lemma}\label{2dimunderalgg}    The  crossed
 $\orlin$-algebra $(L'=\oplus_{\alpha\in G}\, L'_\alpha, \eta', \{\varphi'_\alpha\}_\alpha)$ underlying the HQFT $(A^\circ_B, \tau^\circ_B)$ is
 isomorphic to the $G$-center $L$ of $B$.
\end{lemma}

 \begin{proof}   To compute $L'_\alpha$, we represent $\alpha\in G$ by a loop
 $ g: S^1 \to X =K(G,1)$, where
 $S^1=\{z\in \CC\, \vert\, \vert z\vert=1\}$ with clockwise
 orientation and base point $s=1$. The $X$-curve $M=(S^1, g)$ has a
 canonical trivialization $t$ formed by  one vertex   $s $,
 one edge~$e$, and the $G$-system assigning $\alpha$ to the edge $e$   oriented clockwise. By Section \ref{statesu++},
 $L'_\alpha=A^\circ_M$  is the image of the   projector
 $P_M:A_{(M,t)}\to A_{(M,t)}$, where $A_{(M,t)}=B_\alpha$. Recall that $P_M=\tau(W,g_1:W\to X)$, where
 $W=S^1\times [0,1]$ is viewed as a
cobordism between two copies of~$(M,t)$ and $g_1$ is the composition
of the projection $W\to
 S^1$ with $g:S^1\to X$.   Let $T$ be the CW-decomposition of $W$
 formed by
 the vertices $(s,0)$, $(s, 1)$, the edges $e_0=e\times \{0\}$,
 $e_1=e\times \{1\}$,  $e_2=\{s\}\times [0,1]$, and the face $(S^1-\{s\})\times (0,1)$.
 We orient $e_0, e_1$   clockwise and $e_2$ from $(s,0)  $ to $(s,1)$. The map $g_1$ is presented  by the
 $G$-system on $T$ assigning $\alpha$ to $e_0, e_1$   and $1\in G$
 to $e_2$. It is clear from the definitions that
 $P_M=\tau(W,g_1) :B_\alpha\to B_\alpha$ carries any $a\in B_\alpha$
  to $\sum_{i } p^1_i a q^1_i$, where $\sum_i p^1_i \otimes q^1_i=b_1\in B_1\otimes B_1$. Hence $P_M=\psi_1$ and $L'_\alpha=\Im\,
 \psi_1=L_\alpha$.

 Replacing in this construction the labels of
   $e_1, e_2$  by  $\beta\alpha \beta^{-1},   \beta^{-1}\in G$, respectively,  we obtain a  $G$-system on $T$
 representing a map $g_\beta:W\to
 X$. The $X$-cobordism $(W, g_\beta)$ is   the
 annulus
  $C_{-+}(\alpha; \beta^{-1})$ used to
  define $\varphi'_\beta$ on $L'_\alpha=L_\alpha$ in Section \ref{algebral}.  By definition,
   $$\varphi'_\beta \vert_{L_\alpha}=\tau^\circ  (W, g_\beta) =\tau   (W, g_\beta)\vert_{L_\alpha}:L_\alpha\to
   L_{\beta\alpha\beta^{-1}}\, .$$   The homomorphism
   $\tau   (W, g_\beta):B_\alpha\to B_{\beta \alpha \beta^{-1}}$
   carries any $a\in B_\alpha$ to
 $ \sum_{i } p^\beta_i a q^\beta_i=\psi_\beta (a)
 $, where   $\sum_i p^\beta_i \otimes q^\beta_i=b_\beta\in B_\beta\otimes B_{\beta^{-1}}$.
 Hence $\varphi'_\beta =\psi_\beta\vert_{L }  $.

Inverting the orientation of   $S^1\times \{1\}\subset \partial W$,
we can view   $(W, g_1)$     as an $X$-cobordism between $(M, t)
\amalg (-M,t)$ and $\emptyset$. In the notation of Section
\ref{algebral}, this   cobordism  is $C_{--}(\alpha; 1)$. By
definition, $\eta'_\alpha=\tau^\circ (W, g_1):L_\alpha\otimes
L_{\alpha^{-1}}\to K$ is the restriction of   $\tau (W,
g_1):B_\alpha\otimes B_{\alpha^{-1}}\to K$. For   $a\in B_\alpha$,
$b\in B_{\alpha^{-1}}$,
 $$\tau
(W, g_1)(a\otimes b)=\sum_i \eta_B(ap^1_ibq^1_i,
1_B)=\eta_B(a, \psi_1(b))\, .$$ Since $\psi_1\vert_L=\id$, we    conclude that
 $\eta'=\eta_B \vert_{L\otimes L}$.
  \end{proof}

  By the proof of Lemma
\ref{firstbuiang+}, the $K$-algebra $\rup_1$ splits as   a direct sum of
matrix algebras over $K$ and     $\psi_1: B_1\to   B_1$ maps each of these matrix
algebras onto its center. The  set $I=\Bas(L)$ of the basic idempotents of~$L_1=\psi_1(B_1)$ is  
the set of unit elements of these matrix algebras. More
precisely, each basic idempotent $i\in I$ is the unit element of a
direct summand $\Mat_{n_i} (K)$ of $B_1$, where $n_i\geq 1$.  Let $I_0\subset I$ be the fixed point set
   of the action of $G$ on $I$. For  $i\in I_0$, the  subalgebra $iL$ of $L$ is a crossed $G$-algebra with
   unique basic idempotent  $i$. The pair $(iL,i)$ determines a cohomology class $\nabla_i\in H^2(G;K^*)$ by
   Lemma \ref{lelecrocro2}.

  \begin{theor}\label{firstbuiang+opik}    For any  closed connected oriented surface $W$ endowed with a
  map $\widetilde g:W\to X=K(G,1)$ such that   $g=\widetilde g_\#: \pi_1(W)\to G$ is an epimorphism,
\begin{equation}\label{bianan}\tau_B(W,\widetilde g)=\sum_{i\in I_0} \, n_i^{\chi(W)} \,\,
g^*(\nabla_i) ([-W]))\, .\end{equation}
 \end{theor}

 \begin{proof}  By the previous lemma, $L$ is the underlying crossed $G$-algebra of the HQFT $(A^\circ_B, \tau^\circ_B)$.  
As we know, $L$ splits as a direct sum of simple crossed $G$-algebras $L^{\kappa}$ numerated by the orbits $\kappa$ of the action of $G$ on $I=\Bas(L)$.
  By Lemma \ref{splitd}, the HQFT $(A^\circ_B, \tau^\circ_B)$ splits as a direct sum of HQFTs
 $(A^{\kappa}, \tau^{\kappa})$ such that $L^{\kappa}$ is the underlying crossed $G$-algebra of $(A^{\kappa}, \tau^{\kappa})$ for all orbits $\kappa\subset I$.
 Then $\tau_B(W,\widetilde g)=\tau^\circ_B(W,\widetilde g)=\sum_\kappa \tau^{\kappa}(W,\widetilde g)$.

We claim that  $\tau^{\kappa}(W,\widetilde g)=0$ for any   orbit $\kappa\subset I $ with
at least two  elements. Indeed, the proof of Lemma \ref{lelecrocro2}
shows that
 $L^{\kappa}$ is isomorphic to  $L^{G, H, \theta,k}$, where $H$ is a subgroup  of $G$ of index  $ \vert m\vert \geq
 2$,
  $\theta\in H^2(H;K^*)$, and $k\in K^*$. By Lemma \ref{trac}, $L^{G, H, \theta,k}$
   underlies the   $X$-HQFT   $(A^{G,H,-\theta}, \tau^{G,H,-\theta, k})$.
 By Lemma \ref{liftii}, this $X$-HQFT is  isomorphic to $(A^{\kappa}, \tau^{\kappa})$.
 Thus,
 $(A^{\kappa}, \tau^{\kappa})$ is isomorphic to the transfer of a $K(H,1)$-HQFT.
The assumption
  $g( \pi_1(W))= G$ implies that the map $\widetilde g:W\to X$ does not lift to non-trivial coverings of
  $X$. Hence, $\tau^{\kappa}(W,\widetilde g)=0$.

A one-element orbit  $\kappa\subset I$ is  just  $\{i\}$ for  $i\in I_0$.
By Lemma \ref{lelecrocro2},  $L^{\kappa}=iL$  is
  obtained from $L^{\nabla_i}$ by $k_i$-rescaling, where $k_i=\eta_B(i,i) =n_i^2$.
Hence $L^{\kappa}$     underlies    the $X$-HQFT obtained
  from
   $(A^{-\nabla_i}, \tau^{-\nabla_i)}$ by $k_i$-rescaling.
    By Lemma \ref{liftii}, the latter HQFT is   isomorphic to   $(A^{\kappa}, \tau^{\kappa})$.
    We have therefore $$\tau^{\kappa} (W,\widetilde g)=\tau^{\{i\}} (W,\widetilde g)=k_i^{\chi(W)/2} \,
g^*(-\nabla_i)([W])=n_i^{\chi(W)} \,
g^*(\nabla_i)([-W])\, .$$ Hence
$$\tau_B(W, \widetilde g)
=\sum_{i\in I_0}\, \tau^{\{i\}} (W,\widetilde g)=
\sum_{i\in I_0} \, n_i^{\chi(W)} \,\,
g^*(\nabla_i) ([-W]))\, .$$
  \end{proof}

Note that for $W=S^2$, we  have $G=\{1\}$  and  $\nabla_i=0$ for all $i\in I_0=I$. Formula (\ref{bianan}) gives in this case $\tau_B(W, \widetilde g)=\sum_{i\in I} n_i^2=\dim B_1$.

We now  establish an analogue of Formula (\ref{bianan}) for a compact connected oriented surface $W$ of positive genus whose
  boundary  has $m\geq 1$ components.
We  use the notation introduced in the first two paragraphs of Section \ref{twolemmas}  assuming that 
the given homomorphism $g:\pi=\pi_1(W,w)\to G$
is surjective and that its restriction to the free group $H=\langle x_1,...,x_m\rangle \subset \pi$ is injective.  
As above, 
 $L$ is  the $G$-center of $B$  and  $I_0$ is  the fixed point set
   of the action of $G$ on $ \Bas(L)$. Suppose that for each $k=1,..., m$, we are given  a vector 
   $\overline \gamma_k \in  L_{g(x_k)}=A^\circ_{(C_k, \widetilde g)}$ where $A^\circ=A^\circ_B$ and $C_k$ is the $k$-th component of $\partial W$. Let $I_0^\gamma \subset I_0$ consist  of
 all $i\in I_0$ such that $i \overline \gamma_k\neq 0$ for all $k=1,...,m$
(note that $i \overline \gamma_k$ is the projection of  $ \overline \gamma_k$ to the direct summand $iL_{g(x_k)}$ of $L_{g(x_k)}$). For each $i\in I_0^\gamma$, we   define
 a cohomology class $\nabla^\gamma_i\in H^2(G, g(H);K^*)$ as $\nabla_i$ in
 Lemma \ref{lelecrocro2}  choosing the generating vectors $s_\alpha\in iL_\alpha\cong K$    as follows.
 Set $s_{g(x_k)}=i\overline \gamma_k$ for all $k$. The properties of the  crossed $G$-algebra $iL$  imply that
there is a unique vector   $s_{g(x_k)^{-1}} \in iL_{g(x_k)^{-1}}$ whose left and right products
  with $s_{g(x_k)}$  are equal to $i\in iL_1$.   Each  $\alpha\in g(H) $ expands as a product of the generators 
  $g(x_k)^{\pm 1}\in g(H)$, and we let  $s_\alpha$  be the product of the corresponding vectors $s_{g(x_k)^{\pm 1}}$ 
  (note that $s_\alpha\neq 0$).  For $\alpha\in G-g(H)$, the generating vector  $s_\alpha\in iL_\alpha $ is chosen in an arbitrary way.
The equality $s_\alpha s_\beta=\nabla_{\alpha, \beta} s_{\alpha \beta}$ for $\alpha, \beta \in G$ defines a $K^*$-valued 2-cocycle $\nabla=\{\nabla_{\alpha, \beta}\}_{\alpha, \beta\in G}$  whose restriction to $g(H)$ 
is trivial. This cocycle represents  $\nabla^\gamma_i\in H^2(G, g(H);K^*)$.

 \begin{theor}\label{lik} Set  $\overline\gamma=\otimes_{k=1}^m \overline \gamma_k \in A^\circ_{(\partial W, \widetilde g)}$. Then
  \begin{equation}\label{binan++}\tau^\circ_B(-W, \widetilde g) (\overline\gamma)
=\sum_{i\in I_0^\gamma} \, n_i^{\chi(W)+m}  \,\,
g^*(\nabla^\gamma_i) ([W,\partial W]))\,  ,\end{equation}
where we view $[W,\partial W]$ as an element of $H_2(\pi, H;\ZZ)$, cf.\  Lemma \ref{firstkey1++}. 
\end{theor}

 \begin{proof}  The first part of the proof of Theorem \ref{firstbuiang+opik} applies directly and gives
\begin{equation}\label{kol} 
\tau_B(-W, \widetilde g) (\overline\gamma) 
=\sum_{i\in I_0}\, \tau^{\{i\}} (-W,\widetilde g) (\otimes_{k=1}^m i \overline \gamma_k)
\,.\end{equation}
For $i \in I_0 - I_0^\gamma$,   the $i$-th term on the right-hand side   is equal to $0$. It remains to show that
for every $i\in I_0^\gamma $,  
 \begin{equation}\label{binan++-}\tau^{\{i\}} (-W,\widetilde g) (\otimes_{k=1}^m i \overline \gamma_k)  = n_i^{\chi(W)+m} \,\,
g^*(\nabla^\gamma_i) ([W,\partial W]))\, .\end{equation}

We can assume that $X=K(G,1)$ contains a CW-subspace $Y=K(g(H),1)$ with the same base point $x$. Let $p:[0,1]\to S^1$,
$\{   u_\alpha:S^1\to X \}_{\alpha\in G}$,  $\Delta\subset \RR^3$,  and $\{f_{\alpha, \beta}:\Delta\to X\}_{\alpha, \beta\in G}$ 
be   as in the proof of Theorem \ref{liftii+98}.
 We can assume that $ u_\alpha(S^1)\subset Y$ for all $\alpha \in g(H)$ and $f_{\alpha,\beta}(\Delta)\subset Y$ for all $\alpha, \beta\in g(H)$.  Pick a singular $K^*$-valued 
2-cocycle $\Theta$ on $X$ such that $\Theta(f_{\alpha,\beta})=  \nabla_{\alpha, \beta}^{-1}$ for all $\alpha, \beta\in G$, where  $\nabla=\{\nabla_{\alpha, \beta}\}$ is the 2-cocycle on $G$    determined by the   vectors $ s_\alpha\in iL_\alpha $ chosen   above. Modifying if necessary $\Theta$ by a coboundary, we can  assume additionally  that $\Theta$ annihilates all singular chains in $Y$ and defines thus a  cohomology class   $
[\Theta]\in H^2(X,Y;K^*)$. This class  is equal to $-\nabla^\gamma_i$ under the identification $H^2(X,Y;K^*)=H^2(G, g(H); K^*)$.  Consider the   $X$-HQFT  $(A^{\Theta}, \tau^{\Theta})$ and its $k_i$-rescaling $(A^{\Theta}, \tau^{\Theta, k_i})$ where $k_i=n_i^2$. Let $R$ and $R'$ be the crossed $G$-algebras underlying these two $X$-HQFTs, respectively. For $\alpha\in G$, consider the generating vector   $$p_\alpha=\langle p\rangle \in R_\alpha=R'_\alpha=A^{\Theta}_{S^1,\alpha}\cong K $$ represented by the singular 1-simplex $p:[0,1]\to S^1$ viewed as  a fundamental cycle of the $X$-curve $(S^1,u_\alpha)$.   By the   proof of Theorem \ref{liftii+98}, $p_\alpha p_\beta= \nabla_{\alpha, \beta} \, p_{\alpha\beta}$ for all $\alpha, \beta\in G$. Moreover,  
 the formula $  s_\alpha  \mapsto p_\alpha$   defines   an isomorphism of crossed $\orlin$-algebras   $  iL\to R'$. By Lemma \ref{liftii},  this isomorphism lifts to
an isomorphism of $X$-HQFTs $(A^{\{i\}}, \tau^{\{i\}})\to (A^{\Theta}, \tau^{\Theta,k_i})$.
Therefore $$\tau^{\{i\}} (-W,\widetilde g) (\otimes_{k=1}^m i \overline \gamma_k)  = 
\tau^{\{i\}} (-W,\widetilde g) (\otimes_{k=1}^m   s_{g(x_k)} ) 
$$
$$= 
\tau^{\Theta,k_i} (-W,\widetilde g) (\otimes_{k=1}^m p_{g(x_k)} )
=n_i^{\chi(W)+m} \tau^{\Theta} (-W,\widetilde g) (\otimes_{k=1}^m p_{g(x_k)} ) \, .$$
We can choose $\widetilde g:W\to X$ in its homotopy class so that the restriction of $\widetilde g$ to the $k$-th component $C_k$ of $\partial W$ is the composition of an orientation preserving  
pointed homeomorphism   $f_k:C_k\to S^1$ with $ u_{g(x_k)}:S^1\to X$ for all $k=1,..., m$. Then $p_{g(x_k)}=\langle f_k^{-1}p \rangle \in    A^\Theta_{(C_k, \widetilde g)}$ is represented by the singular 1-simplex $ f_k^{-1}p:[0,1]\to C_k$ viewed as a fundamental cycle in $C_k$. By the definition of $\tau^{\Theta}$,    $$
\tau^{\Theta} (-W,\widetilde g) (\otimes_{k=1}^m \langle f_k^{-1}p \rangle )= \widetilde g^* ( [\Theta])( -[W, \partial W] )= g^*(\nabla^\gamma_i) ([W,\partial W]))\,  .$$
Combining these formulas   we obtain (\ref{binan++-}).
\end{proof}

  \section{Proof of Lemma \ref{firstkey1++}}\label{Tcoatb---+}

\subsection{Preliminaries}
Recall the notation of  Lemma \ref{firstkey1++}:  $q:G'\to G$ is a
group epimorphism with finite kernel $\Gamma$ and  $B=K[G']$ is the
associated  biangular $G$-algebra.  By Lemma \ref{firstbuiang+}, the
$G$-center  $(L=\psi_1(B), \eta_B\vert_L, \{ \psi_\alpha
\vert_L\}_{\alpha\in \orlin})$ of $B$  is a semisimple crossed
$\orlin$-algebra. Let $\ol =\mid(L)
  $ be the $G$-set    of its basic idempotents.
As we know, each $i\in I$ is the unit element of a
direct summand $\Mat_{n_i} (K)$ of $B_1$, where $n_i\geq 1$. Let $  \rho_i^+ :B_1\to \Mat_{n_i} (K)$ be the projection
onto this summand. Clearly,   $  \rho_i^+(i)= E_{n_i}$, where $E_n$
is the   unit $n\times n$ matrix. The homomorphism  $\rho_i^+$ determines a representation
 $\rho_i=  \rho_i^+ \vert_{\Gamma}:\Gamma \to  GL_{n_i} (K)$ of $\Gamma$ and  is recovered from
 $\rho_i$ as its linear extension.     We can
  describe $\rho_i$ as the unique (up to equivalence)
irreducible  representation of $\Gamma$ such that $\rho_i^+ (i)\neq
0$. It is clear that the mapping $I\to \Irr(\Gamma)$, $i\mapsto \rho_i$
is a bijection. We   show that it is $G$-equivariant.
 By definition of the action of $G$ on $I$, for   $\alpha\in G$ and $i \in
   I$,
$$\alpha i =\psi_\alpha(i)=
\vert  \Gamma \vert^{-1} \sum_{a\in q^{-1}(\alpha)}  \, a \,i \,
{a}^{-1}\, .$$ All the summands on the right-hand side are equal   because $i$ lies in the center of~$L_1$. Therefore $\alpha i
={a} i {a}^{-1}$ for any $a\in q^{-1}(\alpha)$. By Section
\ref{e21},
 $  {a \rho}_i (b)=  \rho_i({a}^{-1} b
{a})$ for all $b\in B_1$. Hence
 $ {a \rho}_i (\alpha i)  =
\rho_i(i)\neq 0 $. Therefore  $a \rho_i=\rho_{\alpha i}$ for all
$a\in q^{-1}(\alpha)$, i.e.,  the bijection $I\to \Irr(\Gamma)$,
$i\mapsto \rho_i$ is $G$-equivariant. We identify $I$ with
$\Irr(\Gamma)$ along this bijection.   Though we shall not need it,
note that each idempotent $i\in \ol$ can be computed from the
character $\chi_i$ of $\rho_i$  by   $i=\vert  \Gamma \vert^{-1} \,
\chi_i(1)  \,\sum_{h\in \Gamma } \chi_i (h) \, {h^{-1}}
 $, see, for instance,  \cite{Co}, Chapter 2,
Theorem 5.

In the sequel,   $I_0=I_0(q)$ is  the fixed point set of the action of $G$ on
$I=\Irr(\Gamma)$ and    $(A^\circ=A^\circ_B,
\tau^\circ=\tau^\circ_B)$ is the 2-dimensional HQFT with target
$X=K(G,1)$ determined by  $B$.

\subsection{The case $\partial W=\emptyset$}\label{ppo}  For a closed connected oriented
surface $W$, Lemma \ref{firstkey1++}  is equivalent   to the
following claim: for any pointed map $ \widetilde g:W\to X$ inducing
an epimorphism $g :\pi_1(W)\to G$,
\begin{equation}\label{t3-++-34vbvbvb}\tau^\circ( W,  \widetilde g)
  =
\sum_{i\in I_0  } \, n_i^{\chi(W)} \, g^*(\zeta^{\rho_i
})  ([-W ])   \,  .\end{equation}  To deduce
this formula  from (\ref{bianan}) it is enough to show that  $ \nabla_i=\zeta^{\rho_i}$ for all $i\in I_0$.  Pick $i\in I_0$ and 
set $\rho=\rho_i$ and $n=n_i=\dim\,
\rho$.  Fix
 the lifts   $\{{\widetilde \alpha}\in q^{-1}(\alpha)\}_{ \alpha\in
G}$ and the   matrices $\{M_\alpha\}_\alpha\in GL_n(K)$ as in
Section \ref{e22}. This data determines a $K^*$-valued 2-cocycle $
\{\zeta_{\alpha,\beta}\}_{\alpha,\beta}$ representing $\zeta^\rho $.
The proof of
 Lemma \ref{l1} yields an extension of $\rho:\Gamma\to GL_n(K)  $ to a
 map
 $\overline \rho:G'  \to GL_n(K)  $. Let $\overline\rho^+:B=K[G']\to \Mat_n(K) $ be the induced  $K$-linear
 map.
Then $\overline\rho^+\vert_{B_1}= \rho^+$ and in particular
$\overline\rho^+ (i)=  \rho^+ (i)=E_n$.
  Since
$\overline\rho^+  (B_1)=   \Mat_{n} (K)$, for each $\alpha \in G $,
there is $d_\alpha \in B_1$ such that  $\overline\rho^+   (d_\alpha
) =M_\alpha^{-1} = (\overline\rho^+   (
  {\widetilde \alpha}))^{-1}$.
  Set $$s_\alpha=i\, \psi_1( d_\alpha \, {
{\widetilde \alpha}} )  =   \vert \Gamma \vert^{-1} \sum_{a\in
\Gamma}\, i\,
   a \,d_\alpha \,{\widetilde \alpha} \,a^{-1}  \in i \, L_\alpha  \subset B_\alpha \, .$$
By (\ref{eeq2.8}) and the identity $\zeta_{\alpha,1}= \zeta_{1,
\alpha }=1$, we have $\overline\rho^+ (ab)=\overline\rho^+ (a)\,
\overline\rho^+ (b)$  for $a,b\in B$ provided $a\in B_1$ or $b\in
B_1$.   Therefore
$$
\overline\rho^+ (s_\alpha)= \vert  \Gamma \vert^{-1} \sum_{a\in
\Gamma}
  \overline\rho^+  (i\,  a \,d_\alpha \,{\widetilde \alpha} \,a^{-1})
 = \vert
\Gamma \vert^{-1} \sum_{a\in \Gamma} \overline\rho^+   (i )\,
\overline\rho^+ ( a) \, \overline\rho^+  (d_\alpha )\, \overline\rho^+  (
 {\widetilde \alpha})\,  \overline\rho^+   ({ a}^{-1})$$
$$ = \vert  \Gamma \vert^{-1}
  \sum_{a\in \Gamma} \overline\rho^+
(  a)\,  \overline\rho^+   (a^{-1}) = \vert  \Gamma \vert^{-1}
  \sum_{a\in \Gamma}  \rho
(  a)\,   \rho   (a^{-1}) =  E_n \, .$$   Thus, $s_\alpha $ is a
non-zero vector in $iL_\alpha\cong K$. We take $\widetilde 1=1\in
G'$ and $d_1=1\in B_1$, so that  $s_1=i$.  The cohomology class
$\nabla_i$
    is represented by the normalized
$K^*$-valued 2-cocycle $ \{\nabla_{\alpha,\beta}\}_{\alpha,\beta\in
G}$   defined from   $s_\alpha s_\beta =\nabla_{\alpha,\beta} \,
s_{\alpha \beta}$. Formula (\ref{eeq2.8}) implies that
$$ \nabla_{\alpha,\beta} \, E_n=  \overline\rho^+  ( \nabla_{\alpha,\beta} \, s_{\alpha \beta})= \overline\rho^+  (s_\alpha s_\beta)= \zeta_{\alpha, \beta}
\,\overline\rho^+  (s_{\alpha })\, \,\overline\rho^+  (s_{ \beta})=
\zeta_{\alpha, \beta} \, E_n\, .$$ Hence,
$\nabla_{\alpha,\beta}=\zeta_{\alpha, \beta} $ for all
$\alpha,\beta\in G_\rho $.

\subsection{The case $\partial W\neq \emptyset$}  Pick $k\in \{1,...,m\}$. 
Then $\gamma_k\in q^{-1}(g(x_k))\subset G'$ is one of the generating vectors of $B_{g(x_k)}$ and
$\psi_1( \gamma_k)$ is   its projection   to $L_{ g(x_k)}$. 
 Set $$\overline \gamma_k=\sum_{i\in I_0} \,i\, n_i \,  \psi_1( \gamma_k)\in L_{ g(x_k)}\, .$$ 
 We verify first  that $i \overline \gamma_k=0$ if and only if $t_{\rho_i} (\gamma_k)=0$.
Fix $i\in I_0$ and
 set $\rho=\rho_i$ and $n=n_i=\dim\,
\rho$.  Fix
 the lifts   $\{{\widetilde \alpha}\in q^{-1}(\alpha)\}_{ \alpha\in
G}$ and the   matrices $\{M_\alpha\}_\alpha\in GL_n(K)$ as in
Section \ref{e23-9}. In particular, for $\alpha=g( x_k)$ with $k=1,...,m$, we have $\widetilde {\alpha}=\gamma_k$
 and $\Tr\, M_{\alpha}= t_\rho(\gamma_k)\in \{0,1\}$.  
We extend  $\rho:\Gamma\to GL_n(K)  $ to    $\overline\rho^+:B=K[G']\to \Mat_n(K) $ as above. Observe that 
$i\overline \gamma_k=  i \,n\,  \psi_1(\gamma_k)\in iL_{g(x_k)}$. 
 A computation similar to the one in Section \ref{ppo} shows that
$$\Tr\,  \overline \rho^+(i\overline \gamma_k)= n \Tr\, \overline \rho^+ (i \psi_1(\gamma_k))=
 n \Tr\, \overline \rho^+ (\gamma_k)=n \Tr\, M_{g(x_k)}=n t_\rho(\gamma_k)\, .$$
Consider the   generating vectors 
$\{s_\alpha\in iL_\alpha\cong K\}_{\alpha \in G}$ defined in Section \ref{ppo}.  Comparing the traces of $\overline \rho^+(i\overline \gamma_k)$ and $
\overline \rho^+(s_{g(x_k)})=E_n$, we obtain that either    $i\overline \gamma_k=0$   and  
 $ t_{\rho} (\gamma_k)=0$
or $i\overline \gamma_k=s_{g(x_k)}$ and   $ t_{\rho} (\gamma_k)=1$. 

If $ i\overline \gamma_k\neq 0$ for all $k=1,...,m$, then the   vectors 
$\{s_\alpha \}_{\alpha\in G}$   satisfy the requirements  needed in the definition  $\nabla^\gamma_i$ 
given before   Theorem \ref{lik}. Indeed,  the   computation
 at the end of Section \ref{ppo} shows that the  2-cocycle $
\{\nabla_{\alpha,\beta}\}_{\alpha,\beta\in G}$  defined from   $s_\alpha s_\beta =\nabla_{\alpha,\beta} \,
s_{\alpha \beta}$ is equal to  
the   2-cocycle $
\{\zeta_{\alpha,\beta}\}_{\alpha,\beta\in G}$ derived in  Section \ref{e23-9} from the lifts 
  $\{{\widetilde \alpha}\}_{ \alpha\in
G}$. In particular, $\nabla_{\alpha,\beta}=\zeta_{\alpha,\beta}=1$ for all $\alpha, \beta\in g(H)$. Therefore $s_\alpha s_\beta = 
s_{\alpha \beta}$ for all $\alpha, \beta\in g(H)$. This and the equality  
$s_{g(x_k)}=i\overline \gamma_k$ for all $k$  are precisely the conditions used in the definition of $\nabla^\gamma_i$. 
We conclude that
 $\nabla^\gamma_i=\zeta^{\rho_i, \gamma} \in H^2(G, g(H);K^*)$ and  rewrite Formula (\ref{binan++-}) as 
 $$\tau^{\{i\}} (-W,\widetilde g) (\otimes_{k=1}^m i \overline \gamma_k)  = n_i^{\chi(W)+m} \,\,
g^*(\zeta^{\rho_i, \gamma}) ([W,\partial W]))\, .$$
Dividing both sides by $n_i^m$, we obtain that
$$\tau^{\{i\}} (-W,\widetilde g) (\otimes_{k=1}^m i \psi_1( \gamma_k))  = n_i^{\chi(W)} \,\,
g^*(\zeta^{\rho_i, \gamma}) ([W,\partial W]))\, .$$
Summing up over all $i\in I_0$ such that $i\overline \gamma_k\neq 0$ for all $k$ and using (\ref{kol}) (with $\overline \gamma_k$ 
replaced by $\psi_1( \gamma_k)$),  we obtain 
$$\tau_B (-W,\widetilde g) ([\gamma])  = \sum_{i\in I_0, \,  t_{\rho_i} (\gamma_k)=1\, {\rm {for\, all}}\, k} n_i^{\chi(W)} \,\,
g^*(\zeta^{\rho_i, \gamma}) ([W,\partial W]))\, .$$
This formula is equivalent to the one  claimed by the lemma.

\section{The case of a trivial
bundle}\label{Tcoatb}

For trivial   bundles, Theorem  \ref{t2} simplifies but remains
non-trivial. We begin with a computation of  the cohomology classes
$\zeta^{\rho, \gamma}$ for direct products.

\begin{lemma}\label{l2-j} Let $G, \Gamma$ be  groups, $G'=G\times \Gamma$ and
$q:G'\to G$ be the projection.

(a) The action of $G$ on  $\Irr (\Gamma)$ determined by $q$  is
trivial so that $I_0(q) =\Irr (\Gamma)$.

(b)  Let $\gamma=\{\gamma_k\}_{ k }$ be a finite subset of $G'$,
where $\gamma_k=(\alpha_k\in G,y_k\in \Gamma)$  and
    $\{\alpha_k\}_k $ generate a free group $H\subset G$ of rank
    $\vert \gamma\vert $. Let $$\partial: H_2(G,H;\ZZ)\to
    H_1(H;\ZZ)=\oplus_{k}\,  \, \ZZ\cdot [\alpha_k]$$
    be the  boundary homomorphism, where $[\alpha_k]\in H_1(H;\ZZ)$ is the
      class of~$\alpha_k$.
For any  $\rho\in \Irr  (\Gamma)$ and any  $Z\in H_2(G,H;\ZZ)$ such
that $\partial Z= \sum_k [\alpha_k] $,
\begin{equation}\label{secondid}
  \zeta^{\rho,  \gamma} (Z)\, \prod_{k} t_\rho (\gamma_k) =\prod_{k}
    \,  \Tr\, \rho(y_k) \, .\end{equation}
\end{lemma}

\begin{proof} Claim (a) is obvious. We prove (b).
We first compute the function $t_\rho:G'\to \{0,1\} $. Let
$a=(\alpha ,y)  \in G'$ with $\alpha\in G$ and $y\in \Gamma$. Then
$a\rho  =y\rho=M^{-1} \rho M $ for $M=\rho (y)$. Thus, $t_\rho(a)=0$
if and only if $\Tr \, \rho (y)= 0$. To  prove  (\ref{secondid}) it
suffices to consider the case where
      $\Tr\, \rho(y_k)\neq 0$  and $t_\rho(\gamma_k)=1$ for all $k$.
Set $n=\dim \, \rho$.    Define a lift $\alpha\mapsto \widetilde
\alpha'$ of elements  of $ G$ to $G'$ by $\widetilde
\alpha'=(\alpha, 1)$ for $\alpha\in G-H$ and $\widetilde \alpha'=
(\alpha, q^-(\alpha))$ for $\alpha\in H$, where $q^-:H\to \Gamma$ is
the homomorphism sending $\alpha_k$ to $y_k$ for  all $k$. In the
role of associated conjugating matrices (as in Section \ref{e22}) we
take $M'_\alpha=E_n$ for $\alpha\in G-H$ and $M'_\alpha=\rho
(q^-(\alpha))$ for $\alpha\in H$. It follows from the definitions
that the corresponding 2-cocycle $ \{\zeta'_{\alpha,\beta}\} $ is
trivial. The present  choice of conjugating matrices $M'_\alpha$
differs from the one  in the definition of  $\zeta^{\rho, \gamma}$
where it  is arbitrary on $G-H$ and is given on $H$ by the
homomorphism  $\mu:H\to GL_n(K)$  sending $\alpha_k$ to    $ (\Tr\,
\rho(y_k))^{-1} \, \rho (y_k)$  for  all~$k$. The homomorphisms
$\mu, \rho q^-:H\to GL_n(K)$ are projectively equal, so there is a
homomorphism  $\psi:H\to K^*$ such that
  $ \rho q^-  =\psi \,\mu $. We extend $\psi$ to $G$ by $\psi(G-H)=1$
and conclude that  $\zeta^{\rho, \gamma}$  is represented by the
  cocycle $\{\psi_{\alpha\beta}^{-1} \, \psi_\alpha \,
\psi_\beta \}_{\alpha, \beta\in G}$. Clearly, $\psi(\alpha_k)= \Tr\,
\rho(y_k)$ for all $k$. Therefore
    if  $Z\in H_2(G,H;\ZZ)$ and $\partial Z= \sum_{k} r_k
    [\alpha_k]
    $ with $ r_k\in \ZZ  $, then
$\zeta^{\rho,  \gamma} (Z)=\prod_{k}
    \,  (\Tr\, \rho(y_k))^{r_k}$. This implies (\ref{secondid}).
\end{proof}

  This lemma allows us to simplify  Formula (\ref{t3-++-})
   in the case where $G=\pi$, $g=\id:\pi\to \pi$, $G'=G\times \Gamma$, and $q:G'\to G$ is
the projection. This gives the following theorem. Let $W$ be a
compact connected oriented surface of positive genus $d$ with    $m\geq 0$
boundary components and base point $w$. Let $x_1,..., x_m\in
\pi_1(W,w) $ be as in Section \ref{twolemmas}. Then for any finite
group $\Gamma$ and any conjugacy classes $L_1,..., L_m$ in $\Gamma$,
the number of homomorphisms $  \pi_1(W,w)\to \Gamma $ sending
  $ x_k$ to $ L_k$ for all $k=1,..., m$ is equal to
$$ \vert  \Gamma \vert^{2d-1} \,
\sum_{\rho\in \Irr(\Gamma) } \left\{
(\dim \, \rho)^{2-2d-m} \,
  \prod_{k=1}^m \sum_{y\in L_k}\Tr \, \rho (y) \right\}\,
.$$ This theorem is due to     Frobenius \cite{Fr} for $W=S^1\times
S^1$ and to Mednykh \cite{Me} for all~$W$, see also \cite{Jo}.
Theorem \ref{t2}  for trivial bundles (over surfaces of positive genus) is  an
equivalent reformulation  of
 the Frobenius-Mednykh theorem.

                    \end{document}